\documentclass[12pt]{article}
\usepackage{authblk}
\usepackage{graphicx}
\usepackage{subfig}
\usepackage{multirow}
\usepackage{amsfonts}
\usepackage[a4paper,left=2.4cm,right=2.3cm,top=2cm,bottom=1.8cm]{geometry}%
\usepackage{amsmath,amssymb,mathrsfs}
\usepackage{amsthm}
\usepackage{textcomp}
\usepackage{color}

\usepackage{hyperref}
\usepackage{float}
\usepackage{epsfig}
\usepackage{epstopdf}
\usepackage{array,booktabs}
\usepackage{tabularx}
\usepackage{cleveref}

\usepackage{algorithm}
\usepackage{algorithmicx}
\usepackage{algpseudocode}

\usepackage{longtable}

\usepackage{comment}
\usepackage{bm}

\newtheorem{proposition}{Proposition}
\newtheorem{lemma}{Lemma}
\newtheorem{theorem}{Theorem}
\newtheorem{remark}{Remark}

\newtheorem{assumption}{Assumption}

\begin{document}

\title{Reinforcement-learning-based Algorithms for Optimization Problems and Applications to Inverse Problems}

\author[1]{Chen Xu}
\author[2]{Yun-bin Zhao}
\author[1]{Zhipeng Lu}
\author[3,4]{Ye Zhang\footnote{Corresponding author}}
\affil[1]{Guangdong Laboratory of Machine Perception and Intelligent Computing, Shenzhen MSU-BIT University, China. 
}
\affil[2]{Shenzhen Research Institute of Big Data, Chinese University of Hong Kong, Shenzhen, China}
\affil[3]{Faculty of Computational Mathematics and Cybernetics, Shenzhen MSU-BIT University, 518172 Shenzhen, China}
\affil[4]{School of Mathematics and Statistics, Beijing Institute of Technology, 100081 Beijing, China
\authorcr Email: \{xuchen, zhipeng.lu, ye.zhang\}@smbu.edu.cn, yunbinzhao@cuhk.edu.cn}

\date{}

\maketitle

\begin{abstract}
We design a new iterative algorithm, called REINFORCE-OPT, for solving a general type of optimization problems. This algorithm parameterizes the solution search rule and iteratively updates the parameter using a reinforcement learning (RL) algorithm resembling REINFORCE. To gain a deeper understanding of the RL-based methods, we show that REINFORCE-OPT essentially solves a stochastic version of the given optimization problem, and that under standard assumptions, the searching rule parameter almost surely converges to a locally optimal value. Experiments show that REINFORCE-OPT outperforms other optimization methods such as gradient descent, the genetic algorithm, and particle swarm optimization, via its ability to escape from locally optimal solutions and its robustness to the choice of initial values. With rigorous derivations, we formally introduce the use of reinforcement learning to deal with inverse problems. By choosing specific probability models for the action-selection rule, we can also connect our approach to the conventional methods of Tikhonov regularization and iterative regularization. We take non-linear integral equations and parameter-identification problems in partial differential equations as examples to show how reinforcement learning can be applied in solving non-linear inverse problems. The numerical experiments highlight the strong performance of REINFORCE-OPT, as well as its ability to quantify uncertainty in error estimates and identify multiple solutions for ill-posed inverse problems that lack solution stability and uniqueness.
\end{abstract}
\maketitle

\section{Introduction}
\label{sec:Introduction}

In this paper, we study the theory of using a reinforcement-learning (RL) algorithm to solve the continuous optimization problem:
\begin{equation}
\label{opt}
\max_{\bm{x}\in \mathcal{H}_1} \mathcal{L}(\bm{x}),
\end{equation}
where $\mathcal{H}_1$ is a Hilbert space, and $\mathcal{L}$ is a functional acting from $\mathcal{H}_1$ to $\mathbb{R}$. To facilitate the use of RL, we reformulate the problem (\ref{opt}) as a solution-search problem within the framework of a Markov decision process (MDP) and solve it using an RL method similar to the popular REINFORCE algorithm (e.g. \cite[Eq. (13.8)]{SuttonBarto2018}). Specifically, our algorithm, named REINFORCE-OPT, parameterizes the solution search rule as $\pi_{\bm{\theta}}$ and iteratively updates the parameter $\bm{\theta}\in \mathbb{R}^d$ using the gradient ascent algorithm based on the $\bm{x}$ values obtained by Monte Carlo sampling. In this context, a search rule refers to a non-deterministic function that selects the next points to examine within $\mathcal{H}_1$.

The application of RL to optimization problems has garnered significant attention recently. For example, the search for the optimal $\pi_{\bm{\theta}}$ can be viewed as a search within the parameterized space of optimization algorithms. Based on this concept, Li and Malik  \cite{LiMalik2017} proposed a meta-learning method for continuous optimization algorithm design. Specifically, using a set $\mathcal{F}$  of randomly generated objective functions, they train the search rule $\pi_{\bm{\theta}}$ through guided policy search \cite{Levine2013}. Unlike REINFORCE-OPT, their search rule does not require retraining when dealing with a new objective function generated from the same random source as $\mathcal{F}$. However, their method relies on the gradients of the objective function as inputs to $\pi_{\bm{\theta}}$, limiting its application to nondifferentiable objectives. There are also several works that apply RL to solve combinatorial optimization problems, such as the traveling salesman problem and routing problems \cite{Bello2017, Nazari2018, Kool2019, Mazyavkina2021}, as well as works that combine RL with popular evolutionary algorithms for global optimization \cite{Seyyedabbasi2021, Seyyedabbasi2023, Drugan2019}. A good review for the recent development in this direction can be found in \cite{CLY2024}. While the basic idea behind these RL-based algorithms is similar to ours, a deeper theoretical understanding of these algorithms is often lacking. For example, the precise relationship between RL's goal of maximizing long-term returns and the original optimization problem (\ref{opt}) is not rigorously explored, nor is there a comprehensive theory on the convergence of algorithms.

Our first task is to address both of these issues. In Proposition \ref{interpret}, we prove that as the length of the search path used for each update of $\bm{\theta}$ approaches infinity, REINFORCE-OPT solves a stochastic version of the original problem (\ref{opt}) with a regularization term, i.e. 
\begin{equation}
\label{rlgoal-opt}
\max_{\bm{\theta}\in \mathbb{R}^d} \{ \mathbb{E}_{\bm{x}\sim \mu^{\pi_{\bm{\theta}}}} \left[ \mathcal{L}(\bm{x}) - \beta \lVert \bm{\theta} \rVert^2_w \right] \},
\end{equation}
where $\mu^{\pi_{\bm{\theta}}}$ is a distribution over the solution space $\mathcal{H}_1$ to be explained in Proposition \ref{interpret}, $\beta>0$ is the regularization coefficient, and the weighted Euclidean norm $\lVert \cdot \rVert_w$ is explained later in Sections \ref{MDP-setting}. Additionally, we provide an asymptotic convergence analysis of REINFORCE-OPT in this work.

In the field of optimization, REINFORCE-OPT belongs to the category of meta-heuristic methods, characterized by an iterative process of ``evaluating the $\mathcal{L}$-values for the current batch of $\bm{x}$ samples and generating the next batch based on these evaluations and the search rule''. Many popular methods fall into this category, including evolutionary algorithms like the genetic algorithm (GA) \cite[Section 3.1.9]{locatelli2013} and particle swarm optimization (PSO)  \cite[Section 3.1.5]{locatelli2013}, Bayesian optimization (BO) \cite{Frazier2018}, the estimation of distribution algorithm (EDA) \cite{HauschildPelikan2011}, and the cross-entropy (CE) method \cite{DEBOER2005}. The main difference between these methods and REINFORCE-OPT lies in the latter's parameterization of the search rule using a flexible function $\pi_{\bm{\theta}}$ (e.g., a neural network) and iterative updates to maximize the expected  $\mathcal{L}$-value of the $\bm{x}$ values sampled from $\pi_{\bm{\theta}}$. In contrast, GA and PSO (except for some variants) do not update key parameters in their search rules, such as the crossover and mutation rates in GA or the inertia weight and cognitive coefficient in PSO. BO parameterizes $\mathcal{L}$ rather than the search rule, as it is designed to handle situations where the functional form of $\mathcal{L}$ is unknown or expensive to evaluate. However, BO is impractical for handling large numbers of sampled $\bm{x}$ values due to its computational complexity, which requires at least $N^2$ evaluations of $\mathcal{L}$ \cite[Eq. (3)]{Frazier2018}, where $N$ denotes the number of sampled $\bm{x}$-values for one update of the search rule. Additionally, it is recommended that BO be applied to cases where the dimension of $\bm{x}$ does not exceed 20 \cite{Frazier2018}. EDA and CE update their search rules differently from REINFORCE-OPT. EDA constructs a probabilistic model to fit the distribution of the set $\mathcal{X}$ of $\bm{x}$-samples with larger $\mathcal{L}$-values and then re-samples based on this model. CE \cite[Algorithm 2.1]{DEBOER2005} updates the search rule parameters to increase the probabilities of samples in  $\mathcal{X}$.

Our algorithm offers several advantages over other popular methods:
\begin{enumerate}
\item \emph{No requirement for differentiability}: Unlike the RL-based algorithm in \cite{LiMalik2017} and local optimization methods (e.g., gradient ascent in $\mathcal{H}_1$), REINFORCE-OPT does not require the objective $\mathcal{L}$ to be differentiable\footnote{Note that REINFORCE-OPT performs gradient descent in the parameter space, not in the $\bm{x}$ space. Therefore, it does not rely on the derivative of $\mathcal{L}$. Specifically, the algorithm seeks to solve $\max_{\bm{\theta}} J(\bm{\theta})$ for some differentiable $J(\bm{\theta})$.}. As shown by the iteration formula in (\ref{REINFORCE}), the search path $\{\bm{x}_t\}$ produced by the updating rule $\pi_{\bm{\theta}^*}$  may escape from local optima\footnote{While our algorithm with ability to escape local optima has an advantage over local optimization methods, it may not have  an advantage over the RL-based algorithm in \cite{LiMalik2017}.} of $\mathcal{L}(\bm{x})$ (see Figures \ref{escape-local} and \ref{escape-local2D} for numerical demonstration).

\item \emph{Robustness against initial conditions}: Compared to some of the popular global optimization methods, our experiments in Section \ref{Robustness} demonstrate that REINFORCE-OPT is more robust to the choice of initial values $\bm{x}_0$ than GA, PSO and CE. Our results in Tables \ref{compare-results} and \ref{robustk6} suggest that it is difficult for these three algorithms to locate the global optimum when it lies outside the region of initial population.

\item \emph{Handling higher dimensions}: As shown in Section \ref{Ex1}, REINFORCE-OPT is capable of handling an $\bm{x}$-dimension of 64, which significantly exceeds the upper bound of 20 suggested for BO \cite{Frazier2018}.

\item \emph{Theoretical convergence}: While theoretical results on the convergence of GA for nonconcave functional forms of $\mathcal{L}$ in infinite solution spaces are lacking, this work provides an asymptotic convergence analysis for REINFORCE-OPT.

\item \emph{Limitations of GA theory}: Existing theoretical works on GA assume either a finite solution space \cite{Rudolph1994,BhandariMurthy1996,GreenhalghMarshall2000}, a concave objective $\mathcal{L}$  \cite{State1995}, or focus on general evolutionary algorithms rather than GA \cite{Rudolph1996, Rudolph2013, ChenHe2021}. We believe that the advantage of REINFORCE-OPT over GA and PSO in locating the global maximum for problem (\ref{opt}) stems from the reliable updating mechanism of gradient ascent in the $\bm{\theta}$-space.
\end{enumerate}

The second task in this work is to apply REINFORCE-OPT to solve a stochastic formulation of the following general inverse problem:
\begin{itemize} \item[ ] (\textbf{IP}) \emph{ Given an observed value of the random variable $\bm{y}^\delta$, estimate the solution $\textbf{x}_e$ of the following model
\begin{equation}
\label{inversep}
f(\textbf{x}_e) + \delta \xi = \bm{y}^\delta, 
\end{equation}
where $f: \mathcal{H}_1 \supset \mathcal{D}(f) \rightarrow \mathcal{H}_2$ is a known map between two Hilbert spaces, $\delta>0$ denotes the noise level, and $\xi$ represents an unobservable Hilbert space process\footnote{A Hilbert space process $\xi$ is a bounded linear mapping $\xi: \mathcal{H}_2 \to L^2(\Omega,\mathcal{A},P)$, where $(\Omega,\mathcal{A},P)$ is a probability space.} on $\mathcal{H}_2$.  }
\end{itemize} 
Rigorously, problem (\ref{inversep}) can be understood as follows: for each $z\in \mathcal{H}_2$, we have access to the value of $\langle\bm{y}^\delta, z \rangle$ that satisfies
\begin{equation*}
\langle f(\textbf{x}_e) , z \rangle + \delta \langle\xi, z \rangle = \langle\bm{y}^\delta, z \rangle,
\end{equation*}
where for simplicity we use $\langle \cdot, \cdot \rangle$ to denote the inner products for both Hilbert spaces. The noise term $\xi_z := \langle \xi, z\rangle$, i.e., $\xi$ can be interpreted as a random element in the algebraic dual space of $\mathcal{H}_2$. Model (\ref{inversep}) includes the standardized Gaussian white-noise model where $\xi_z \sim N(0, \|z\|^2_{\mathcal{H}_2})$ and $\text{Cov} [\xi_{z_1}, \xi_{z_2}] = \langle z_1, z_2 \rangle$ for all $z_1, z_2 \in \mathcal{H}_2$, where $\|\cdot\|$ denotes the norm in the Hilbert space.

For ill-posed problems\footnote{We refer to \cite[Def. 1.1]{Hofmann1998} and \cite[Def. 3]{Hofmann2018} for a rigorous definition of the (local) ill-posedness of a nonlinear operator equation.},  some regularization approaches are 
required to find stable approximations to the solution $\textbf{x}_e$ of inverse problem (\ref{inversep}). Specific regularization methods for similar inverse problems can be found in \cite{HarrachJahnPotthast2020,HarrachJahnPotthast2023,Jin2023IP}. Within the framework of Tikhonov regularization, the approximate solution of (\ref{inversep}) can be taken as the solution of the minimization problem
\begin{equation}
\label{Tikhonov}
\min_{\bm{x}\in \mathcal{H}_1} \|f(\bm{x}) - \bm{y}^\delta\|^2 + \alpha \Omega(\bm{x}),
\end{equation}
where $\Omega(\bm{x})$ is the regularization
term and $\alpha > 0$ denotes the regularization
parameter. However, the methods commonly used to solve (\ref{Tikhonov}) have two drawbacks: 
\begin{enumerate}
\item[(a)] The solution estimator $\bm{x}^\delta_\alpha$ is deterministic and therefore does not account for the randomness introduced by the noise $\delta \xi$ in the problem (\ref{inversep}).

\item[(b)] Difficulties arise when the objective in (\ref{Tikhonov}) is nonconvex and/or nondifferentiable; both are common in practical applications. 
\end{enumerate}

To address the two issues, we apply REINFORCE-OPT to solve (\ref{Tikhonov}), which by Proposition \ref{interpret} is equivalent to solving the following statistical formulation of (\ref{Tikhonov}) with double regularization penalties: 
\begin{equation}
\label{statisticalFormulation}
\bm{\theta}^* = \arg\max_{\bm{\theta}}  \mathbb{E}_{\bm{x}\sim \mu^{\pi_{\bm{\theta}}}} [ -\| f(\bm{x})-\bm{y}^\delta \|^2  - \alpha \Omega(\bm{x}) - \beta \lVert \bm{\theta} \rVert^2_w ],
\end{equation}
where the parameterized distribution $\mu^{\pi_{\bm{\theta}}}$ is explained in Proposition \ref{interpret}. In this new formulation, the designed regularization solution $\bm{x}$ is a random variable following the distribution $\mu^{\pi_{\bm{\theta}}}$ with parameters $\bm{\theta} = \bm{\theta}^*$. Also, our experiments in Section \ref{experiment-opt} demonstrate the ability of REINFORCE-OPT to handle nonconvex cases. In addition to addressing the above-mentioned issues (a) and (b), another advantage of REINFORCE-OPT over existing methods in the field of inverse problems is its ability to handle the case of multiple solutions of (\ref{inversep}). Commonly used supervised learning frameworks, such as neural networks, cannot be directly applied to model the multi-valued inversion map $f^{-1}$ (see, e.g., \cite{XuZhang2022,YangXuZhang2022}). This advantage will be demonstrated in  Section \ref{Simu3}.

This paper is organized as follows: Section \ref{MDP-setting} reformulates the optimization problem in (\ref{opt}) within the framework of an MDP and  explores the relationship between the goals of the reinforcement learning method and the optimization problem. In Section \ref{Main}, we present the REINFORCE-OPT algorithm and show the convergence theorem. Section \ref{RLandRegu} introduces the application of REINFORCE-OPT to inverse problems and establishes connections between this algorithm and classical regularization methods. Section \ref{experiment-opt} demonstrates the advantages of REINFORCE-OPT over other commonly used optimization methods through experimental comparisons, while Section \ref{simulation} showcases the performance of REINFORCE-OPT with numerical examples involving auto-convolution equations and inverse chromatography problems. Finally, Section \ref{Conclusion} provides concluding remarks, and additional technical proofs are included in the Appendices. All codes are available at \url{http://github.com/chen-research/REINFORCE-OPT}, with the reinforcement learning-related codes implemented using TensorFlow Agent \cite{TFAgents}.

\section{Reinforcement Learning for Optimization}
\label{MDP-setting}

The following notations will be used throughout the paper. 
\begin{longtable}{p{2cm} p{8.5cm} p{3cm}}
\midrule
\textbf{Notation}   &\textbf{Description}   &\textbf{Reference}  \\
\toprule 
$\mathbb{Z}^+$ & The set of positive integers. & \\




$\bm{y}_k^{\delta}$ & The $k^{th}$ observed sample of $\bm{y}^\delta$. & (\ref{inversep}).\\

$\Omega(\bm{x})$ & The regularization term. & (\ref{Tikhonov}).\\

$\alpha,\beta$ & The non-negative regularization coefficients. & (\ref{Tikhonov}) and (\ref{statisticalFormulation}).\\

$\pi_{\bm{\theta}}(\cdot|\bm{x})$ & A probability density of the action $\bm{a}$ given the current state $\bm{x}$. & Near (\ref{rlgoal}).\\

$\mu^{\pi_{\bm{\theta}}}$ & The invariant distribution of the Markov chain $\{\bm{x}_t\}_{t=0}^{+\infty}$ generated by following policy $\pi_{\bm{\theta}}$. &Theorem \ref{interpret}.\\

$\bm{x}_t$ & The state in step $t$ in a trajectory. & (\ref{pmeasure}).\\

$\bm{a}_t$ & The action taken in step $t$ in a trajectory. & (\ref{pmeasure}).\\

$T$ & The trajectory length $\{\bm{x}_{t}\}_{t=0}^{T-1}$. & (\ref{pmeasure}).\\

$J_T(\bm{\theta})$ & The expected sum of rewards obtained in a trajectory of length $T$ generated following the policy $\pi_{\bm{\theta}}$. & (\ref{pmeasure}).\\

$\mathbb{E}_{\bm{x}_t,\bm{a}_t}[\cdot | \pi_{\bm{\theta}}]$ & The expectation over $(\bm{x}_t,\bm{a}_t)$ given that they are the step $t$ state-action pair in a trajectory generated by $\pi_{\bm{\theta}}$. This expectation is also determined by the initial state $\bm{x}_0$. However, we consider the initial state's distribution fixed. & (\ref{pmeasure}). \\

$\overline{U(\bm{0};B_1)}$ & $\{ \bm{\theta} \mid \bm{\theta}\in\mathbb{R}^d, \| \bm{\theta}\|\leq B_1\}$. & Assumption \ref{Assump1}. \\

$\nabla$ & It denotes a partial derivative with respect to $\bm{\theta}$. & \\

$\nabla J_T(\bm{\theta})$ & $\frac{\partial J_T({\bm{\theta}})}{\partial\bm{\theta}}$. &  \\

$\hat{\nabla} J_T(\bm{\theta})$ & Sample estimate of $\nabla J_T{\bm{\theta}}$. & (\ref{REINFORCE}). \\

$g(\cdot)$ & A Lipschitz map from $\mathbb{R}^d$ to $\mathbb{R}^d$, with the Lipschitz coefficient of $L_g$. & Assumption \ref{Assump0}. \\

$M_{n+1}$ & $\hat{\nabla} J_T(\bm{\theta}_n)-\nabla J_T(\bm{\theta}_n)$. & (\ref{theta-update}). \\

$a_n$ & The updating step size for $\{\bm{\theta}_n\}$. & (\ref{REINFORCE}).\\

$s_n$ & $s_n:=\sum_{i=0}^{n-1} a_i$ for any $n>0$ and $s_0:=0$. & Lemma \ref{Lemma2.1}.\\

$[t]^-$ & $[t]^-:=\text{max}_{n\geq 0}\{ s_n | s_n\leq t \}$. & (\ref{bracket-}). \\

$[t]^+$ & $[t]^+:=\text{min}_{n\geq 0}\{ s_n | s_n\geq t \}$. &(\ref{bracket-}). \\

$\{\bm{\theta}_n\}_{n=0}^{+\infty}$ & The iterate sequence. & (\ref{theta-update}) or (\ref{REINFORCE}). \\

$\bar{\bm{\theta}}(\cdot)$ & $\bar{\bm{\theta}}(t):=\lambda_t \bm{\theta}_n + (1-\lambda_t)\bm{\theta}_{n+1}$ for $t\in [s_n,s_{n+1}]$, where $\lambda_t := \frac{s_{n+1}-t}{s_{n+1}-s_n}\in [0,1]$. & Lemma \ref{Lemma2.1}.\\

$\bm{\theta}^s(\cdot)$ & A solution to ODE (\ref{ODE-f}) with $\bm{\theta}^s(s)=\bar{\bm{\theta}}(s)$. & Lemma
\ref{Lemma2.1}.\\

$U^+$& $\{\bm{x}\in \mathbb{R}^D: \bm{x} \text{ has non-negative entries} \}$. & Below (\ref{xe}).\\
\bottomrule
\end{longtable}

Let us first reformulate problem (\ref{opt}) within an MDP framework\footnote{An MDP involves an agent that takes a series of actions in a dynamic environment, receiving rewards based on state-action pairs. Each action depends solely on the current state, independent of past states and actions, which is known as the Markov property. For more details on MDPs, see \cite[Section 3.1]{SuttonBarto2018}.}, allowing us to apply RL algorithms iteratively to solve it. Suppose that  there is an agent moving through $\mathcal{H}_1$, where its location represents the state. At each time step $t$, the agent takes a step (action $\bm{a}_t \in \mathcal{H}_1$) from the current state $\bm{x}_t$, transitioning to a new state $\bm{x}_{t+1}=\bm{x}_{t}+\bm{a}_{t}$ and receiving a reward  $R(\bm{x}_t,\bm{a}_t)$. In this framework, the approximate solution to problem (\ref{opt}) is the limit of the sequence $\{ \bm{x}_{t} \}$.

RL algorithms aim to find the $\mathbf{a}_t$-selection rule $\pi$ (time-independent), which is a map from the $\bm{x}$-space $\mathcal{S}$ to the set of probability distributions $PD(\mathcal{A})$ over the $\bm{a}$-space $\mathcal{A}$, that maximizes the expected average rewards. Specifically, after parameterizing the policy with $\pi_{\bm{\theta}}$, they aim to solve
\begin{equation}
\label{rlgoal}
\max_{\bm{\theta}\in \mathbb{R}^d} \{ J_T(\bm{\theta}) -\beta \lVert \bm{\theta} \rVert^2_w \},
\end{equation}
where $\beta>0$ is the regularization coefficient, $\lVert \cdot \rVert_w$ denotes the weighted Euclidean norm, i.e., $\lVert \bm{\theta} \rVert^2_w = \sum^d_{i=1} w_i \theta^2_i$ with positive numbers $w_i$, and\footnote{See \cite[Eq.(13.15)]{SuttonBarto2018} for the definition of a long-term performance measure. Another common way of defining the performance measure is the expected sum of discounted rewards, i.e., $J(\pi) := \mathbb{E}_{\bm{a}_t \sim \pi(\bm{x}_t)} \left[ \lim_{T\rightarrow \infty}\sum_{t=0}^T \gamma^{t} R(\bm{x}_t,\bm{a}_t)\right]$, where $\gamma\in [0,1)$ is the discount factor.} 
\begin{equation}
\label{pmeasure}
J_T(\bm{\theta}):=  \frac{1}{T} \sum_{t=0}^{T-1}\mathbb{E}_{\bm{x}_t,\bm{a}_t } \left[ R(\bm{x}_{t},\bm{a}_t) | \pi_{\bm{\theta}} \right],
\end{equation}
where
\begin{itemize}
\item $\pi_{\theta}(\cdot|\bm{x})$ is a distribution over the action space determined by the current state $\bm{x}$.
\item $\bm{x}_t \in \mathcal{S}$ denotes the $t^{th}$ step's state and $\mathcal{S}\subset \mathcal{H}_1$ is called the state space.
\item $\bm{a}_t \in \mathcal{A}$ denotes the agent's $t^{th}$ action and $\mathcal{A}\subset \mathcal{H}_1$ is called the action space. The distribution of $\bm{a}_t$ is $\pi_{\bm{\theta}}(\cdot|\bm{x}_t)$.
\item $\mathbb{E}_{\bm{x}_t,\bm{a}_t } \left[ \cdot | \pi_{\bm{\theta}} \right]$ denotes an expectation with respect to $(\bm{x}_t,\bm{a}_t)$, where $(\bm{x}_t,\bm{a}_t)$ is the $t^{th}$ step state-action pair of a trajectory generated following $\pi_{\bm{\theta}}$. Note that the distribution of the initial state $\bm{x}_0$ is independent from $\pi_{\bm{\theta}}$. It is determined by the problem environment or set by the user.
\item The transition rule from a state-action pair to the next state is
\begin{equation}
\label{iteration}
\bm{x}_{t+1} = \bm{x}_{t} + \bm{a}_{t}.
\end{equation}
We note that the transition rule can be set differently for a general problem. Nevertheless, we adopt (\ref{iteration})  throughout this paper since it is proper for the inverse problem (\ref{inversep}).
\item The reward function $R$ should be defined so that solving (\ref{rlgoal}) is consistent with finding a $\bm{a}_t$-selection rule $\pi$ ($\Delta\bm{x}_{t}$-selection rule) such that the agent will reach the minimum of $\mathcal{L}$ as soon as possible. There are multiple options for the functional form of $R$, and we choose the following simplest one:
\begin{equation}
\label{neg-R}
R(\bm{x},\bm{a}):=\mathcal{L}(\bm{x}+\bm{a}).
\end{equation}

\item $T\in\mathbb{Z}^+$ is the pre-selected trajectory length.
\end{itemize}

Therefore, RL can be viewed as a method for finding an $\bm{x}_t$-iteration rule to solve (\ref{opt}). Throughout the paper, a trajectory generated by the policy $\pi_{\bm{\theta}}$ refers to $\{(\bm{x}_t,\bm{a}_t) \}_{t=0}^{T-1}$, where $\bm{x}_0$ is sampled from a distribution that is potentially independent from $\pi_{\bm{\theta}}$, $\bm{a}_t \sim \pi_{\bm{\theta}}(\cdot|\bm{x}_t)$, and $\bm{x}_{t+1}=\bm{x}_t+\bm{a}_t$. We also call $\{\bm{x}_t\}_{t=0}^{T-1}$ a trajectory generated by $\pi_{\bm{\theta}}$. It is clear that $\{\bm{x}_t\}$ is a Markov chain.

To better understand the relationship between the RL's objective (\ref{rlgoal}) and the optimization problem (\ref{opt}), we present the following proposition, which demonstrates that the RL algorithm seeks to solve a stochastic version of (\ref{opt}). \\

\begin{proposition}
\label{interpret}
For any $\bm{\theta}\in\mathbb{R}^d$, under the conditions that the Markov chain $\{\bm{x}_t\}$ generated following the policy $\pi_{\bm{\theta}}$ has a unique invariant probability measure\footnote{Conditions that ensure the existence of a invariant probability measure of $\{\bm{x}_t\}$ can be found in Appendix C. To adapt to the conditions, we can rewrite $\bm{x}_{t+1}=\bm{x}_t+\bm{a}_t=\bm{x}_t+\bm{m}_{\bm{\theta}}(\bm{x}_t)+\bm{z}_{t+1} = f(\bm{x}_t,\bm{z}_{t+1})$, where $\bm{m}_{\bm{\theta}}(\bm{x}_t)$ denotes the mean of $\pi_{\bm{\theta}}(\cdot|\bm{x}_t)$, $\{\bm{z}_t\}$ is a sequence of i.i.d. random variables, and $\pi_{\theta}(\cdot)$ is a deterministic map.} $\mu^{\pi_{\bm{\theta}}}$, and that $\mathcal{L}(\cdot)$ is integrable with respect to the distribution $\mu_t$ of $\bm{x}_t$ for each positive integer $t$.
Then, as $T\rightarrow+\infty$, the RL's goal (\ref{rlgoal}) becomes
\begin{equation}
\label{rlgoal-infty}
\max_{\bm{\theta}} \{ \mathbb{E}_{\bm{x}\sim \mu^{\pi_{\bm{\theta}}}} \left[ \mathcal{L}(\bm{x}) \right]  - \beta \lVert \bm{\theta} \rVert^2_w \}.
\end{equation}
\end{proposition}

\begin{proof} By the property of a Markov chain's invariant probability, the probability measure $\mu_t$ of $\bm{x}_t$ converges weakly to $\mu^{\pi_{\bm{\theta}}}$ as $t\to +\infty$, which further implies that $\mu'_T:=\frac{1}{T}\sum_{t=1}^{T}\mu_t$ converges weakly to $\mu^{\pi_{\bm{\theta}}}$ as $T\to+\infty$. This result and the Portemanteau theorem (e.g., \cite[Theorem 13.16]{Klenke2013}) imply the second-to-last equality in the following derivation.
\begin{equation}
\begin{split}
\lim_{T\rightarrow +\infty} J_T(\bm{\theta})
=& \lim_{T\rightarrow +\infty} \frac{1}{T} \sum_{t=0}^{T-1}\mathbb{E}_{\bm{x}_t,\bm{a}_t } \left[ R(\bm{x}_{t},\bm{a}_t) | \pi_{\bm{\theta}} \right],\\
=& \lim_{T\rightarrow +\infty} \frac{1}{T} \sum_{t=0}^{T-1}\mathbb{E}_{\bm{x}_t,\bm{a}_t } \left[ \mathcal{L}(\bm{x}_{t+1})  | \pi_{\bm{\theta}} \right], \quad \text{by } (\ref{neg-R}), \\
=&\lim_{T\rightarrow +\infty}\frac{1}{T} \sum_{t=0}^{T-1}  \int_{\bm{x}\in\mathcal{S}} \mathcal{L}(\bm{x}) d\mu_{t+1} \\
=&\lim_{T\rightarrow +\infty}   \int_{\bm{x}\in\mathcal{S}} \mathcal{L}(\bm{x}) d\mu'_{T} \\
=&\mathbb{E}_{\bm{x}\sim \mu^{\pi_{\bm{\theta}}}}[ \mathcal{L}(\bm{x}) ], 
\label{pointwise-conv}
\end{split}
\end{equation}
which yields the required results by noting the stationary regularization term $\beta \lVert \bm{\theta} \rVert^2_w$. 
\end{proof}

\section{REINFORCE-OPT and Convergence Analysis}
\label{Main}

Before presenting the REINFORCE-OPT algorithm and the main theorem, we first revisit the policy gradient theorem.

\subsection{Gradient Estimate and Parameter Update Rule}

The well-known \emph{policy gradient theorem} provides the expression for the gradient $\nabla J_T(\bm{\theta})$. It has several versions; here, we derive the one specific to our problem, which can also be found in \cite{zhao2012}. Throughout this paper, we assume that the reward function is bounded. Starting from the initial $\bm{x}_0$ with a given distribution, let $h_T:=\{(\bm{x}_t,\bm{a}_t)\}_{t=0}^{T-1}$ denote the trajectory generated following $\pi_{\bm{\theta}}$, and let $\mathcal{H}_T$ denote the sample space of $h_T$. Also, define
$$  R(h_T):=\frac{1}{T} \sum_{t=0}^{T-1}R(\bm{x}_t,\bm{a}_t). $$
Then, from the definition of $J_T(\bm{\theta})$ in (\ref{pmeasure}), we derive that 
\begin{align*}
\nabla J_T(\bm{\theta}) &= \nabla \mathbb{E}_{h_T\sim \pi_{\bm{\theta}}} [R(h_T)] = \nabla \int_{h\in \mathcal{H}_T} p_{\bm{\theta}}(h_T) R(h_T) d h_T =\int_{h_T\in \mathcal{H}_T}  \nabla p_{\bm{\theta}}(h_T) R(h_T) d h_T \\
&=\int_{h_T\in \mathcal{H}_T}   p_{\bm{\theta}}(h_T) \nabla\ln p_{\bm{\theta}}(h_T) R(h_T) dh_T =\mathbb{E}_{h_T\sim \pi_{\bm{\theta}}} \left[\nabla\ln p_{\bm{\theta}}(h_T) R(h_T)\right]\\
&=\mathbb{E}_{h_T\sim \pi_{\bm{\theta}}} \left[R(h_T) \nabla \sum_{t=0}^{T-1}\ln \pi_{\bm{\theta}}(\bm{a}_t|\bm{x}_t) \right] =\mathbb{E}_{h_T\sim \pi_{\bm{\theta}}} \left[R(h_T)  \sum_{t=0}^{T-1}\nabla\ln \pi_{\bm{\theta}}(\bm{a}_t|\bm{x}_t) \right],
\end{align*}
where $p_{\bm{\theta}}(h_T)$ denotes the probability density of trajectory $h_T$ following policy $\pi_{\bm{\theta}}$, and the derivative and integration is interchangeable (see the third equality above) due to the Lebesgue's dominated convergence theorem (DCT) if\footnote{To adapt to DCT, view $\nabla p_{\bm{\theta}}$ as a sequence limit, i.e., $\nabla p_{\bm{\theta}}=\lim_{k\rightarrow +\infty} \frac{p_{\bm{\theta}_k}-p_{\bm{\theta}}}{\bm{\theta}_k-\bm{\theta}}$, where $\lim_{k\rightarrow +\infty}\bm{\theta}_k =\bm{\theta}$.} $\nabla p_{\bm{\theta}}(h_T)$ exists for all $h_T\in \mathcal{H}_T$. Based on the above expression, a sample estimate of $\nabla J_T(\bm{\theta})$ is straightforward, namely  
\begin{align}
\label{gradient-est}
\hat{\nabla} J_T(\bm{\theta}) := \frac{1}{L} \sum_{l=0}^{L-1} \left[R(h^l_{T})  \sum_{t=0}^{T-1}\nabla\ln \pi_{\bm{\theta}}(\bm{a}^l_t|\bm{x}^l_t) \right],
\end{align}
where $h^l_{T}$ denotes the $l^{th}$ sample of $h_{T}$.


We follow a stochastic gradient ascent rule to iteratively improve the objective function  $J_T(\bm{\theta})-\beta\|\bm{\theta}\|^2_w$ in (\ref{rlgoal}). Specifically, given the current value $\bm{\theta}_n$, we generate $L$ trajectories\footnote{Note that the initial state $\bm{x}_0$ in each trajectory is drawn from a user-specified distribution that is irrelevant to $\pi_{\bm{\theta}_n}$ and remains fixed as $n$ increases. One common choice is to set $\bm{x}_0$ as a fixed point.} following $\pi_{\bm{\theta}_n}$ and update $\bm{\theta}_n$ according to
\begin{align}
\label{REINFORCE}
\bm{\theta}_{n+1} = \bm{\theta}_n + a_n (\hat{\nabla} J_T(\bm{\theta}_n) - 2\beta \bm{\theta}_n), \quad n\in \mathbb{Z}^+\cup \{0\},
\end{align}
where $\bm{\theta}_0\in\mathbb{R}^d$ is randomly initialized, $\hat{\nabla} J_T(\bm{\theta})$ is computed according to (\ref{gradient-est}) using the generated $L$ trajectories, $a_n>0$ is the update step size, and the last term is for regularization purposes. Note that $L$ and $T$ are pre-selected by the user and do not change with $n$.

\subsection{REINFORCE-OPT}
We now describe the new algorithm, named REINFORCE-OPT, for solving the potentially nonconvex optimization problem (\ref{opt}), as shown in Algorithms \ref{alg:Framwork} and \ref{alg:Framwork1}.

\begin{algorithm}[htb]
\caption{Algorithm for training the action rule $\pi_{\bm{\theta}}$.}
\label{alg:Framwork}
\begin{algorithmic}[1]
\State{Design the structure of $\pi_{\bm{\theta}}$.}
\State{Construct the reward function $R$.}
\State{Define the distribution of the initial state $\bm{x}_0$.}
\State{Set values of the hyper-parameters $\{\mathcal{N}_{\bm{\theta}}, T, \text{H}_0, N, L, \beta\}$.}
\State{Initialize the weights in $\mathcal{N}_{\bm{\theta}}$ as $\bm{\theta}_0$, and $n=0$.}
\While{$n<N$}
    \State Follow $\pi_{\bm{\theta}_n}$ to generate $L$ length-$T$ trajectories $\{h_l^T\}_{l=0}^{L-1}$.
    \State{Compute $\bar{\bm{x}}_{t} = \frac{1}{L} \sum_{l=0}^{L-1}\bm{x}_{l,t}$ for each $t\in \{0,1,2,...,T-1\}$, where $\bm{x}_{l,t}$ is the step-$t$ state in trajectory $h_l^T$.}
    \State{Compute $\text{r}=\sum_{t=0}^{T-1} R(\bar{\bm{x}}_{t},0)$.}
    \If{$\text{r}>\text{H}_0$}
    \State break;
    \EndIf
    \State {Compute $\hat{\nabla}J_T(\bm{\theta})$ according to (\ref{gradient-est}).
    }
    \State Follow {(\ref{REINFORCE}) to update $\bm{\theta}_n$ to get $\bm{\theta}_{n+1}$. 
    }
\EndWhile
\end{algorithmic}
\end{algorithm}

\begin{algorithm}[htb]
\caption{REINFORCE-OPT for solving the optimization problem (\ref{opt}).}
\label{alg:Framwork1}
\begin{algorithmic}[1]
\State{Input the objective function $\mathcal{L}(\cdot)$.}
\State{Train the action rule $\pi_{\bm{\theta}}$ according to Algorithm \ref{alg:Framwork}.}
\State{Produce samples of approximate solution $\{\bm{x}\}$.}
\State{Construct the statistics of the approximate solution, i.e., the expectation and confidence intervals.}
\end{algorithmic}
\end{algorithm}

A brief explanation of each step of Algorithm \ref{alg:Framwork} is provided as follows:
\begin{itemize}
\item The structure of the distribution $\pi_{\bm{\theta}}$ of actions can be designed as a multivariate normal distribution, with its mean and covariance equal to the output of a neural network $\mathcal{N}_{\bm{\theta}}$ whose input is a state $\bm{x}$. $\pi_{\bm{\theta}}$ is designed in this way except the one in Section \ref{escape-from-local}.

\item The functional form of the reward function $R(\cdot,\cdot)$ can be constructed by (\ref{neg-R}). 

\item The initial state $\bm{x}_0$ is used for generating $L$ trajectories $\{h_T^l\}_{l=0}^{L-1}$ in Eq. (\ref{gradient-est}).

\item The meanings of hyper-parameters are: the architecture of neural network $\mathcal{N}_{\bm{\theta}}$, i.e., the number of hidden layers and nodes; the trajectory length $T$; the threshold $\text{H}_0$ for the stop-training rule; the prescribed total number of $\bm{\theta}$-updates $N$; the number $L$ of trajectories used for one update of $\bm{\theta}$; the regularization parameters $\alpha$ and $\beta$.

\item Note that the larger $R(\bm{x},0)$ is, the closer $\bm{x}$ is to the global maximum of $\mathcal{L}(\cdot)$. This is the reason we use r defined in line 9 of Algorithm \ref{alg:Framwork} as a performance indicator of the policy $\pi_{\bm{\theta}}$.

\end{itemize}

\subsection{Convergence Analysis}
In this section, we use the ODE method in stochastic-approximation theory to prove that REINFORCE-OPT asymptotically finds a local optimum of $\bm{\theta}$ for any pre-selected positive integers $L$ and $T$. For this purpose, we rewrite the updating rule (\ref{REINFORCE}) as
\begin{align}
\label{theta-update}
\bm{\theta}_{n+1} = \bm{\theta}_n + a_n (g(\bm{\theta}_n) + M_{n+1}), \quad n\in \mathbb{Z}^+\cup \{0\},
\end{align}
where
\begin{align}
\label{gM}
g(\bm{\theta}_n) := \nabla J_T(\bm{\theta}_n)-2\beta\bm{\theta}_n, \quad M_{n+1}:= \hat{\nabla} J_T(\bm{\theta}_n) -  \nabla J_T(\bm{\theta}_n).
\end{align}

It is worth stressing that the standard convergence analysis for stochastic gradient descent (e.g. \cite[Theorem 2]{Bottou1991}) does not apply to (\ref{theta-update}) since the distribution of the error $M_n$ may change with $n$.  In what follows, 
we first state some necessary assumptions required to show our main theoretical result, i.e. Theorem \ref{main-theorem}. Assumption \ref{Assump0} and \ref{Assump1} are standard in the convergence analysis of stochastic approximation algorithms (e.g., \cite[Section 8.1]{Borkar2022} and \cite{KarmakarBhatnagar2018,BhatnagarBorkar1998,Holzleitner2021}). The boundedness assumption (cf. Assumption \ref{Assump1}) is also common for the convergence analysis of algorithms such as stochastic gradient descent (SGD)\footnote{It is stated in \cite[Section 4.1]{Mertikopoulos2020} that a large number of researches in the literature on SGD assume that the trajectory is bounded.}.

\begin{assumption}
\label{Assump0}
\begin{enumerate}
\item $g:\mathbb{R}^{d}\rightarrow \mathbb{R}^d$ defined as in (\ref{gM}) is Lipschitz, i.e. there exists a constant $L_g>0$ such that for any $\bm{\theta}, \bm{\alpha}\in \mathbb{R}^d$ it holds that 
\begin{equation*}
\begin{split}
\lVert g(\bm{\theta}) - g(\bm{\alpha}) &\rVert \leq L_g \lVert \bm{\theta} - \bm{\alpha} \rVert. \\
\end{split}
\end{equation*}

\item $\{ M_n\}$ are integrable and satisfy, for any $n\geq 0$, 
\begin{equation*}
\mathbb{E}\left[ \lVert M_{n+1} \rVert^2 | \mathcal{F}_n  \right] \leq K(1+ \lVert\bm{\theta}_n \rVert^2),
\end{equation*}
where $K>0$ is some scalar, and the filtration $\mathcal{F}_n$ will be defined in Lemma \ref{Mn-mds} later.

\item The step sizes $\{a_n\}$ are positive scalars satisfying
\begin{equation*}
\sum_n a_n = +\infty, \quad \sum_n a^2_n < +\infty. 
\end{equation*}
\end{enumerate}
\end{assumption}

Examples of $\{a_n\}$ satisfying Point 3 above in Assumption \ref{Assump0} are $a_n = \frac{1}{n}, \frac{1}{n \log n}, \cdots$.

\begin{assumption}
\label{Assump1}
Assume that there exists $B_1>0$ such that, almost surely,
\begin{equation*}
\bm{\theta}_n \in \overline{U(\bm{0};B_1)}:=\{ \bm{\theta} \mid \bm{\theta}\in\mathbb{R}^d, \| \bm{\theta}\|\leq B_1\}, \textit{ for all } n\in\mathbb{Z}^+\cup\{0\}.
\end{equation*}
\end{assumption}

\begin{assumption}
\label{Assump2}
Assume that $\sup_{\bm{\theta}\in \overline{U(\bm{0};B_1)}}\{J_T(\bm{\theta})-\beta\|\bm{\theta} \|^2\} <+\infty$, and the set of maximum points,
$$H^*:=\{\bm{\theta}^* \in \overline{U(\bm{0};B_1)} \mid   J_T(\bm{\theta}^*)-\beta\| \bm{\theta}^* \|^2 \geq J_T(\bm{\theta})-\beta\| \bm{\theta}\|^2 \text{ for any } \bm{\theta}\in \overline{U(\bm{0};B_1)}\},$$
is non-empty and contains all the critical points of $(J_T(\bm{\theta})-\beta\| \bm{\theta}\|^2)$ in $\overline{U(\bm{0};B_1)}$.
\end{assumption}

This assumption is on the $\bm{\theta}$-space rather than the $\bm{x}$-space. Hence, it does not rule out the nonconvex optimization problems, such as the ones in Section \ref{experiment-opt}.

\begin{assumption}
\label{Assump3}
Assume that $\overline{U(\bm{0};B_1)}$ as defined in Assumption \ref{Assump1} is an invariant set for the ODE $\dot{\bm{\theta}}(t)=g(\bm{\theta}(t))$.
\end{assumption}

\begin{assumption}
\label{Assump4}
Assume that $\ln\pi_{\bm{\theta}}(\bm{a}|\bm{x})$ is differentiable with respect to $\bm{\theta}$ for any $\bm{\theta}\in \mathbb{R}^d$, any $\bm{a}\in\mathcal{A}$, and any $\bm{x}\in \mathcal{S}$.
\end{assumption}

This assumption is satisfied if $\pi_{\bm{\theta}}$ is modeled as a multivariate Gaussian distribution, with its mean and standard deviations generated by a neural network $\mathcal{N}_{\bm{\theta}}$, as is the case in most of our experiments in this work.

Before proving our main theorem, we need to establish two useful technical results.

\vskip 0.1in 

\begin{lemma}
\label{Mn-mds}
Given point 2 in Assumption \ref{Assump0}, \textit{$\{M_n\}$ defined in (\ref{theta-update}) is a martingale difference sequence with respect to the increasing $\sigma$-fields
}
\[ \mathcal{F}_n:=\sigma(\bm{\theta}_m,M_m,m \leq n), \;   n\geq0, \]
where $\sigma(\bm{\theta}_m,M_m,m \leq n)$ denotes the smallest $\sigma$-field that contains $\{\sigma(M_m)\}_{m=0}^n$ and $\{\sigma(\bm{\theta}_m)\}_{m=0}^n$.
\end{lemma}

\begin{proof}
It suffices to prove that $U_n:=\sum_{k=0}^nM_k$ is a martingale with respect to $\mathcal{F}_n$. First, it is clear that $U_n$ is integrable under point 2 of Assumption \ref{Assump0}. Second, $U_n$ is $\mathcal{F}_n$-measurable since it is the sum of a finite number of $\mathcal{F}_n$-measurable mappings (see \cite[Theorem 1.91]{Klenke2013}). Finally,
\begin{equation*}
\begin{split}
\mathbb{E}[U_{n+1} | \mathcal{F}_n ] &= \mathbb{E}[M_{n+1} | \mathcal{F}_n ] + \mathbb{E}[U_{n} | \mathcal{F}_n ] \\
&= \mathbb{E}[M_{n+1} | \mathcal{F}_n ] + U_{n}\mathbb{E}[\mathbf{1} | \mathcal{F}_n], \quad \text{ since }U_n\text{ is }\mathcal{F}_n\text{-measurable}, \\
&= \mathbb{E}[ \;\hat{\nabla}J_T(\bm{\theta}_{n})-\nabla J_T(\bm{\theta}_{n})  \;|  \mathcal{F}_n ]  + U_{n}, \quad\text{according to the definition of }M_{n+1},\\
&= \mathbb{E} [ \textbf{0} |\mathcal{F}_n] + U_{n}, \\
&= U_{n}.
\end{split}
\end{equation*}
Therefore, $\{U_n\}$ satisfies the conditions required for being a martingale. Hence, $M_n = U_n-U_{n-1}$ is a martingale difference sequence. 
\end{proof}

\vskip 0.1in

\begin{lemma}
\label{Lemma2.1}
\textit{Under Assumptions \ref{Assump0}, \ref{Assump1}, and \ref{Assump3}, for any $\Delta>0$ we have}
\begin{equation}
\label{Lemma2.1-eq}
\lim_{s\rightarrow+\infty} \left( \sup_{t\in[s,s+\Delta]}\lVert \bar{\bm{\theta}}(t) - \bm{\theta}^s(t) \rVert  \right) = 0, \quad a.s.,
\end{equation}
\textit{where}

\begin{enumerate}
 \item[\emph{1.}] \textit{$\bar{\bm{\theta}}(\cdot):[0,+\infty)\rightarrow \mathbb{R}^l$ is the continuous-time linear interpolation for $\{\bm{\theta}_n\}_{n=0}^{\infty}$, defined as $\bar{\bm{\theta}}(t):=\lambda_t \bm{\theta}_n + (1-\lambda_t)\bm{\theta}_{n+1}$ for $t\in [s_n,s_{n+1}]$, where $s_n :=\sum_{i=0}^{n-1} a_i$ for any $n>0$, $s_0:=0$, and $$\lambda_t := \frac{s_{n+1}-t}{s_{n+1}-s_n}\in [0,1].$$ }

\item[\emph{2.}]  $\{\bm{\theta}_n\}_{n=0}^{+\infty}$ is the sequence produced by the iteration (\ref{theta-update}) (the same as (\ref{REINFORCE})).

\item[\emph{3.}] For any $s\geq 0$, $\bm{\theta}^s(t)$ denotes a solution to the following ODE, with $\bm{\theta}^s(s)=\bar{\bm{\theta}}(s)$:
\begin{equation}
\label{ODE-f}
\dot{\bm{\theta}}(t) = g(\bm{\theta}(t)), \quad t\geq s.
\end{equation}
\end{enumerate}
\end{lemma}

\begin{remark}
Note that $\bar{\bm{\theta}}(t)$ and $\bm{\theta}^s(t)$ start from the same point at time $t=s$ and are close to each other only in a finite period $[s,s+\Delta]$. The reason they are not close to each other in the entire time interval $[s,+\infty)$ is that the noise $\{M_{n}\}_{n=1}^{+\infty}$ is a martingale difference sequence rather than a sequence of i.i.d. mean-zero variables. We define some symbols that are frequently used in the proof:
\begin{itemize}
\item For any $t\geq 0$, define
\begin{equation}
\label{bracket-}
[t]^-:=\text{max}_{n\geq 0}\{ s_n | s_n\leq t \} \text{ and } [t]^+:=\text{min}_{n\geq 0}\{ s_n | s_n\geq t \}.
\end{equation}
\item Let $m_{\Delta}$ denote the non-negative integer such that
\begin{equation}
\label{mDelta}
s_{n+m_{\Delta}}:= [s_{n} + \Delta]^-.
\end{equation}
\item $B_g:=\sup_{\bm{\theta}\in \overline{U(\bm{0};B_1)}}\|g(\bm{\theta})\|\in [0,+\infty)$.
\end{itemize}
\end{remark}

\vskip 0.1in
 
\begin{proof}[Proof of Lemma \ref{Lemma2.1}] The proof of Lemma \ref{Lemma2.1} consists of 6 steps. 

\textbf{Step 1}. We prove that Eq. (\ref{Lemma2.1-eq}) can be implied from the fact that for any fixed $\Delta>0$,
\begin{equation}
\label{Lemma2.1a}
\lim_{n\rightarrow+\infty} \left( \sup_{t\in[s_n,[s_n+\Delta]^-]}\lVert \bar{\bm{\theta}}(t) - \bm{\theta}^{s_n}(t) \rVert  \right) = 0, \quad a.s.
\end{equation}
The proof of this implication is given in Appendix B. Thus it is sufficient to prove Eq. (\ref{Lemma2.1a}). To this goal, we will perform the next few steps (steps 2-5) to show that for any fixed and sufficiently large integer $n$ we have
\begin{equation}
\label{preview-eq}
\sup_{t\in[s_n,[s_n+\Delta]^-]}\lVert \bar{\bm{\theta}}(t) - \bm{\theta}^{s_n}(t) \rVert \leq K_n e^{L_g\Delta}+B_g \sup_{k\geq 0}a_{n+k},
\end{equation}
where $K_n$ is as defined in Eq. (\ref{Kn-def}). Based on this fact, we eventually prove in Step 6 that $\lim_{n\rightarrow +\infty} K_n = 0, a.s$. In addition, point 3 of the assumption \ref{Assump0} implies that $\lim_{n\rightarrow \infty} a_n = 0$, which further implies $\lim_{n\rightarrow +\infty} B_g \sup_{k\geq }a_{n+k} = 0$. Therefore, the right-hand side of (\ref{preview-eq}) approaches 0 almost surely as $n\rightarrow +\infty$, which further implies (\ref{Lemma2.1a}).

 Note that $\Delta>0$ is considered fixed. Since $\lim_{n\rightarrow \infty} a_n=0$, there exists $N_{\Delta}>0$ such that whenever $n>N_{\Delta}$ we have $a_n=s_{n+1}-s_n<\Delta \Rightarrow s_{n+1}<s_n+\Delta$, which further implies $m_{\Delta}\geq 1$. In step 2-5, $n$ is a fixed positive integer greater than $N_\Delta$.

\textbf{Step 2}. For any $m\in\mathbb{Z}^+$, we rewrite $\bar{\bm{\theta}}(s_{n+m})$ as
\begin{equation*}
\bar{\bm{\theta}}(s_{n+m}) = \bar{\bm{\theta}}(s_{n}) + \sum_{k=0}^{m-1}a_{n+k}g( \bar{\bm{\theta}}(s_{n+k}) ) + \delta_{n,n+m},
\end{equation*}
where $\delta_{n,n+m}=\sum_{i=0}^{m-1} a_{n+i}M_{n+i+1}$. This equality holds because $\bar{\bm{\theta}}(s_{k})=\bm{\theta}_k$ for any non-negative integer $k$, and the $\bm{\theta}$-updating algorithm specified in (\ref{theta-update}).

Using the definition of $\bm{\theta}^{s}(\cdot)$ in (\ref{ODE-f}), we rewrite $\bm{\theta}^{s_n}(s_{n+m})$ as
\begin{equation*}
\begin{split}
\bm{\theta}^{s_n}(s_{n+m}) &= \bm{\theta}^{s_n}(s_{n}) + \int_{s_n}^{s_{n+m}} g(\bm{\theta}^{s_n}(t)) dt\\
&= \bar{\bm{\theta}}(s_{n}) + \int_{s_n}^{s_{n+m}} g(\bm{\theta}^{s_n}(t)) dt, \quad\text{ since } \bm{\theta}^{s_n}(s_{n})=\bar{\bm{\theta}}(s_{n}),\\
&= \bar{\bm{\theta}}(s_{n}) + \sum_{k=0}^{m-1} a_{n+k}g(\bm{\theta}^{s_n}(s_{n+k}))+\int_{s_n}^{s_{n+m}}\left[ g(\bm{\theta}^{s_n}(y)) - g(\bm{\theta}^{s_n}( [y]^- ) )  \right] dy.
\end{split}
\end{equation*}

\textbf{Step 3.} We take the difference between the two equations in Step 2, and then prove that for any fixed $n\in \mathbb{Z}^+$,
\begin{equation}
\label{diff-bound}
\lVert  \bar{\bm{\theta}}(s_{n+m}) - \bm{\theta}^{s_n}(s_{n+m}) \rVert \leq L_g \sum_{i=0}^{m-1} a_{n+i} \lVert \bar{\bm{\theta}}(s_{n+i}) - \bm{\theta}^{s_n} (s_{n+i}) \rVert + K_n, \text{ for all } m\in \mathbb{Z}^+,
\end{equation}
where
\begin{equation}
\label{Kn-def}
K_n := B_g L_g \sum_{k=0}^\infty a^2_{n+k} + \sup_{j\in \mathbb{Z}^+}\lVert \delta_{n,n+j} \rVert.
\end{equation}
Recall that $B_g:=\sup_{\bm{\theta}\in \overline{U(\bm{0};B_1)}}\|g(\bm{\theta})\|\in [0,+\infty)$, and note that $\sum_{k=0}^\infty a^2_{n+k}<\infty$ due to Point 3 of Assumption \ref{Assump0}. Taking the difference between the two equations in Step 2 gives
\begin{equation}
\begin{split}
\lVert  \bar{\bm{\theta}}(s_{n+m}) - \bm{\theta}^{s_n}(s_{n+m}) \rVert &= \lVert  \sum_{k=0}^{m-1} a_{n+k}\left[ g( \bar{\bm{\theta}}(s_{n+k})) - g( \bm{\theta}^{s_n} (s_{n+k}) )  \right] \\
 &- \int_{s_n}^{s_{n+m}} \left[ g( \bm{\theta}^{s_n} (y) ) -  g( \bm{\theta}^{s_n} ([y]^-) )   \right] + \delta_{n,n+m} \rVert \\
&\leq L_g \sum_{k=0}^{m-1} a_{n+k} \lVert \bar{\bm{\theta}}(s_{n+k}) - \bm{\theta}^{s_n}(s_{n+k})  \rVert \\
&+ \int_{s_n}^{s_{n+m}} \lVert g( \bm{\theta}^{s_n}(y)) - g(\bm{\theta}^{s_n}([y]^-))  \rVert dy +\sup_{j\in \mathbb{Z}^+}\lVert \delta_{n,n+j} \rVert,
\end{split}
\end{equation}
where the property\footnote{For any integrable function $h:\mathbb{R}\rightarrow \mathbb{R}^l$ and any $a<b$, we have $\|\int_{a}^b h(t) dt \|\leq \int_{a}^b \|h(t)\|dt$. For the proof, define $v:=\int_{a}^b h(t) dt$. Then,
$$\|v\|^2 = v^{\text{T}}v = \int_{a}^b v^{\text{T}} h(t) dt \leq \int_{a}^b \|v\|\cdot \|h(t)\| dt.$$
Dividing both the left  and the right  by $\| v\|$ gives the desired inequality.} that the norm of an integral is no greater than the integral of the integrand's norm is used to derive the inequality. We now bound the second term in the above:
\begin{equation*}
\begin{split}
\int_{s_n}^{s_{n+m}} \lVert g( \bm{\theta}^{s_n}(y)) - g(\bm{\theta}^{s_n}([y]^-))  \rVert dy &= \sum_{k=0}^{m-1}\int_{s_{n+k}}^{s_{n+k+1}} \lVert g( \bm{\theta}^{s_n}(y)) - g(\bm{\theta}^{s_n}([y]^-))  \rVert dy \\
&= \sum_{k=0}^{m-1}\int_{s_{n+k}}^{s_{n+k+1}} \lVert g( \bm{\theta}^{s_n}(y)) - g(\bm{\theta}^{s_n}(s_{n+k}))  \rVert dy \\
& \leq \sum_{k=0}^{m-1}\int_{s_{n+k}}^{s_{n+k+1}} L_g \lVert \bm{\theta}^{s_n}(y) - \bm{\theta}^{s_n}(s_{n+k})  \rVert dy\\
&\leq \sum_{k=0}^{m-1}\int_{s_{n+k}}^{s_{n+k+1}} L_g \lVert g(\bm{\theta}^{s_n}(\tau_y))(y-s_{n+k})  \rVert dy, 
\end{split}
\end{equation*}
where the inequality in the fourth line is from the mean value theorem (MVT). Since $\|\bm{\theta}^{s_n}(\tau_y)\|\leq B_1$ (because $\overline{U(\bm{0};B_1)}$ is an invariant set for the ODE $\dot{\bm{\theta}}(t)=g(x(t))$ from Assumption \ref{Assump3}, $\tau_y\geq s_n$, and $\bm{\theta}^{s_n}(s_n)=\bm{\theta}_n\in \overline{U(\bm{0};B_1)}$), we have that 
\begin{equation*}
\begin{split}
\int_{s_n}^{s_{n+m}} \lVert g( \bm{\theta}^{s_n}(y)) - g(\bm{\theta}^{s_n}([y]^-))  \rVert dy  
&\leq L_g B_g \sum_{k=0}^{m-1}\int_{s_{n+k}}^{s_{n+k+1}}  (y-s_{n+k})  dy\\
&\leq L_g B_g \sum_{k=0}^{m-1}\int_{s_{n+k}}^{s_{n+k+1}}  (s_{n+k+1}-s_{n+k})dy\\
&\leq L_g B_g \sum_{k=0}^{m-1} a^2_{n+k},
\end{split}
\end{equation*}
This finishes the proof of Eq. (\ref{diff-bound}).

\textbf{Step 4.} In accordance with the discrete Gronwall inequality (see Lemma \ref{DiscreteGronwal} in Appendix A), Eq. (\ref{diff-bound}) implies that, for any fixed $n\in \mathbb{Z}^+$,
\begin{equation*}
\lVert  \bar{\bm{\theta}}(s_{n+m}) - \bm{\theta}^{s_n}(s_{n+m}) \rVert \leq K_n e^{L_g\sum_{i=0}^{m-1}a_{n+i}},
\end{equation*}
holds for all $m\in\mathbb{Z}^+$. Therefore, for any $m\in \{0,1,...,m_\Delta\}$ (where $m_\Delta$ is defined in (\ref{mDelta}) and $m_{\Delta}\geq 1$ because of the content of the last paragraph in Step 1) we have
\begin{equation}
\label{L4-res}
\lVert  \bar{\bm{\theta}}(s_{n+m}) - \bm{\theta}^{s_n}(s_{n+m}) \rVert \leq K_n e^{L_g \Delta},
\end{equation}
which holds because\footnote{Eq.(\ref{L4-res}) holds trivially for $m=0$.} for $m\in [1,m_\Delta]$,
$$\sum_{i=0}^{m-1}a_{n+i}\leq \sum_{i=0}^{m_\Delta-1}a_{n+i} = s_{n+m_{\Delta}}-s_n = [s_n+\Delta]^- - s_n \leq \Delta.$$

\textbf{Step 5.} Now, we use Eq. (\ref{L4-res}) to bound $\lVert \bm{\theta}^{s_n}(t) - \bar{\bm{\theta}}(t)  \rVert$ for any $t\in [s_{n+m}, s_{n+m+1}]$, where $m$ is any integer in $[0, m_\Delta-1]$. (Recall that $\lVert \bm{\theta}^{s_n}(t) - \bar{\bm{\theta}}(t)  \rVert$ appears in (\ref{preview-eq}).)
\begin{equation*}
\begin{split}
\lVert \bm{\theta}^{s_n}(t) - \bar{\bm{\theta}}(t) \rVert &= \lVert \bm{\theta}^{s_n}(t) - (\lambda_t \bm{\theta}_{n+m}+(1-\lambda_t)(\bm{\theta}_{n+m+1})) \rVert\\
&= \lVert \lambda_t(\bm{\theta}^{s_n}(t) -  \bm{\theta}_{n+m})+(1-\lambda_t)(\bm{\theta}^{s_n}(t)-\bm{\theta}_{n+m+1}) \rVert\\
&= \lVert \lambda_t\left(\bm{\theta}^{s_n}(t) -  \bar{\bm{\theta}}(s_{n+m})\right)+(1-\lambda_t)\left(\bm{\theta}^{s_n}(t)-\bar{\bm{\theta}}(s_{n+m+1})\right) \rVert\\
&= \lVert \lambda_t\left(\bm{\theta}^{s_n}(s_{n+m}) - \bar{\bm{\theta}}(s_{n+m}) +\int_{s_{n+m}}^t g(\bm{\theta}^{s_n}(s))ds \right)\\
& \quad +(1-\lambda_t)\left(\bm{\theta}^{s_n}(s_{n+m+1})-\bar{\bm{\theta}}(s_{n+m+1}) - \int_t^{s_{n+m+1}} g(\bm{\theta}^{s_n}(s))ds\right)\rVert\\
&\leq K_n e^{L_g \Delta} +\int_{s_{n+m}}^{s_{n+m+1}} \lVert g(\bm{\theta}^{s_n}(s))\rVert ds, \quad  \text{by (\ref{L4-res}) and } \lambda_t\in [0,1],\\
&\leq K_n e^{L_g \Delta} + B_g a_{n+m},
\end{split}
\end{equation*}
where recall that $B_g:=\sup_{\bm{\theta}\in \overline{U(\bm{0};B_1)}}\|g(\bm{\theta})\|\in [0,+\infty)$. Therefore, for any $t\in[s_{n},s_{n+m_{\Delta}}]=[s_{n},[s_{n}+\Delta]^-]$, we have
\begin{equation*}
\lVert \bm{\theta}^{s_n}(t) - \bar{\bm{\theta}}(t) \rVert \leq K_n e^{L_g \Delta} + B_g \sup_{m\in\mathbb{Z}^+\cup\{0\}}a_{n+m},
\end{equation*}
which further implies
\begin{equation*}
\sup_{t\in [s_n, [s_n+\Delta]^-]}\lVert \bm{\theta}^{s_n}(t) - \bar{\bm{\theta}}(t) \rVert \leq K_n e^{L_g \Delta} + B_g \sup_{m\in\mathbb{Z}^+\cup\{0\}} a_{n+m}.
\end{equation*}
This proves (\ref{preview-eq}).

\textbf{Step 6.} Taking the limit on both sides of the above inequality  gives
\begin{equation}
\lim_{n\rightarrow +\infty} \sup_{t\in [s_n, [s_n+\Delta]^-]}\lVert \bm{\theta}^{s_n}(t) - \bar{\bm{\theta}}(t) \rVert\leq \lim_{n\rightarrow+\infty}[ K_n e^{L_g \Delta} + B_g \sup_{m \in \mathbb{Z}^+\cup\{0\}} a(n+m) ]=0,
\end{equation}
where the last equality follows from  $\lim_{n\rightarrow +\infty} \sup_{m\in\mathbb{Z}^+\cup\{0\}}a(n+m)=0$ implied from Point 3 in Assumption \ref{Assump0} and $\lim_{n\rightarrow +\infty}K_n =0, a.s  $ to which the proof is given below.

Recall (\ref{Kn-def}),
\begin{equation*}
K_n = B_g L_g \sum_{k=0}^{\infty} a^2_{n+k} + \sup_{j\in\mathbb{Z}^+}\lVert \delta_{n,n+j} \rVert.
\end{equation*}
Note that (i) $B_g<+\infty$ is the upper bound of $\lVert g(\cdot) \rVert$ in the region $\{x | x\in \mathbb{R}^d, \lVert x_n \rVert\leq B_1 \}$ and $L_g$ is the Lipschitz constant; (ii)
  $\sum_{k>0} a^2_{n+k}\rightarrow 0$ as $n\rightarrow \infty$ according to Point 3 of Assumption \ref{Assump0};
(iii) $\delta_{n,n+j} = \sum_{i=0}^{j-1} a_{n+i} M_{n+i+1}$.
Hence, $\lim_{n\rightarrow\infty}B_g \lambda_g \sum_{k=0}^{\infty} a^2_{n+k}=0$. To show that $\lim_{n\rightarrow +\infty}K_n =0, a.s, $ it remains to prove that $$\lim_{n\rightarrow \infty} \sup_{j\in\mathbb{Z}^+} \lVert \delta_{n,n+j} \rVert=0, \; a.s.$$ 
Define $\zeta_n:= \sum_{i=0}^{n-1}a_iM_{i+1}, n\geq 0$. Then $\zeta_n$ is a martingale with proof similar to that of Lemma \ref{Mn-mds}. From Point 2 of Assumption \ref{Assump0}, for any non-negative integer $n$ we have
\begin{equation*}
\begin{split}
\mathbb{E}[ \lVert \zeta_{n+1} - \zeta_n \rVert^2 | \mathcal{F}_n  ] &= \mathbb{E}[a^2_n\lVert M_{n+1} \rVert^2 | \mathcal{F}_n]\\
&\leq a^2_nK(1+ \lVert \bm{\theta}_n \rVert^2 ) \\
&\leq a^2_nK(1+ B_1^2 ), \quad \text{according to Assumption }\ref{Assump1}.
\end{split}
\end{equation*}
Therefore, according to Lemma \ref{C3} (see Appendix A), $\{\zeta_n\}$ converges almost surely, which implies, according to the Cauchy criterion, that
\begin{equation*}
\lim_{n\rightarrow +\infty} (\sup_{m\in\mathbb{Z}^+}\lVert \delta_{n,n+m} \rVert) = \lim_{n\rightarrow +\infty}(\sup_{m\in\mathbb{Z}^+}\lVert \zeta_{n+m} - \zeta_{n} \rVert)=0,\;  a.s.
\end{equation*}
This finishes the proof of (\ref{Lemma2.1a}) and hence the proof of Lemma \ref{Lemma2.1}.
\end{proof}

We are now ready to present and prove our main result.

\begin{theorem}
\label{main-theorem}
Under Assumptions \ref{Assump0}-\ref{Assump4}, the parameter-update rule (\ref{theta-update}) produces ${\bm{\theta}_n}$ that converges almost surely to $H^*$, the set of maximum points of $J_T(\bm{\theta})-\beta\|\bm{\theta} \|^2$ in $\overline{U(\bm{0};B_1)}$ as defined in Assumption \ref{Assump2}, with $J_T(\bm{\theta})$ as defined in (\ref{pmeasure}). That is,
\begin{equation}
\label{thete-H}
\lim_{n\rightarrow +\infty} \inf_{\bm{\theta}^*\in H^*}\| \bm{\theta}_n - \bm{\theta}^*\| =0,\quad  a.s.
\end{equation}
\end{theorem}

\emph{Proof.}
Define $F(\bm{\theta}):=-J_T(\bm{\theta})+\beta\|\bm{\theta}\|^2$ and $\mu:=\inf_{\|\bm{\theta}\|\leq B_1} F(\bm{\theta})$. Then, $\nabla F(\bm{\theta})=-g(\bm{\theta})$. From Point 1 of Assumption {\ref{Assump0}}, $F(\bm{\theta})$ is continuously differentiable, which implies that $F(\bm{\theta})$ attains both the maximum and minimum on the closed and bounded set $\overline{U(\bm{0};B_1)}$. Hence, $\mu>-\infty$ and  $ \emptyset\neq \{ \bm{\theta}\in \overline{U(\bm{0};B_1)} \mid   F(\bm{\theta}) = \mu\} = H^*$ (recall Assumption \ref{Assump2} for the definition of $H^*$).

We state an inequality whose proof is given later: for any $\varepsilon>0$ there exists $T_\varepsilon>0$ such that whenever $t>T_{\varepsilon}$ we have
\begin{equation}
\label{pt1}
 F(\bar{\bm{\theta}}(t)) \leq \mu+(B_g+1)\varepsilon.
\end{equation}
Using (\ref{pt1}), we prove that
\begin{align}
\label{x-star-inclusion}
\Theta^*:=\lim_{n\rightarrow +\infty} \bar{\bm{\theta}}(\tau_n)\in H^*
\end{align}
for any sequence $\{\tau_n\}_{n=1}^{+\infty}$ such that $\lim_{n\rightarrow+\infty} \tau_n = +\infty$ and $\{ \bar{\bm{\theta}}(\tau_n) \}_{n=0}^\infty$ converges.
Specifically, for any $\varepsilon>0$, suppose whenever $n>N_{\varepsilon}$ we have $\lVert \bar{\bm{\theta}}(\tau_n) - \Theta^*\rVert < \varepsilon$. Then, whenever $n>N_\varepsilon$ and $\tau_n>T_\varepsilon$ hold at the same time, one has
\begin{align*}
F(\Theta^*) &= F(\Theta^*) - F(\bar{\bm{\theta}}(\tau_n)) +  F(\bar{\bm{\theta}}(\tau_n))  \\
&\leq B_g \lVert \Theta^* - \bar{\bm{\theta}}(\tau_n)  \rVert + F(\bar{\bm{\theta}}(\tau_n)), \quad \text{by MVT}\\
&\leq B_g \lVert \Theta^* - \bar{\bm{\theta}}(\tau_n) \rVert + \mu + (B_g+1)\varepsilon,\quad \text{by }(\ref{pt1})\\
&< \mu + (2B_g+1)\varepsilon.
\end{align*}
Since $\varepsilon>0$ can be arbitrarily small, the above inequality implies that $ F(\Theta^*)  \leq \mu$, which further implies (according to the definition of $\mu$) that $ F(\Theta^*)  = \mu$. Hence, $\Theta^*\in H^*$.

Next, we prove Eq. (\ref{thete-H}) by contradiction. Assume that Eq. (\ref{thete-H}) is false. Then, there exists $\varepsilon_0>0$ and an increasing sequence $\{s_n\}_{n=0}^\infty\subset \mathbb{R}$ such that $\lim_{n\rightarrow +\infty} s_n = +\infty$ and $\inf_{\bm{\theta}\in H^*}\| \bar{\bm{\theta}}(s_n) - \bm{\theta} \|>\varepsilon_0$. This implies that any limit point $y^*$ of the set $\{\bar{\bm{\theta}}(s_n)\}_{n=0}^\infty$, which exists since $\{\bar{\bm{\theta}}(s_n)\}_{n=0}^\infty$ is bounded\footnote{From Assumption \ref{Assump1}, $\{\bar{\bm{\theta}}(t)\}_{t\geq 0}\in \overline{U(\mathbf{0};B_1)}$.}, satisfies $\inf_{\bm{\theta}\in H^*}\| y^* - \bm{\theta} \|>0$. Therefore, $y^*\notin H^*$, which contradicts the previously derived result of (\ref{x-star-inclusion}). Hence, Eq. (\ref{thete-H}) holds and Theorem \ref{main-theorem} is proved. 

Now, it is sufficient to prove  Eq. (\ref{pt1}). For any $\varepsilon>0$, define $$T_{\varepsilon}:=1+C/\Delta_{\varepsilon}+ S_\varepsilon,$$ where (i) $C:= \sup_{\bm{\theta}\in \overline{U(\bm{0};B_1)}}  (F(\bm{\theta}) -\mu) $; (ii) $\Delta_\varepsilon := \inf_{\bm{\theta}\in O_\varepsilon}\lVert g(\bm{\theta}) \rVert^2$, where $O_\varepsilon:=\{\bm{\theta}\in\overline{U(\bm{0};B_1)}\mid F(\bm{\theta})\geq \mu+\varepsilon\}$; (iii)  $S_\varepsilon>0$ denotes the threshold such that whenever $s\geq S_\varepsilon$ the following inequality holds:
$$  \sup_{t\in[s,s+C/\Delta_\varepsilon+1]}\lVert \bar{\bm{\theta}}(t) - \bm{\theta}^s(t) \rVert \leq \varepsilon, \quad a.s. $$
Note that, almost surely, $S_\varepsilon$ exists due to $(\ref{Lemma2.1-eq})$. We also note that $\Delta_\varepsilon>0$ because  $O_\varepsilon$ is bounded and closed\footnote{This set $O_\varepsilon$ is closed because any limit point $\Theta^*$ of this set belongs to the closed set $\overline{U(\bm{0};B_1)}$ and satisfies $F(\Theta^*)\geq \mu+\varepsilon$ as a result of the continuity of $F$, which implies that $\Theta^*\in O_\varepsilon$.},    $\lVert g(\bm{\theta}) \rVert^2$ is continuous since $g(\bm{\theta})$ is Lipschitz, and   $\lVert g(\bm{\theta}) \rVert^2$ is positive on $O_\varepsilon$ since $O_\varepsilon \cap H^*=\emptyset$\footnote{Recall that in Section \ref{Theorem1-proof} it was proved that $H^*=\{ \bm{\theta}\in \overline{U(\bm{0};B_1)} \mid   F(\bm{\theta}) = \mu\}$.} and   $H^*$ contains all the critical points of $F$ in $\overline{U(\bm{0};B_1)}$) according to Assumption \ref{Assump2}.

Whenever $t>T_\varepsilon$, define $s:=t-C/\Delta_\varepsilon-1$.
\begin{align*}
F(\bar{\bm{\theta}}(t)) &=  F(\bar{\bm{\theta}}(t))-F(\bm{\theta}^s(t)) + F(\bm{\theta}^s(t)),\\
&\leq | F(\bar{\bm{\theta}}(t))-F(\bm{\theta}^s(t)) | + F(\bm{\theta}^s(t)) \\
&\leq B_g \lVert \bar{\bm{\theta}}(t)-\bm{\theta}^s(t) \rVert +  F(\bm{\theta}^s(t)),   \quad\text{by MVT }\\
&\leq B_g \lVert \bar{\bm{\theta}}(t)-\bm{\theta}^s(t) \rVert +\mu+ \varepsilon, \quad\text{It will be justified later} \tag{A}\\
&= B_g \lVert \bar{\bm{\theta}}(s+C/\Delta_\varepsilon+1) -\bm{\theta}^s(s+C/\Delta_\varepsilon+1) \rVert +\mu + \varepsilon, \quad\text{by definition of $s$}\\
&\leq B_g \sup_{\tau\in [s,s+C/\Delta_\varepsilon+1]}\lVert \bar{\bm{\theta}}(\tau) -\bm{\theta}^s(\tau) \rVert +\mu + \varepsilon, \\
&\leq B_g\varepsilon + \mu + \varepsilon, \quad\text{since }s=t-C/\Delta_\varepsilon-1>T_\varepsilon-C/\Delta_\varepsilon-1=S_\varepsilon, \\
&= \mu + (B_g+1) \varepsilon.
\end{align*}
This is (\ref{pt1}). The final step is to justify the inequality (A) above, which is done below.

We prove the inequality $F(\bm{\theta}^s(t))\leq \mu+\varepsilon$ in (A) by contradiction. Assume that
$ F(\bm{\theta}^s(t)) > \mu + \varepsilon.$ This implies the following inequality for any $u\in [s,t]$:
$$ F(\bm{\theta}^s(u)) \geq F(\bm{\theta}^s(t)) \geq \mu+\varepsilon, $$
which follows from the fact $\frac{dF(\bm{\theta}^s(u))}{du}=-\lVert g(\bm{\theta}^s(u)) \rVert^2\leq 0$ since by definition $\dot{\bm{\theta}}^s(u) = g(\bm{\theta}^s)(u)$ for any $u\geq s$. The above inequality implies that for any $u\in [s,t]$,
\[ \bm{\theta}^s(u)\in O_\varepsilon \Rightarrow \lVert g(\bm{\theta}^s(u)) \rVert^2 \geq \Delta_\varepsilon.  \]
Hence,
\begin{align*}
F(\bm{\theta}^s(t)) - F(\bm{\theta}^s(s)) &= \int_s^{t} dF(\bm{\theta}^s(u)) \\
&= \int_s^{t} -\dot{\bm{\theta}^s}^{\text{T}}(u)g(\bm{\theta}^s(u))ds,\quad \text{T denotes matrix transpose}\\
&= \int_s^{t} -g^{\text{T}}(\bm{\theta}^s(u)) g(\bm{\theta}^s(u)) du\\
&= - \int_s^{t}\lVert g(\bm{\theta}^s(u)) \rVert^2 ds\\
&\leq -\Delta_\varepsilon (t-s)\\
&< -C,  \quad\text{from the definition of }s.
\end{align*}
have hence
$$ F(\bm{\theta}^s(t)) < F(\bm{\theta}^s(s)) - C = F(\bm{\theta}^s(s))+\inf_{\bm{\theta}\in \overline{U(\bm{\theta};B_1)}}(\mu -F(\bm{\theta}))  \leq \mu,$$
which is a contradiction, since $F(\bm{\theta})\geq \mu$ for all $\bm{\theta}\in \overline{U(\bm{0};B_1)}$ and $\bm{\theta}^s(t)\in \overline{U(\bm{0};B_1)}$ because of Assumption \ref{Assump3}. This completes the justification for (A), the proof of Eq. (\ref{pt1}), and thus the proof of Theorem \ref{main-theorem}.


\section{Reinforcement Learning for Inverse Problems}
\label{RLandRegu}
In this section, we study the theory of applying REINFORCE-OPT to solve the inverse problem (\ref{Tikhonov}). To convert (\ref{Tikhonov}) to the maximization problem (\ref{opt}), there are multiple options for defining the objective $\mathcal{L}(\cdot)$. For example, one can simply set 
\begin{equation}
\label{neg-obj}
\mathcal{L}(\bm{x}):= -\|f(\bm{x})-\bm{y}^\delta \|^2 - \alpha \Omega(\bm{x}).
\end{equation}

Under the above definition of $\mathcal{L}$ and the conditions required by Theorem \ref{interpret}, the  goal (\ref{rlgoal}) of reinforcement learning is equivalent to solving
\begin{equation}
\label{simple-goal}
 \bm{\theta}^*=\arg\min_{\bm{\theta}\in \mathbb{R}^d} \{ \mathbb{E}_{\bm{x}\sim \mu^{\pi_{\bm{\theta}}} }  \left[ \| f(\bm{x}) - \bm{y}^{\delta} \|^2 + \alpha\Omega(\bm{x})  \right] +\beta \| \bm{\theta} \|^2 \},
\end{equation}
where recall that $\mu^{\pi_{\bm{\theta}}}$ denotes the stationary distribution of the Markov chain $\{\bm{x}_t\}$ produced by following\footnote{$\bm{x}_{t+1}=\bm{x}_t+\bm{a}_t$ where $\bm{a}_t\sim \pi_{\bm{\theta}}(\cdot|\bm{x}_t)$, for any non-negative integer $t$.} $\pi_{\bm{\theta}}$. In what follows, through an investigation of $\pi_{\bm{\theta}^*}$, we reveal the connection between the RL-based method and the classical regularization methods to solve ill-posed inverse problems.

\subsection{Reinforcement-learning Goal (\ref{simple-goal}) Yields Tikhonov Regularization}
\label{ip-exp1}

We draw a connection between REINFORCE-OPT and Tikhonov regularization. For simplicity, we focus on linear inverse problems in the finite-dimensional case, i.e. $A\bm{x} + \delta\xi = \bm{y}^\delta$, where $\bm{x}\in\mathbb{R}^p$, $\bm{y}^{\delta}\in \mathbb{R}^q$, and $A: \mathbb{R}^p \to \mathbb{R}^q$ is a linear operator for some $p, q\in \mathbb{Z}^+$. Specifically, the Tikhonov regularization aims to solve
\begin{equation}
\label{inverse-lp}
\min_{\bm{x}\in \mathbb{R}^p} \|A\bm{x} - \bm{y}^\delta\|^2 + \alpha \|\bm{x}\|^2,
\end{equation}
where $\alpha\geq0$ is a regularization coefficient. For a linear operator $A$, the solution obtained using the Tikhonov regularization method has a closed-form expression: $x^*=(A^{\text{T}} A + \alpha I)^{-1} A^{\text{T}} \bm{y}^\delta$. In this subsection, we demonstrate that if $\pi_{\bm{\theta}}$ is set to a specific form, the mean of the limit distribution of $\{\bm{x}_t\}_{t=1}^\infty$ equals $\bm{x}^*$, where the sequence $\{\bm{x}_t\}_{t=1}^\infty$ is generated following $\pi_{\bm{\theta}^*}$.  

\vskip 0.1in 

\begin{theorem}
\label{ConnectToTikhonov}
Let $\bm{\theta}^*$ be the solution to (\ref{simple-goal}) under the following settings: $f(\bm{x}):= A \bm{x}$, $\beta:=0$, $\Omega(\bm{x}):=\alpha\|\bm{x}\|^2$, and $\pi_{\bm{\theta}}$ is modeled as
\begin{equation}
\label{pi-exampleI}
\pi_{\bm{\theta}}(\cdot|\bm{x}):= \mathcal{N}(\bm{\theta}-\bm{x},\Sigma),\quad\text{for any }\bm{x}\in \mathcal{S}:=\mathbb{R}^q,
\end{equation}
where $\mathcal{N}$ denotes the Gaussian distribution and $\Sigma \in \mathbb{R}^{q\times q}$ is any symmetric positive definite matrix. Then the mean of $\mu^{\pi_{\bm{\theta}^*}}$ equals the solution obtained by applying the Tikhonov regularization method to (\ref{Tikhonov}). 
\end{theorem}

\begin{proof}
With the form of $\pi_{\bm{\theta}}$ in Eq. (\ref{pi-exampleI}), the invariant distribution $\mu^{\pi_{\bm{\theta}}}$ is $\mathcal{N}(\bm{\theta},\Sigma)$ (the proof is postponed to the end of this part). With this information and the conditions assumed in the theorem, we rewrite (\ref{simple-goal}) as follows:
\begin{align*}
\bm{\theta}^*&= \arg \min_{\bm{\theta}} \mathbb{E}_{\bm{x}\sim \mu^{\pi_{\bm{\theta}}}}[ \| A \bm{x}-\bm{y}^\delta \|^2 + \alpha \|\bm{x}\|^2 ]  \\
& = \arg\min_{\bm{\theta}} \| A \mathbb{E} \bm{x}-\bm{y}^\delta \|^2 + \mathbb{E} \| A \bm{x} - A\mathbb{E} \bm{x}\|^2  + \alpha \mathbb{E} \| \bm{x}-\mathbb{E} \bm{x}\|^2 + \alpha \| \mathbb{E} \bm{x} \|^2  \\
&  = \arg\min_{\bm{\theta}} \| A \bm{\theta} -\bm{y}^\delta \|^2 + Trace (A \Sigma A^{\text{T}})  +  \alpha Trace (\Sigma)+ \alpha\|\bm{\theta}\|^2,
\end{align*}
where\footnote{We explain the last equality. Let $\bm{z}:=A \bm{x} - A\mathbb{E} \bm{x}\in\mathbb{R}^q$. Then,
$\mathbb{E} \| A \bm{x} - A\mathbb{E} \bm{x}\|^2=\mathbb{E}\left[ \bm{z}^{\text{T}}\bm{z} \right] = \mathbb{E}\left[ Trace(\bm{z}\bm{z}^{\text{T}}) \right]=Trace\left( \mathbb{E}[\bm{z}\bm{z}^{\text{T}}]\right) = Trace\left( A \text{Var}(\bm{x}) A^T \right) = Trace\left( A \Sigma A^T \right). $
} the expectation $\mathbb{E}$ is taken with respect to $\bm{x}\sim \mu^{\pi_{\bm{\theta}}}$. Then, the first-order condition implies
\begin{equation}
\label{IPsolution1}
\bm{\theta}^* = (A^{\text{T}} A + \alpha I)^{-1} A^{\text{T}} \bm{y}^\delta.
\end{equation}
Since the invariant distribution $\mu^{\pi_{\bm{\theta}^*}}$ of the Markov chain $\{\bm{x}_t\}$ generated by $\pi_{\bm{\theta}^*}$ is $\mathcal{N}(\bm{\theta}^*,\Sigma)$, the mean of $\bm{x}_t$'s limiting distribution as $t\to+\infty$ is $\bm{\theta}^* = (A^{\text{T}} A + \alpha I)^{-1} A^{\text{T}} \bm{y}^\delta$, which equals the Tikhonov regularization solution to the above linear inverse problem.

Now, it remains to prove that with $\pi_{\bm{\theta}}$ defined in (\ref{pi-exampleI}), the invariant distribution $\mu^{\pi_{\bm{\theta}}}$ of $\{\bm{x}_t\}$ generated by $\pi_{\bm{\theta}}$ is the Gaussian distribution $\mathcal{N}(\bm{\theta},\Sigma)$. According to the knowledge on Markov chains (e.g., \cite[Section 1.4]{Douc2018}), the invariant probability measure $\mu^{\pi_{\bm{\theta}}}$ satisfies the following two equations:
\begin{equation}
\label{stationary-dist}
\left\{
\begin{array}{l}
\int_{\bm{x}\in \mathbb{R}^p} \mu^{\pi_{\bm{\theta}}}(d\bm{x}) =1, \\
\mu^{\pi_{\bm{\theta}}}(\mathfrak{B}) = \int_{\bm{x}'\in \mathbb{R}^p} p^{\pi_{\bm{\theta}}}(\mathfrak{B} | \bm{x})\mu^{\pi_{\bm{\theta}}}(d\bm{x}),\quad\text{for any Borel set }\mathfrak{B}\subset\mathbb{R}^p,
\end{array}
\right.
\end{equation}
where $p^{\pi_{\bm{\theta}}}$ denotes the transition kernel of the Markov chain $\{\bm{x}_t\}$, and $p^{\pi_{\bm{\theta}}}(\mathfrak{B} | \bm{x})$ is the probability for the next state $\bm{x}'=\bm{x}+\bm{a}$ to be in $\mathfrak{B}$ conditional on the current state $\bm{x}$. By (\ref{pi-exampleI}), $\bm{a} \sim \mathcal{N}(\bm{\theta}-\bm{x},\Sigma)$, which implies $\bm{x}'\sim \mathcal{N}(\bm{\theta},\Sigma)$. Hence,
$p^{\pi_{\bm{\theta}}}(\mathfrak{B} | \bm{x}) = \mu^{\mathcal{N}(\bm{\theta},\Sigma)}(\mathfrak{B})$ for any $\bm{x}\in \mathbb{R}^p$, where $\mu^{\mathcal{N}(\bm{\theta},\Sigma)}$ denotes the Gaussian probability measure. Then, for any Borel $\mathfrak{B}\subset\mathbb{R}^p$, by Eq. (\ref{stationary-dist}),
\begin{align*}
\mu^{\pi_{\bm{\theta}}}(\mathfrak{B}) =\int_{\bm{x}\in \mathbb{R}^p} p^{\pi_{\bm{\theta}}}(\mathfrak{B} | \bm{x})\mu^{\pi_{\bm{\theta}}}(d\bm{x})
= \mu^{\mathcal{N}(\bm{\theta},\Sigma)}(\mathfrak{B}) \int_{\bm{x}\in \mathbb{R}^p} \mu^{\pi_{\bm{\theta}}}(d\bm{x}) =\mu^{\mathcal{N}(\bm{\theta},\Sigma)}(\mathfrak{B}).
\end{align*}
This proves that $\mu^{\pi_{\bm{\theta}}}$ is the probability measure of the Gaussian distribution $\mathcal{N}(\bm{\theta},\Sigma)$.
\end{proof}

\subsection{Reinforcement-learning Formulation (\ref{simple-goal}) Yields Iterative Regularization}

Consider again the same finite-dimensional linear inverse problem, $A\bm{x} + \delta\xi = \bm{y}^\delta$, as in the previous subsection. Our goal is to demonstrate that when $\pi_{\bm{\theta}}$ also follows a Gaussian distribution with a carefully designed mean, the updates of $\{\bm{x}_t\}$ under $\pi_{\bm{\theta}^*}$ become equivalent to certain well-known (stochastic) iterative schemes in the field of inverse problems. 

\vskip 0.1in

\begin{theorem}
\label{landweber}
Let $\bm{\theta}^*$ denote the solution to (\ref{simple-goal}) under the following settings: $f(\bm{x}):=A\bm{x}$, $\beta:=0$, and $\pi_{\bm{\theta}}$ is modeled as
\begin{equation}
\label{pi-example2}
\pi_{\bm{\theta}}(\cdot|\bm{x}):= \mathcal{N}(\bm{\theta}  - B \bm{x}, \sigma^2 I),\quad\text{for any }\bm{x}\in \mathcal{S}:=\mathbb{R}^q,
\end{equation}
where 
$B := \omega (A^\text{T} A  + \epsilon I)$,  $\omega \in (0, \frac{1}{3(\|A\|^{2}+\epsilon)})$, $I$ represents the identity matrix in $\mathbb{R}^{q\times q}$, and $\|A\|$ denotes the spectral norm of $A$.  Then the iteration rule by $\pi_{\bm{\theta}^*}$ is the same as
\begin{itemize}
\item \emph{the conventional Landweber method}: $\bm{x}_{t+1} = \bm{x}_{t} + \omega A^{\text{T}} (\bm{y}^\delta - A \bm{x}_{t})$, if $\epsilon=\alpha=\sigma = 0$ and $A^{\text{T}}A$ is invertible;
\item \emph{the stochastic Landweber method} \cite[Formulas (5) and (80)]{ZhangChen2023} and \cite[Formula (6)]{ZhangChen2024}: $\bm{x}_{t+1} = \bm{x}_{t} + \omega A^{\text{T}} (\bm{y}^\delta - A \bm{x}_{t}) + \sigma \bm{z}_t$, if $\epsilon=\alpha=0$ and $\sigma \neq 0$ and $A^{\text{T}}A$ is invertible;
\item \emph{the (stochastic) Krasnoselskii-Mann's acceleration method} \cite[Formula (6)]{ZhangHofmann2021}: $\bm{x}_{t+1} = (1-\omega \epsilon) \bm{x}_{t} + \omega A^{\text{T}}( \bm{y}^\delta - A \bm{x}_{t}) + \sigma \bm{z}_t$, if $\epsilon = \alpha \neq 0$.
\end{itemize}
\end{theorem}

\begin{proof}
In the following proof, starting from (\ref{pi-example2}), we derive the expression for $\mu^{\pi_{\bm{\theta}}}$ and substitute it into (\ref{simple-goal}) to obtain the solution $\bm{\theta}^*$ in (\ref{IPsolution2Bais}). Next, in (\ref{xt}), we write the transition rule using $\pi_{\bm{\theta}^*}$ and demonstrate that it is equivalent to the (stochastic) Landweber iteration under the conditions specified in the theorem.

From (\ref{iteration}) and (\ref{pi-example2}), the transition rule in the trajectory generated by $\pi_{\bm{\theta}}$ is (recall that $\bm{a}_t\sim \pi_{\bm{\theta}}(\cdot|\bm{x}_t)$) 
\begin{equation}
\label{transition-exampleII}
\bm{x}_{t+1} = \bm{x}_t + \bm{a}_t = \bm{\theta} + (I - B) \bm{x}_t + \sigma \bm{z}_t,
\end{equation}
where $\bm{z}_t\sim \mathcal{N}(0,I)$ is independent of $\bm{x}_t$. This implies
\begin{equation}
\bm{x}_t = \sum_{i=0}^{t-1} (I - B)^i \bm{\theta} + \sum_{i=0}^{t-1} (I - B)^i\sigma \bm{z}_{t-i-1} +  (I - B)^t \bm{x}_0.
\end{equation}
This equality and the fact  $\{\bm{z}_t\}$ being independent $\mathcal{N}(0,I)$ imply that $\bm{x}_t$ follows $\mathcal{N}(\mu_t,\Sigma_t)$ given that $\bm{x}_0$ is a fixed vector, where
$$\mu_t:=\sum_{i=0}^{t-1} (I - B)^i \bm{\theta} +  (I - B)^t \bm{x}_0,$$
and
$$\Sigma_t:= \sum_{i=0}^{t-1} (I - B)^i \sigma I \left( (I - B)^i \sigma I \right) ^{\text{T}} =  \sigma^2 \sum_{i=0}^{t-1} (I - B)^{2i}. $$
Since $\omega \in (0, \frac{1}{3(\|A\|^{2}+\epsilon)})$, we have\footnote{$\|B\|\leq \omega (\|A\|^2+\epsilon)<1/3$. Also, by definition $B$ is symmetric and positive definite. Hence, $\|I-B\|<1$.} $\|I - B\|<1$. Hence,
$$\lim_{t\rightarrow \infty} \mu_t = B^{-1} \bm{\theta} \text{ and } \lim_{t\rightarrow \infty} \Sigma_t = \sigma^2 (2B - B^2)^{-1}.$$
This implies, according to Levy's theorem for characteristic functions (e.g., \cite[Theorem 26.3]{Billingsley1995}), that the limiting distribution $\mu^{\pi_{\bm{\theta}}}$ of $\bm{x}_t$ as $t\rightarrow \infty$ is Gaussian, with its mean equal to $B^{-1} \bm{\theta}$ and its variance equal to $\sigma^2 (2B - B^2)^{-1}$, namely
\begin{equation}
\label{xNormal}
\mu^{\pi_{\bm{\theta}}} = \mathcal{N}(B^{-1}\bm{\theta}, \sigma^2(2B - B^2)^{-1}).
\end{equation}
Now we derive an expression for the solution $\bm{\theta}^*$ to (\ref{simple-goal}) when $\Omega(\bm{x})=\alpha\|\bm{x}\|^2$ and $\beta=0$. We rewrite (\ref{simple-goal}) as
\begin{align*}
& \min_{\bm{\theta}} \mathbb{E}_{\bm{x}\sim d^{\pi_{\bm{\theta}}}}[ \| A \bm{x}-\bm{y}^\delta \|^2 + \alpha \|\bm{x}\|^2 ]  \\
& \qquad = \min_{\bm{\theta}} \| A \mathbb{E} \bm{x}-\bm{y}^\delta \|^2 + \mathbb{E} \| A \bm{x} - A\mathbb{E} \bm{x}\|^2  + \alpha \mathbb{E} \| \bm{x}-\mathbb{E} \bm{x}\|^2 + \alpha \| \mathbb{E} \bm{x} \|^2  \\
& \qquad = \min_{\bm{\theta}} \| A \mathbb{E} \bm{x}-\bm{y}^\delta \|^2 + \text{Trace}\left(A \cdot \text{Var} \cdot (\bm{x})A^T \right)  + \alpha \text{Trace}\left( \text{Var} \cdot (\bm{x}) \right) + \alpha \| \mathbb{E} \bm{x} \|^2  \\
& \qquad = \min_{\bm{\theta}} \| AB^{-1}\bm{\theta} -\bm{y}^\delta \|^2 + \sigma^2 \text{Trace}\left( A(2B-B^2)^{-1} A^T \right)  \\
&\qquad \qquad \qquad + \alpha \sigma^2 \text{Trace}\left( (2B-B^2)^{-1} \right) + \alpha \| B^{-1}\bm{\theta} \|^2, \qquad\text{by } (\ref{xNormal}).
\end{align*}
From the first-order necessary condition for optimality, the solution $\bm{\theta}^*$ is
\begin{equation}
\label{IPsolution2Bais}
\bm{\theta}^* = \left( B^{-1}A^\text{T}AB^{-1}+\alpha B^{-2} \right)^{-1} B^{-1} A^\text{T}\bm{y}^{\delta}.
\end{equation}
Plugging this result to (\ref{transition-exampleII}) gives the iteration rule in the trajectory generated by $\pi_{\bm{\theta}^*}$: 
\begin{equation}
\begin{split}
\label{xt}
\bm{x}_{t+1} &= \bm{x}_t + \bm{a}_t = \bm{x}_t+ \bm{\theta}^* - B\bm{x}_t + \sigma \bm{z}_t \\
&=\bm{x}_t + \left( B^{-1}A^\text{T}AB^{-1}+\alpha B^{-2} \right)^{-1} B^{-1} A^\text{T}\bm{y}^{\delta}  - B \bm{x}_t + \sigma \bm{z}_t \\
&=\bm{x}_t + \left( B^{-1}(A^\text{T}A+\alpha I)B^{-1} \right)^{-1} B^{-1} A^\text{T}\bm{y}^{\delta}  - B \bm{x}_t + \sigma \bm{z}_t \\
&=\bm{x}_t +  B(A^\text{T}A+\alpha I)^{-1} A^\text{T}\bm{y}^{\delta}  - B \bm{x}_t + \sigma \bm{z}_t \\
&= \bm{x}_t + \omega (A^\text{T} A  + \epsilon I) \left[ (A^\text{T} A  + \alpha I)^{-1} A^\text{T}\bm{y}^{\delta}  - \bm{x}_t \right]  + \sigma \bm{z}_t, 
\end{split}
\end{equation}
which leads to the proposed iterative regularization schemes for different choices of parameters $\epsilon, \alpha$ and $\sigma$. 
\end{proof}


Before we close this section, let us now highlight the essential contributions of this work:
\begin{itemize}
\item \emph{Theoretical contributions to RL-based methods}: This work advances the theoretical understanding of RL-based methods for solving potentially nonconvex optimization problems with continuous solution spaces. To our knowledge, we are the first to establish the equivalence between the complex objective of RL-based methods and the stochastic version (\ref{rlgoal-opt}) of the optimization goal, and the first to provide a theoretical convergence analysis for RL-based optimization methods. In Theorem \ref{main-theorem}, we rigorously prove that, under standard assumptions, REINFORCE-OPT generates a sequence $\{\bm{\theta}_t\}_{t=0}^{\infty}$ that almost surely converges to a locally optimal value. Our proof employs the ordinary-differential-equation method (e.g., \cite{BorkarMeyn2000} and \cite[Section 8.1]{Borkar2022}).
    
\item \emph{Introduction of RL to inverse problems}: This work rigorously introduces an RL algorithm to the field of inverse problems and, for the first time, connects RL with classical regularization methods (both variational regularization and iterative regularization) for solving general inverse problems (\ref{inversep}), as discussed in Section \ref{RLandRegu}. Existing research applying RL to inverse problems is limited and primarily focuses on empirical applications. For example, the authors in \cite{Sridharan2022} use deep RL\footnote{Deep RL leverages deep neural networks as function approximators in RL algorithms. See \cite{Yuxi-Li2018} for more details.} to solve a physical chemistry problem involving the inversion of nuclear magnetic resonance spectra for molecular structures. The authors in \cite{DAversana2022} design an inversion algorithm based on Q-learning for geoscience applications, while the authors in \cite{SantosaAnderson2022} show how to transform Bayesian sequential optimal experimental design, a method for solving inverse problems, into a problem solvable by RL. \cite{Dong2024} applies RL to jointly learn good samplers and reconstructors for MRI inverse problems.

\item \emph{Practical applicability}: Our numerical results given in Section \ref{Simulations} indicate that REINFORCE-OPT is capable of solving inverse problems of the form (\ref{inversep}) involving complex functions (i.e. non-linear mappings without explicit formula) $f$, showcasing its practicality and success as a method for inverse problems.
\end{itemize}

\section{Numerical experiments}
\label{Simulations}

\subsection{Performance of REINFORCE-OPT for Continuous Optimization}
\label{experiment-opt}

In this section, we illustrate the advantages of REINFORCE-OPT over several mainstream optimization methods through  experiments. For ease of visualization, most of the experiments are conducted in one- or two-dimensional $\bm{x}$ spaces, without loss of generality.

For the purpose of illustrating these advantages, we report quantities obtained in the training process of the algorithm, rather than the output stated in Steps 3 and 4 in Algorithm \ref{alg:Framwork1}. For example, to show that the algorithm may be able to escape from local optimums, we report the path $\{\bm{x}_t\}_{t=0}^{T-1}$ produced by the trained $\pi_{\theta}$. To show that the algorithm is more robust than GA and PSO against the choice of initial $\bm{x}_0$ (GA and PSO choose an initial population, i.e., a group of $\bm{x}_0$-values), for a proper comparison, we report the best solution candidate found during the training of REINFORCE-OPT in Algorithm \ref{alg:Framwork}.

\subsubsection{Escape from Local Optimum}
\label{escape-from-local}
In this experiment, the problem to solve is $\max_{x\in\mathbb{R}}  \mathcal{L}(x)$ with $\mathcal{L}(x)=-(x^2-1)^2 - 0.3(x-1)^2+5$, where the starting value is fixed at $x_0=-1.5$. We follow (\ref{REINFORCE}) to solve this problem. For a comparison with the gradient ascent algorithm in a one-dimension case, we set the action space as $\{-1,0,1\}$, and construct the action distribution $\pi_\theta(\cdot|\bm{x})$ as a neural net whose output consists of the probability of each action. The hyper-parameters of this neural net are selected by the commonly used validation-set approach (i.e., from a pre-selected candidate set, the set of hyper-parameter values that lead to the best performance are selected). For example, the selected hidden layer structure is (10,5). The number of $\bm{\theta}$-updates is set as 3,500.

Also, since gradient ascent involves the quantity of a step size $\gamma$, for a better comparison, we add this quantity to REINFORCE-OPT by modifying the MDP iteration rule (\ref{iteration}) as $ x_{t+1} = x_t + \gamma a_t$, where we set $\gamma = 0.1$.
With the trained parameter value $\hat{\bm{\theta}}^*$, we generate a path $\{x_t\}_{t=1}^{35}$ using $\pi_{\hat{\bm{\theta}}^*}$. That is, $x_{t+1} =x_{t} + \gamma a_t,$ where $a_t$ equals the mode of the distribution $\pi_{\hat{\bm{\theta}}^*}(\cdot|x_t)$. The results are plotted in Figure \ref{escape-local}, which illustrates the ability of REINFORCE-OPT to escape from the local maximum to the global maximum of $x=1.0$. The performance of the gradient ascent (performed 35 steps) is also plotted. As the graph shows, its search is trapped in the local maximum.

\begin{figure}[H]
\begin{tabular}{cc}
	\centering
        \subfloat[REINFORCE Trajectory]{\includegraphics[scale=0.35]{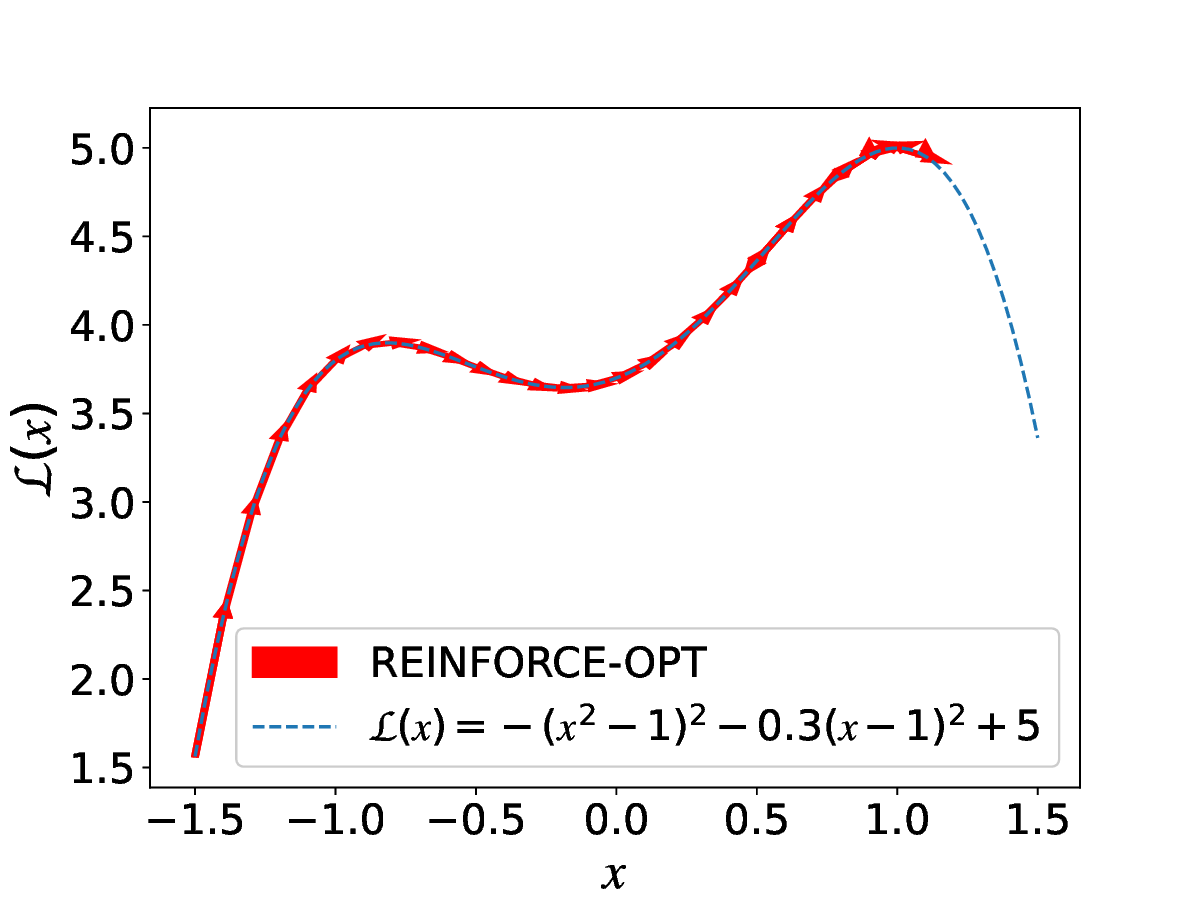} }
        &
        \subfloat[Gradient-Ascent Trajectory]{\includegraphics[scale=0.35]{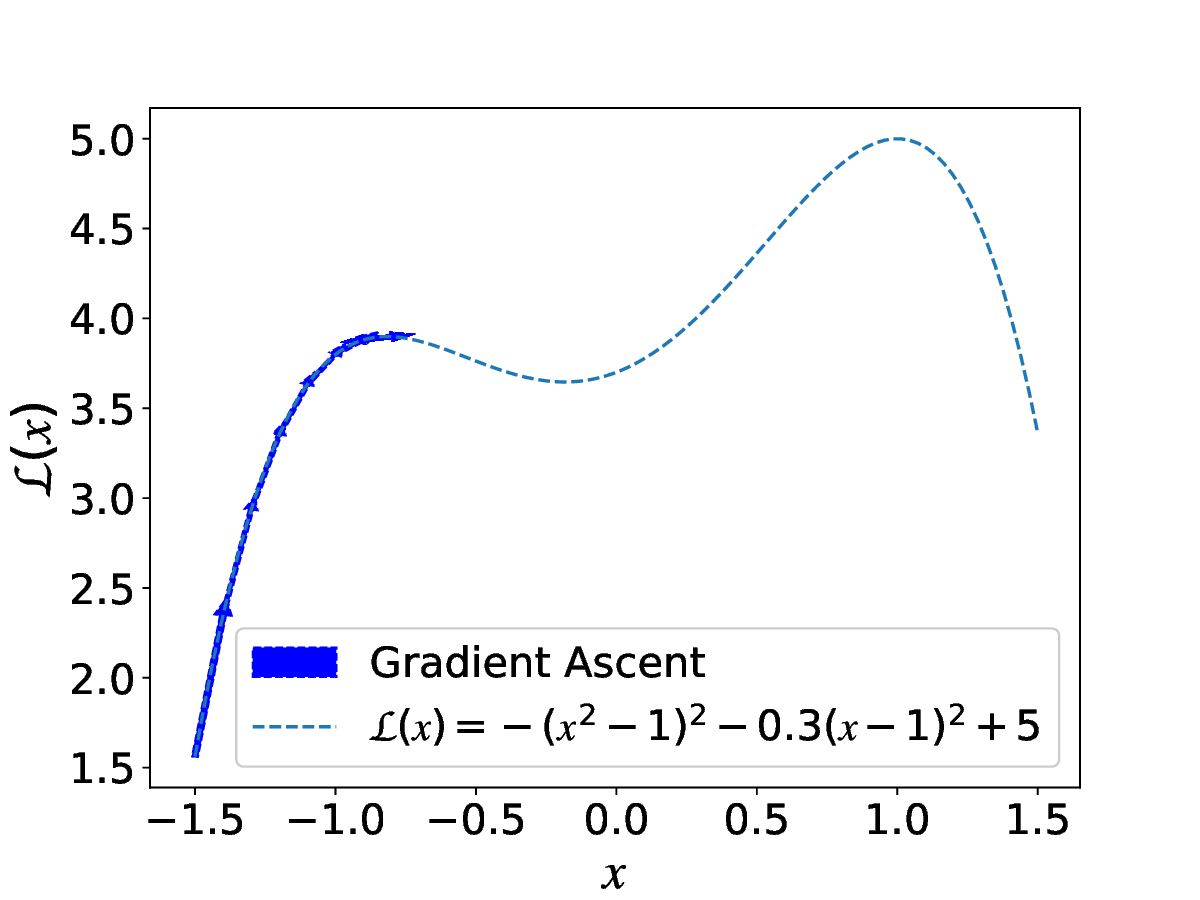} }
\end{tabular}
       	\caption{Illustration: REINFORCE-OPT escapes from the local minimum. It seems that the gradient ascent agent moves fewer steps than REINFORCE-OPT. But in fact it keeps moving back and forth around the local minimum.}
\label{escape-local}
\end{figure}

\begin{figure}[H]
\begin{tabular}{cc}
	\centering
        \subfloat[The 1D Case]{\includegraphics[scale=0.35]{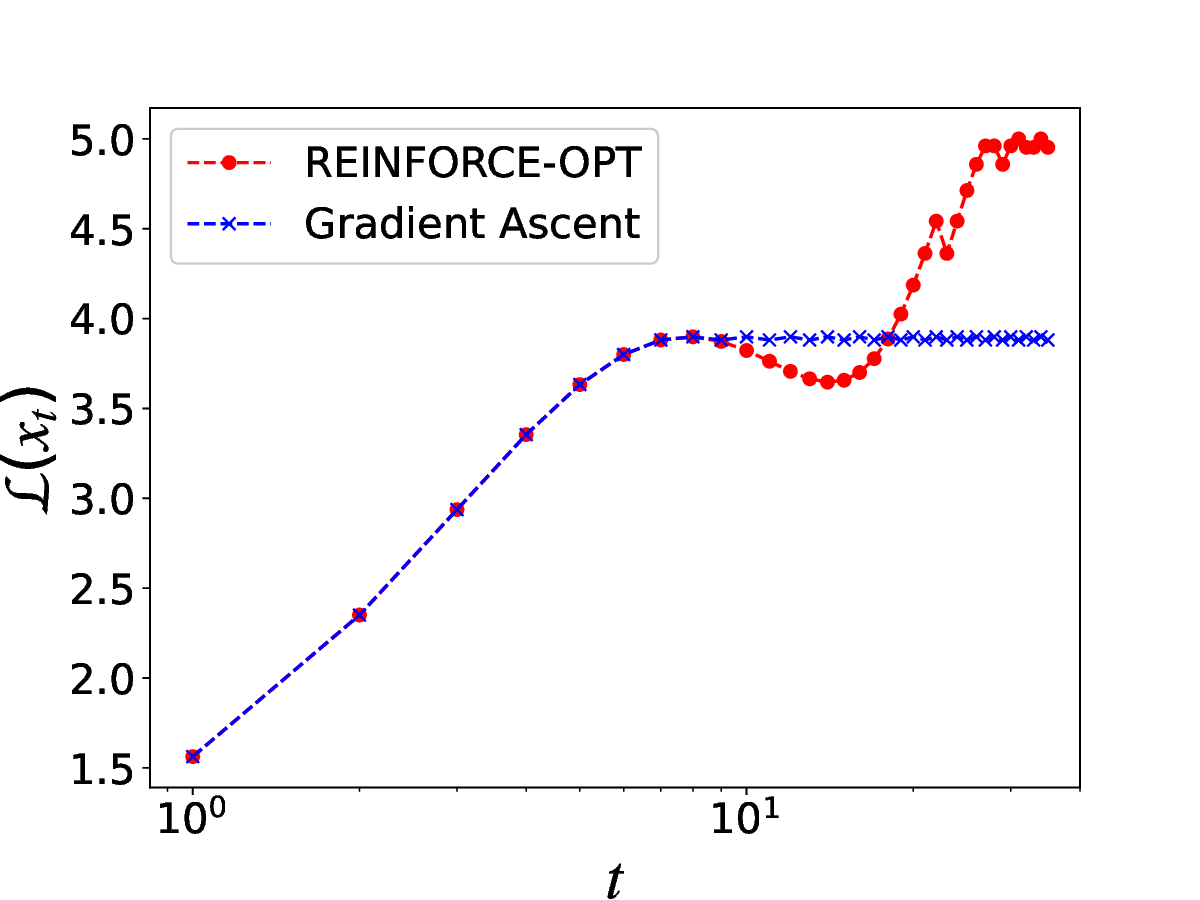} }
        &
        \subfloat[The 2D Case]{\includegraphics[scale=0.35]{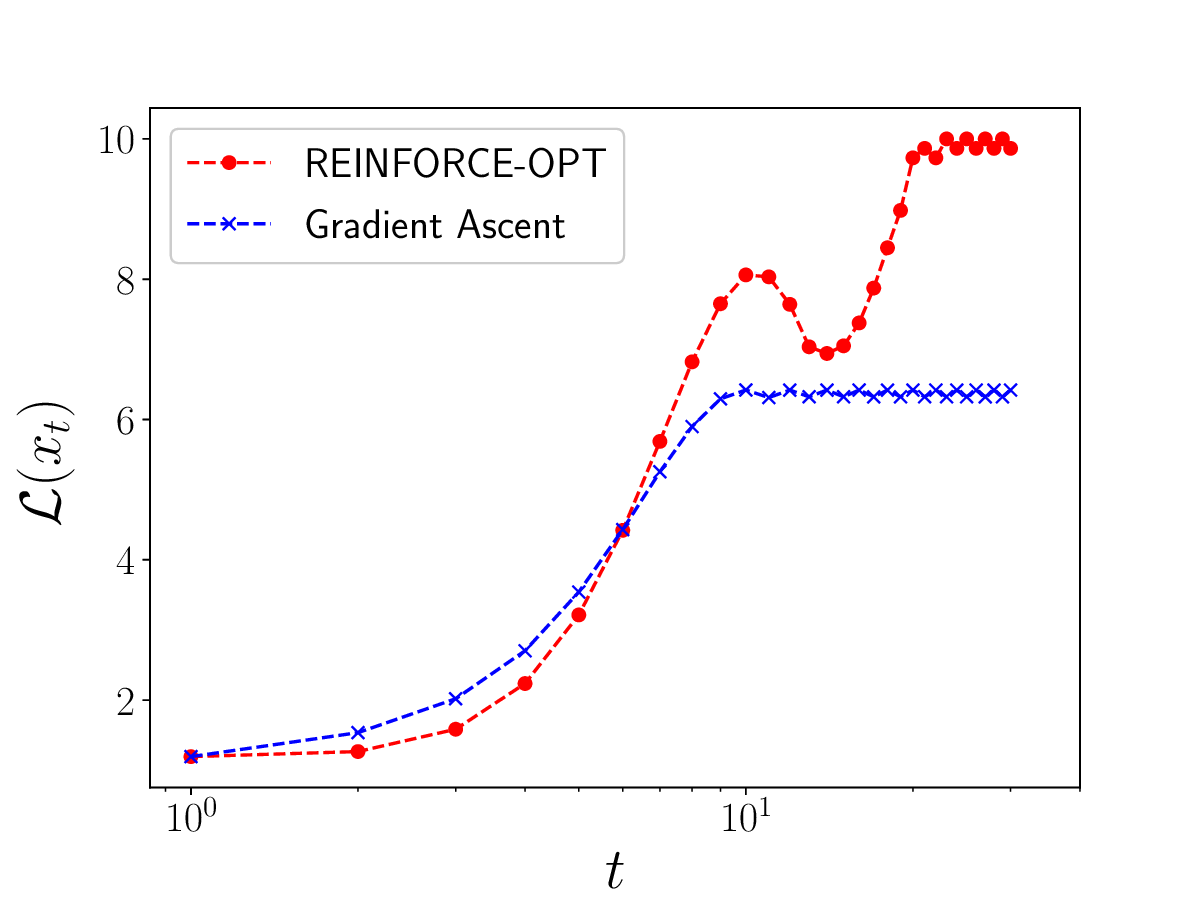} }
\end{tabular}
       \caption{The fitness trajectory of the two methods in Figure \ref{escape-local}, where the fitness at step $t$ is $\mathcal{L}(x_t)$.}
\label{loss-traj}
\end{figure}

Next, we repeat the experiment with a more complicated objective that resembles Eq. (78) in \cite{Bjorn2023}, where we set the action space as $\{\frac{(a,b)}{\sqrt{a^2+b^2+10^{-6}}}| a,b\in\{-1,0,1\} \}$ and $\gamma=0.1$. Hence, each $\bm{x}$-step is normalized to have a size of $0.1$. The result is shown in Figure \ref{escape-local2D}, which again verifies the ability of REINFORCE-OPT to escape from the local optimum.

\begin{figure}[H]
\begin{tabular}{cc}
	\centering
        \subfloat[Objective: $\mathcal{L}(x_1,x_2)=8+\cos(10x_1)+\cos(10x_2)-5(x_1^2+x_2^2)$.]{\includegraphics[scale=0.4]{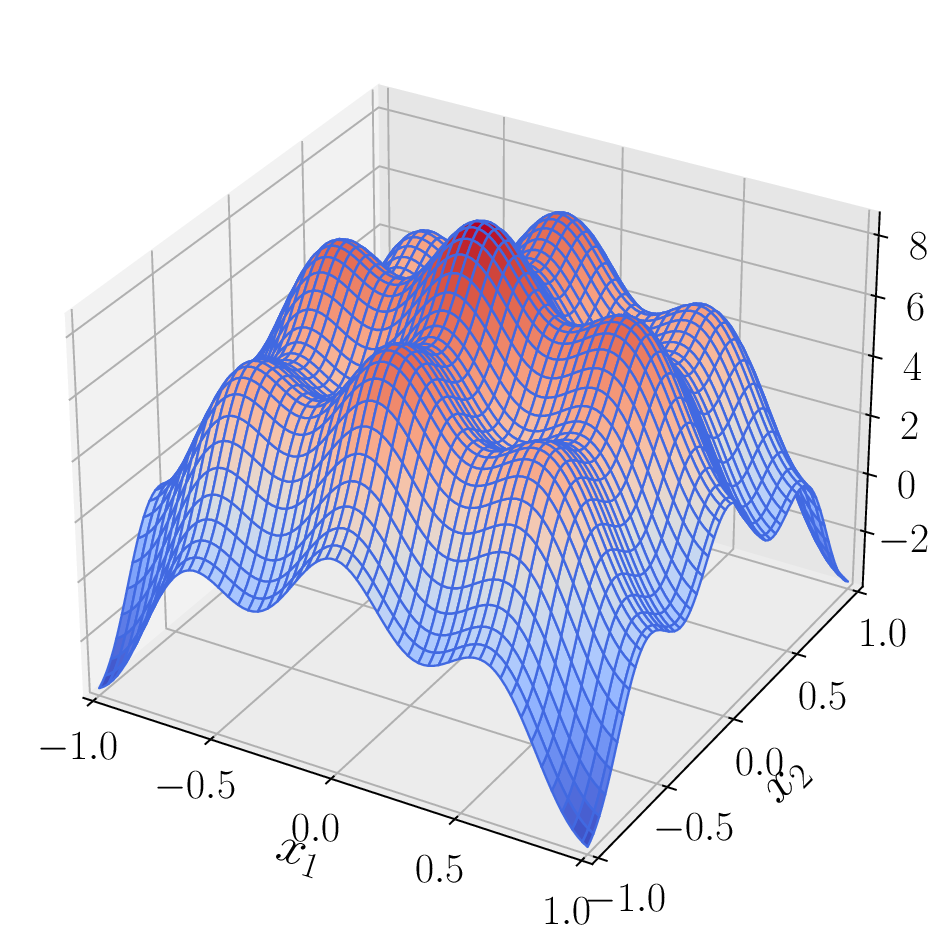} }
        &
        \subfloat[Searc h Trajectory in the contour plot of $\mathcal{L}(\bm{x})$. Note that $\bm{x}$-points on the same curve have equal $\mathcal{L}$ values. ]{\includegraphics[scale=0.4]{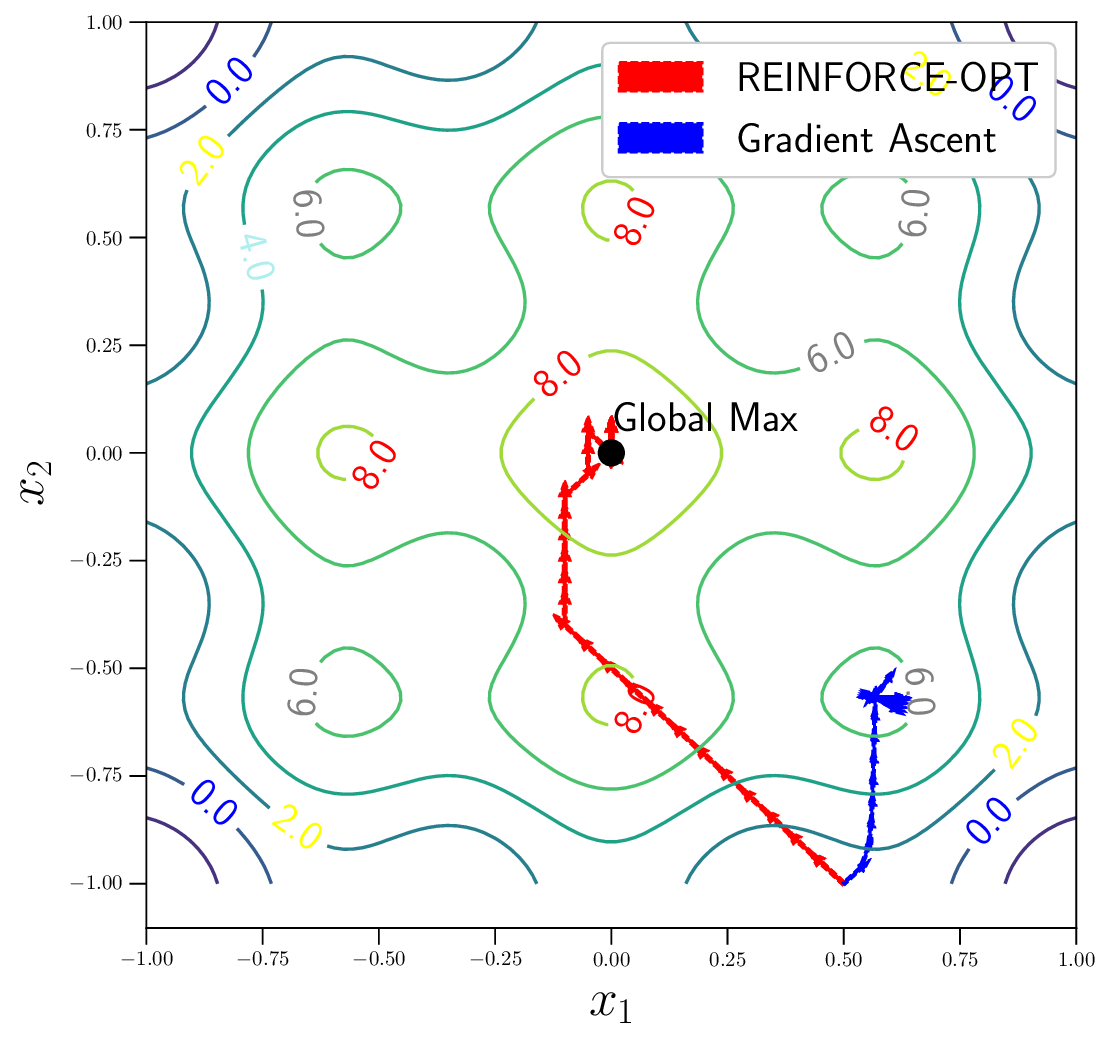} }
\end{tabular}
       	\caption{Illustration: REINFORCE-OPT escapes from the local minimum.}
\label{escape-local2D}
\end{figure}

\subsubsection{Robustness Against Initial Values} \label{Robustness} In this subsection, we compare REINFORCE-OPT ($\ref{REINFORCE}$) with other three popular global optimization methods, the genetic algorithm \cite{gad2023pygad},  the particle swarm optimization \cite{pyswarmsJOSS2018}, and the cross-entropy method for optimization \cite[Algorithm 4.1]{BOTEV201335}. All four algorithms start from an initial value $\bm{x}_0$ (or an initial population), and iteratively decide on the next group of $\bm{x}$-values to be examined according to a search heuristic. Each such group is called a generation. For REINFORCE-OPT, more specifically, we refer a generation to $\{\bm{x}_t^l: t\in \{0,1,2,...,T-1\}, l\in\{0,1,2,...,L-1\} \}.$ This generation is used to compute $\hat{\nabla}J_T(\bm{\theta})$ in (\ref{gradient-est}), the term used for one update of $\bm{\theta}$. Also, the search heuristic $\pi_{\bm{\theta}}(\cdot|\bm{x})$ for REINFORCE-OPT is modeled as a multivariate normal distribution with its mean and standard deviation equal to the output of a neural net $\mathcal{N}_{\theta}(\bm{x})$.

The selections of the initial $\bm{x}$ values are difficult to make. Therefore, the robustness of the method's performance against the choice of initial values is important. We compare the robustness of three methods  through solving a simple but representative problem: $\max_{\bm{x}\in\mathbb{R}^k} \mathcal{L}(\bm{x})$ with
$$ \mathcal{L}(\bm{x}):=-\ln((\bm{x}-\bm{m}_1)^2+0.00001)-\ln((\bm{x}-\bm{m}_2)^2+0.01),$$ 
where all the entries of $\bm{m}_1\in \mathbb{R}^k$ are set as $-0.5$, and all the entries of $\bm{m}_2\in \mathbb{R}^k$ are set as $0.5$. The graph of $\mathcal{L}$ when $k=2$ is plotted in Figure \ref{2D-f}, which shows that $\mathcal{L}$ has a global maximum at $\bm{m}_1$ and a local maximum at $\bm{m}_2$. For each $k\in\{2,6\}$ and each method, we perform two experiments with varied $\bm{x}_0$ values (or varied initial populations). Table \ref{compare-parameters} reports the hyper-parameters' values (selected by the validation-set approach) for one\footnote{In Table \ref{compare-parameters}, the second column specifies which one of the two experiments the reported hyper-parameters' values correpond to.} of the two experiments. The hyper-parameter values for other experiments can be found in our code, i.e. \url{http://github.com/chen-research/REINFORCE-OPT}.

\begin{figure}[H]
	\centering
	\includegraphics[scale=0.5]{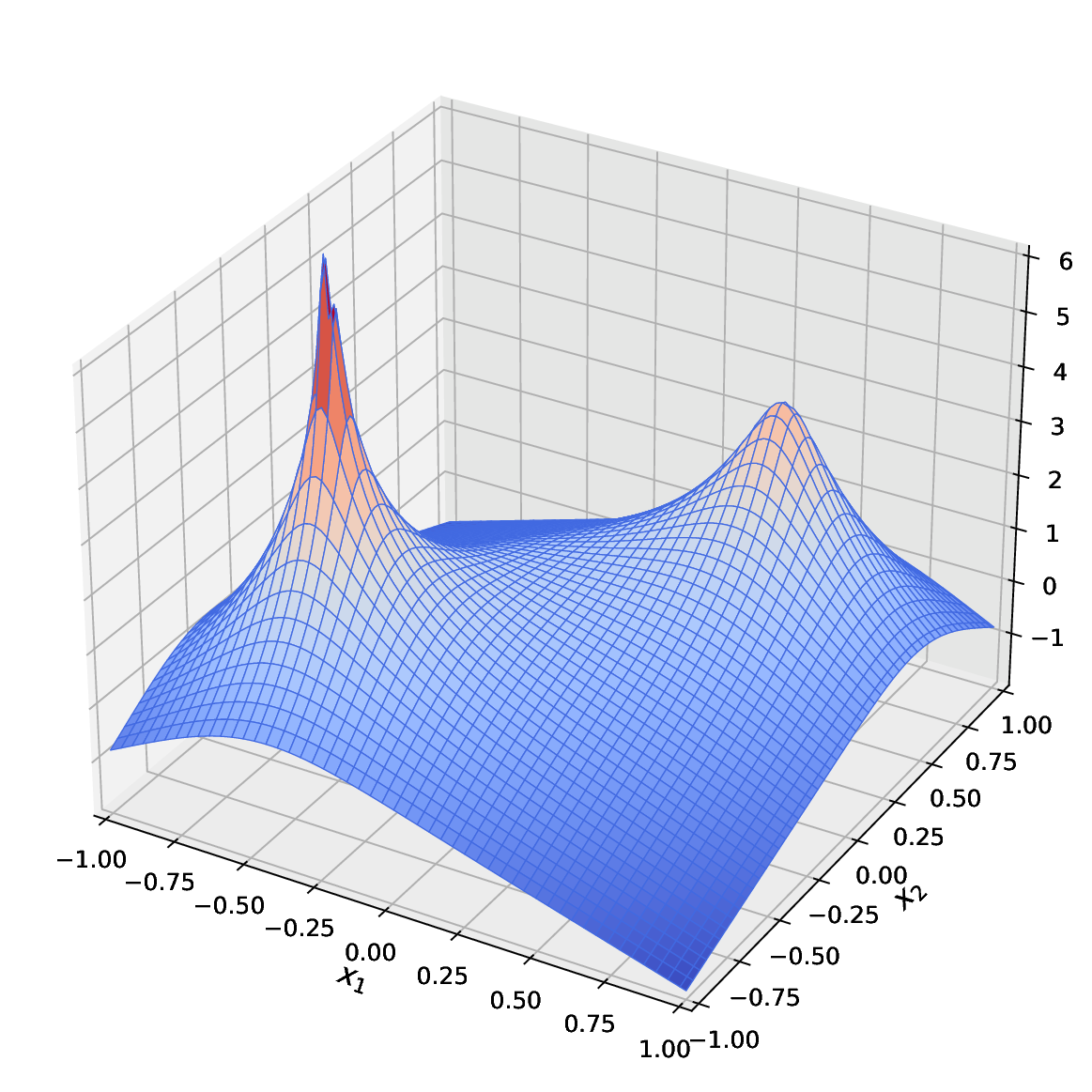}
       	\caption{The graph of $\mathcal{L}(\bm{x}):=-\ln((\bm{x}-\bm{m}_1)^2+0.00001)-\ln((\bm{x}-\bm{m}_2)^2+0.01)$ for the 2D case.}
\label{2D-f}
\end{figure}

\begin{table}[!htb]
\caption{Hyper-parameters for one of the two experiments performed for each method when $k=6$.}
\label{compare-parameters}
\centering%
\begin{tabular}{ p{2cm}  p{3cm} p{7.5cm}}
\midrule
Method & The initial population or $\bm{x}_0$. & The hyper-parameters' values.\\
\toprule
REINFORCE-OPT ($\ref{REINFORCE}$) & Each entry of $\bm{x}_0$ is uniformly sampled from the interval of $[0.2,0.3]$. & Hidden layers: (18,18,18). Activation: tanh. $L=60$. $T=25$. $a_n=\frac{5}{50,000+n}$. \\
Genetic Algorithm \cite{gad2023pygad}
 & Each entry of any candidate in the initial population is uniformly sampled from [0,1]. & Number of mating-parent pairs = 20. Gene mutation percent$=10\%$. Crossover type: single point.\\
Particle Swarm Optimization \cite{pyswarmsJOSS2018} & Each entry of any candidate in the initial population is uniformly sampled from [0,1].  &  Cognitive parameter $c_1=0.2$. Social parameter $c_2=0.2$. Inertia parameter $w=0.5$. \\
Cross-entropy Method \cite[Algorithm 4.1]{BOTEV201335} & Each entry of any candidate in the initial population is uniformly sampled from [0,1].  & $N^e=30$, where $N^e$ denotes the number of top solution candidates in a generation used to update the Gaussian mean $\mu$ and standard deviation $\sigma$. $\alpha=0.9$, which is a weight used in the update $(\mu',\sigma')=(\mu,\sigma)+\alpha((\mu,\sigma))$.\\
\bottomrule
\end{tabular}
\footnotetext{All hyper-parameters are selected by the validation-set approach, except that for a fair comparison, the generation size for the genetic algorithm, particle swarm optimization and the cross-entropy method, is set equal to REINFORCE-OPT's generation size $LT$.}
\end{table}
For each pre-selected initial value (or initial population), we perform one of the four investigated methods, and report the $\mathcal{L}$-value of the best solution candidate found among all the generations. Specifically, let $\{\mathbf{x}_n\}$ denote the best solution candidate found up to the $n^{th}$ generation. We report $\textbf{x}^*=\arg\max_{n}\mathcal{L}(\textbf{x}_n)$ in Tables \ref{compare-results} and \ref{robustk6}. The results indicate that if the range of the initial population does not cover the global maximum, GA will be trapped in the local maximum. PSO seems unaffected by the selection of the initial population if $k=2$. However, this robustness disappears when the dimension increases a little to $k=6$. On the contrary, REINFORCE-OPT is able to find solutions close to the global maximum in both of the two cases, regardless of the location of the initial value $\bm{x}_0$, even if $\bm{x}_0$ is much closer to the local maximum $\bm{m}_2$ that to the global maximum $\bm{m}_1$. This indicates that the robustness of REINFORCE-OPT is stronger than that of GA and PSO.

\begin{table}[!htb]
\caption{Experiment Results When $k=2$}
\label{compare-results}
\centering%
\begin{tabular}{ p{2cm}  p{5cm}  p{2.5cm} p{1cm} }
\midrule
 & Initial Value or Population & Best Solution Candidate $\textbf{x}^*$ & $f(\textbf{x}^*)$ \\
\toprule
REINF. OPT. & $\bm{x}_0=[0,0]$.   &  $(-0.500,-0.500)$       & 10.752\\
                               & $\bm{x}_0\sim$ uniform$[0.2,0.3]$.  &$(-0.500,-0.501)$   &10.761\\

Genetic Algo.          & For any $\bm{x}\in\mathcal{P}_0$, $\bm{x}\sim$ uniform$[-1,1]$.                    & $(-0.499,-0.502)$  & 10.266\\
                               
                               & For any $\bm{x}\in\mathcal{P}_0$, $\bm{x}\sim$ uniform $[0,1].$ & (0.494, 0.494)   & 3.917 \\
Part. Swarm & For any $\bm{x}\in\mathcal{P}_0$, $\bm{x}\sim$ uniform$[-1,1]$.                &  $ (-0.500,-0.500) $  & 10.815 \\
                               & For any $\bm{x}\in\mathcal{P}_0$, $\bm{x}\sim$ uniform$[0,1].$ &   $ (-0.500,-0.500) $  & 10.815 \\
Cross-Entropy. & For any $\bm{x}\in\mathcal{P}_0$, $\bm{x}\sim$ uniform$[-1,1]$.                &  $ (-0.500,-0.500) $  & 10.815 \\
                               & For any $\bm{x}\in\mathcal{P}_0$, $\bm{x}\sim $uniform$[0,1].$ &   $ (0.495,0.495) $  & 3.917 \\

\bottomrule
\end{tabular}
\footnotetext{Total number of generations: 1000. $\bm{m}_1=[-.5,-.5]$. $\bm{m}_2=[.5,.5]$.}
\footnotetext{\textbf{Notation.} $\mathcal{P}_0$ denotes the initial population. For any real random vector $\bm{z}$ and any two real scalar $a<b$, $\bm{z} \sim$ uniform$[a, b]$ denotes that each entry of $\bm{z}$ is sampled uniformly from the interval $[a, b]$.}
\end{table}

\begin{table}
\caption{Experiment Results When $k=6$}
\label{robustk6}
\centering
\begin{tabular}{ p{1.9cm}  p{4.5cm}  p{1.7cm}  p{1.7cm} p{0.7cm} }
\midrule
 & Initial Value or Population & $\frac{\|\textbf{x}^*-\bm{m}_1\|^2}{6}$ & $\frac{\|\textbf{x}^*-\bm{m}_2\|^2}{6}$  & $f(\textbf{x}^*)$ \\
\toprule
REINF.-OPT & $\bm{x}_0=[0,0,...,0]\in \mathbb{R}^{6}$.        & $6.776\times 10^{-4}$  &  0.990 & 3.720 \\
                               &  $\bm{x}_0\sim$uniform$[0.2,0.3]$. & $8.533\times 10^{-4}$ &0.978   &3.501 \\
Genetic Algo.          & $\bm{x}\sim $uniform$[-1,1]$ for any $\bm{x}\in\mathcal{P}_0$.      & $3.963\times 10^{-6}$  &  1.000 & 8.503 \\
                               & $\bm{x}\sim$uniform$[0,1]$ for any $\bm{x}\in\mathcal{P}_0$.  &  0.998    &$3.747\times 10^{-5}$  &0.123 \\
Part. Swarm & $\bm{x}\sim $uniform$[-1,1]$ for any $\bm{x}\in\mathcal{P}_0$.                 & $1.638\times 10^{-3}$ & 0.926   & 2.905 \\
                               & $\bm{x}\sim$uniform$[0,1]$ for any $\bm{x}\in\mathcal{P}_0$. &  0.997    &$2.796\times 10^{-6}$  &2.815\\
Cross-Entropy & $\bm{x}\sim $uniform$[-1,1]$ for any $\bm{x}\in\mathcal{P}_0$.                 &$2.769\times 10^{-6}$  &  $1.000$  & 9.720 \\
                               & $\bm{x}\sim$uniform$[0,1]$ for any $\bm{x}\in\mathcal{P}_0$. &  0.997    &$2.796\times 10^{-6}$  &2.815\\
\bottomrule
\end{tabular}
\footnotetext{Total Number of Generations: 30000. $\bm{m}_1 = [-.5,-.5,\cdots,-.5]$. $\bm{m}_2 = - \bm{m}_1$.}
\footnotetext{\textbf{Notation.} $\mathcal{P}_0$ denotes the initial population. For any real random vector $\bm{z}$ and any two real scalar $a<b$, $\bm{z} \sim$ uniform$[a, b]$ denotes that each entry of $\bm{z}$ is sampled uniformly from the interval $[a, b]$.}
\end{table}

\subsection{Experiments on Inverse Problems}
\label{simulation}

In this section, we apply REINFORCE-OPT (i.e., Algorithm~\ref{alg:Framwork1}) to solve two typical non-linear inverse problems. In these two examples, we assume that a signal $\bm{y}^\delta$ in (\ref{inversep}) is observed. We set the objective function as
\begin{equation}
\label{reward-def}
\mathcal{L}(\bm{x}):= \frac{1}{\|f(\bm{x})-\bm{y}^{\delta} \|^2+\alpha \Omega(\bm{x})+0.001},
\end{equation}
where 0.001 is added to prevent the denominator from being a zero. We use this objective rather than (\ref{neg-obj}) because it works better in our experiments. As for the functional form of the policy, we set $\pi_{\bm{\theta}}(\cdot|\bm{x})$ as a multivariate Gaussian distribution with its mean and covariance equal to the output of a neural network $\mathcal{N}_{\bm{\theta}}$ with input $\bm{x}$ and weights $\bm{\theta}$. 

\subsubsection{Auto-convolution Equation}
\label{Ex1}

In the first group of numerical simulations, we consider the following 1D auto-convolution equation,
\begin{equation}
\label{AutoConv}
\int_{0}^{t} x(t-s) x(s) \,ds = y(t).
\end{equation}
which has many applications in spectroscopy (e.g., the modern methods of ultrashort laser pulse characterization) \cite{Baumeister1991, GerthHofmann2014}, the structure of solid surfaces \cite{Dai2013}, and nano-structures \cite{Fukuda2010}.
To the best of our knowledge, the results regarding the existence of $x(t)$ given $y(t)$ for the non-linear integral equation (\ref{AutoConv}) are quite limited. Hence, to avoid the possibility that there does not exist any $x(t)$ corresponding to a randomly chosen $y(t)$, for our experiment we pre-define
\begin{equation}
\label{AutoConv-x}
x(t) = 10 t (1-t)^2,
\end{equation}
and let $y(t)$ be the corresponding right-hand side in (\ref{AutoConv}).


We test the ability of REINFORCE-IP to numerically solve a discretized version of $(\ref{AutoConv})$ given noised $y(t)$. Let $f$ be the map from $\bm{x}=[x(0),x(\frac{1}{D-1}),x(\frac{2}{D-1}),...,x(1)]\in\mathbf{R}^D$, the discretization of $\{x(t)\}_{t\in[0,1]}$, to $\bm{y}=[y_0,y_1,...,y_{D-1}]\in \mathbb{R}^D$, a discrete-time approximation to $\{ y(t)\}_{t\in[0,1]}$, where $(x(t),y(t))$ satisfy Eq. (\ref{AutoConv}). Here, $D\in\mathbb{Z}^+$ denotes the number of discrete points, and we set it as $D=64$ in the experiment. The approximation $\bm{y}=f(\bm{x})$ is computed in the following way\footnote{From Eq. (\ref{AutoConv}), for any $j\in\{1,2,...,D-1\}$
$$y(t_j) = \sum_{i=1}^{j}\int_{t_{i-1}}^{t_i} x(t_j-s)x(s)ds\approx \frac{1}{D-1} \sum_{i=1}^{j}\frac{1}{2} \left(x(t_j-t_i)x(t_i)+x(t_j-t_{i-1})x(t_{i-1})\right)=y_j.$$
}: $y_0=0$, and for any $j\in\{1,2,...,D-1\}$,
$$ y_j = \frac{1}{D-1}\sum_{i=1}^{j} \frac{1}{2}\left(x\left(t_j-t_i\right)x\left(t_i\right)+x\left(t_j-t_{i-1}\right)x\left(t_{i-1}\right)\right),$$
where $t_j:=j/(D-1). $
The inverse problem we will use REINFORCE-IP to solve is
\begin{align}
\label{inversep-exp1}
f(\bm{x}) + \bm{\epsilon} = \bm{y}^{\delta},
\end{align}
where $\bm{y}^{\delta}=f(\mathbf{x}_e)+\bm{\epsilon}$, $\bm{\epsilon}$ is mean-zero Gaussian noise, and
\begin{equation}
\label{xe}
\mathbf{x}_e=[x(0),x(\frac{1}{N-1}),x(\frac{2}{N-1}),...,x(1)]\in\mathbb{R}^D
\end{equation}
with $x(t)$ defined in Eq. (\ref{AutoConv-x}) with $D=64$. This inverse problem has a unique solution $\mathbf{x}_e$ in $U^+:=\{\bm{x}\in \mathbb{R}^D: \bm{x} \text{ has non-negative entries} \}$, as implied by the following theorem on (\ref{AutoConv}), which further implies there are in total two solutions to (\ref{inversep-exp1}) in $\mathbb{R}^D$, $\mathbf{x}_e$ and $-\mathbf{x}_e$.

\vskip 0.1in 

\begin{theorem}
\label{ThmUnique}

  \emph{ (i) (Weak uniqueness \cite{GerthHofmann2014})} For $y\in L^2[0,2]$, the integral equation has at most two different solutions $x_1(t),x_2(t) \in L^1[0,1]$. Moreover,  $x_1(t)+x_2(t)=0, t \in [0,1]$ holds.

 \emph{(ii) (Uniqueness \cite{ WOS:A1994NF91100011})} Define
        \begin{align*}
             S^{+}:&=\{x \in L^1[0,1]:x(t) \ge 0  \}, \\
             S_{\epsilon}^{+}:&=\{ x \in L^1[0,1]: \epsilon=\sup\{s: x(t)=0, \text{~a.e. ~in~} [0,s]\} \}, \\
             Y_{\epsilon}^{+}:&=\{ y \in L^2[0,2]: \epsilon=\sup\{s: y(t)=0, \text{~a.e. ~in~} [0,s]\} \}.
        \end{align*}
        Then, for $y\in Y_0^{+}$, the equation (\ref{AutoConv}) has at most one solution in $S^{+}$. Moreover, the solution belongs to $S_0^{+}$.

\end{theorem}

Next, we state some common assumptions and settings for the REINFORCE-OPT algorithm applied in all our later experiments to solve (\ref{inversep-exp1}). Assume that instead of the exact signal $f(\mathbf{x}_e)$, a set $\{ \bm{y}_k^{\delta} \}_{k=1}^{100}$ of noised observations of the right-hand side in (\ref{inversep-exp1}) are available, and these observations are generated according to $\bm{y}^{\delta}_k:= f(\mathbf{x}_e) + \bm{\epsilon}_k$, where $\bm{\epsilon}_k\sim \mathcal{N}(\bm{0},\Sigma_k)$ is the Gaussian noise whose covariance $\Sigma_k$ is a diagonal matrix. The square roots of the diagonals in $\Sigma_k$ are set as $1\%\mathbf{y}_e$. 

We construct the policy $\pi_{\theta}(\cdot|\bm{x})$, a distribution over the action space determined by the given state $\bm{x}$, as a multivariate Gaussian distribution whose mean and standard deviation are transformed output\footnote{Let $[\mathbf{z}_1, \mathbf{z}_2]$ be an output of the neural network $\mathcal{N}_{\bm{\theta}}(\bm{x})$. Then, we let the moving average of $\mathbf{z}_1$ with a window length of 3 to serve as the mean of $\pi_{\bm{\theta}}(\cdot|\bm{x})$, and $\log(1+e^{\mathbf{z}_2})$ serve as the standard deviation. The moving average is used to ensure the smoothness of the predictions of $\mathbf{x}_e$. The function $\log(1+e^{\mathbf{z}_2})$ is called softplus and is a typical method to force outputs from a neural net to be positive.} from a feed-forward neural network $\mathcal{N}_{\bm{\theta}}$ with input $\bm{x}$.

As a common practice in the literature of solving (\ref{inversep-exp1}), assume that we are given the information that the first and last entry in the solution $\mathbf{x}_e$ are all zero. To use this information, we set the regularizer in (\ref{reward-def}) as
\begin{equation*}
\Omega(\bm{x}) = |x_0|+|x_{D-1}|.
\end{equation*}

\subsubsection{Simulation I}
\label{Simu1}

Suppose we are given the additional information that there is a solution to (\ref{inversep-exp1}) in $U^+$. To use it, we let the initial state in any trajectory generated by $\pi_{\bm{\theta}}$ be a fixed point in $U^+$, which is randomly chosen as $\bm{x}_0=[0.01,0.01,...,0.01]\in\mathbb{R}^{64}$. This setting increases the likelihood of the RL algorithm finding the solution $\mathbf{x}_e$ in $U^+$, since $\bm{x}_0$ is closer to the solution in $U^+$ (and hence easier to be found by the agent $\pi_{\bm{\theta}}$) than any solution not in $U^+$. The experiment result shows that this setting is effective, since all the produced solution estimates are approximations of $\mathbf{x}_e\in U^+$ (see Figure \ref{example1-1obs}).

We follow the REINFORCE-IP algorithm, i.e., Algorithm \ref{alg:Framwork}, to repeatedly update $\bm{\theta}$, until the sum of rewards\footnote{Specifically, after every 100 updates of $\bm{\theta}$, we use $\pi_{\bm{\theta}}$ to generate $L=1,000$ trajectories $\{ \bm{x}_{l,t} \}_{t= 0}^{T-1}$ for $l\in\{0,2,...,L-1\}$, take the mean of each step to get $\{ \bar{\bm{x}}_t \}_{t=0}^{T-1}$, and compute the sum $\sum_{t=0}^{T-1}R(\bar{\bm{x}}_t,\bm{0})$.
This sum is considered the performance of $\pi_{\bm{\theta}}$.} in a trajectory generated by the policy exceeds a threshold $\text{H}_0$ or the total number of $\theta$-updates exceeds 20,000, whichever comes first. Intuitively, this process keeps decreasing the denominator in Eq. (\ref{reward-def}). 

The value of the stop-training-threshold $\text{H}_0$,  $T$ (the number of total time steps taken in a trajectory), $\alpha$, $\beta$ (recall (\ref{rlgoal})), and hyper-parameters of the neural net $\mathcal{N}_{\theta}$ are all selected through the validation-set approach. The selected values are reported below Figure \ref{example1-poslog}, which plots the performance of $\pi_{\bm{\theta}}$ during the training process. 

\begin{figure}[htb]
	\centering
	\includegraphics[scale=0.4]{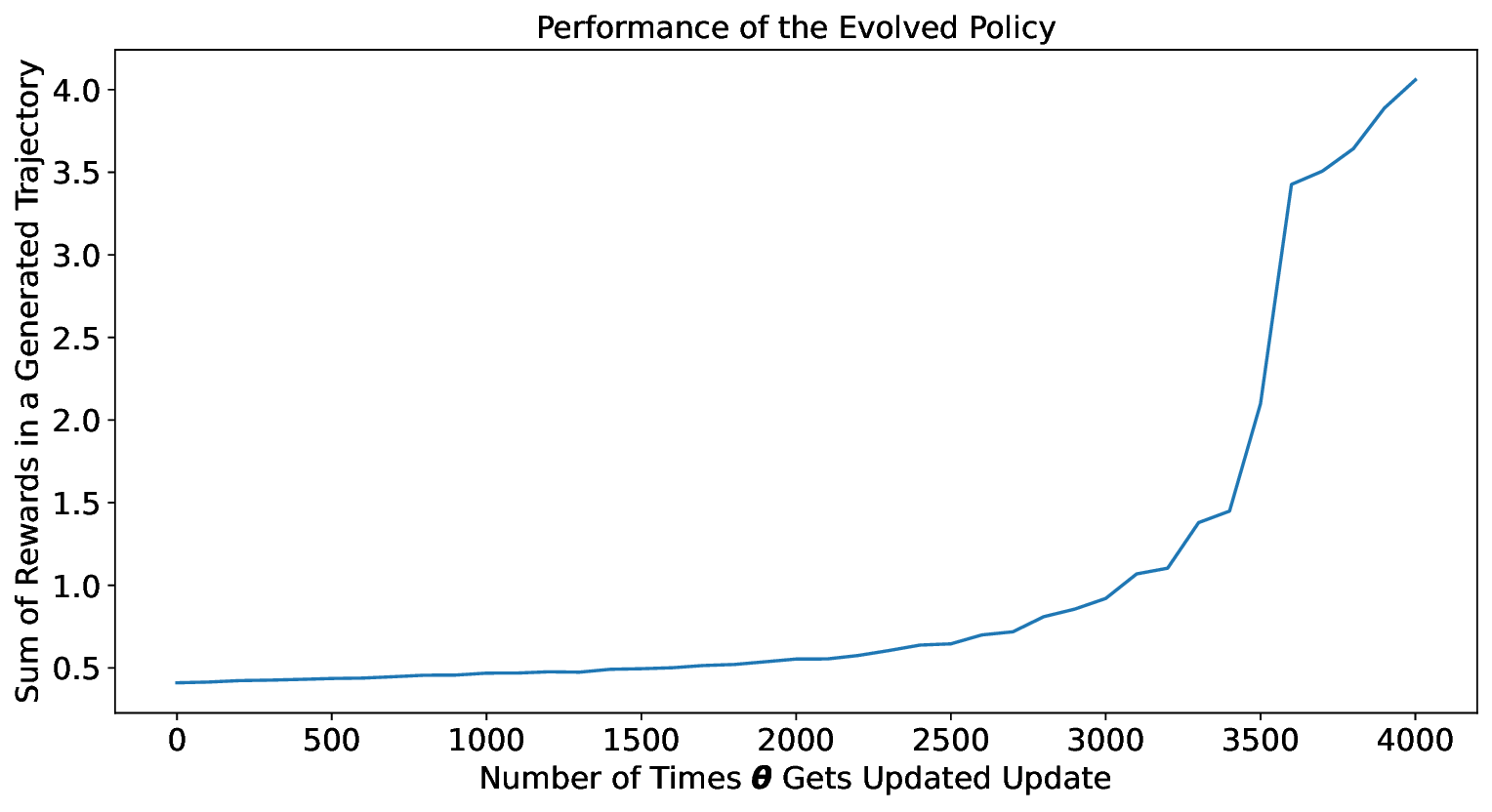}
       	\caption{Performance of the Evolving Policy for Simulation One. After every 100 updates of $\bm{\theta}$-update, we calculate the performance of $\pi_{\bm{\theta}}$. For model hyper-parameters, we set $T=10$, $\alpha=0.2$, $\beta=0.125$, and $\text{H}_0=4.0$ (the performance threshold for stopping training). The hidden layers of $\mathcal{N}_{\bm{\theta}}$ are $(64,64,64)$ with the activation function ReLU. The learning rate is set as $a_n:=\frac{0.001}{50,000+n}$.}
\label{example1-poslog}
\end{figure}

\begin{figure}[htb]
	\centering
	\includegraphics[scale=0.35]{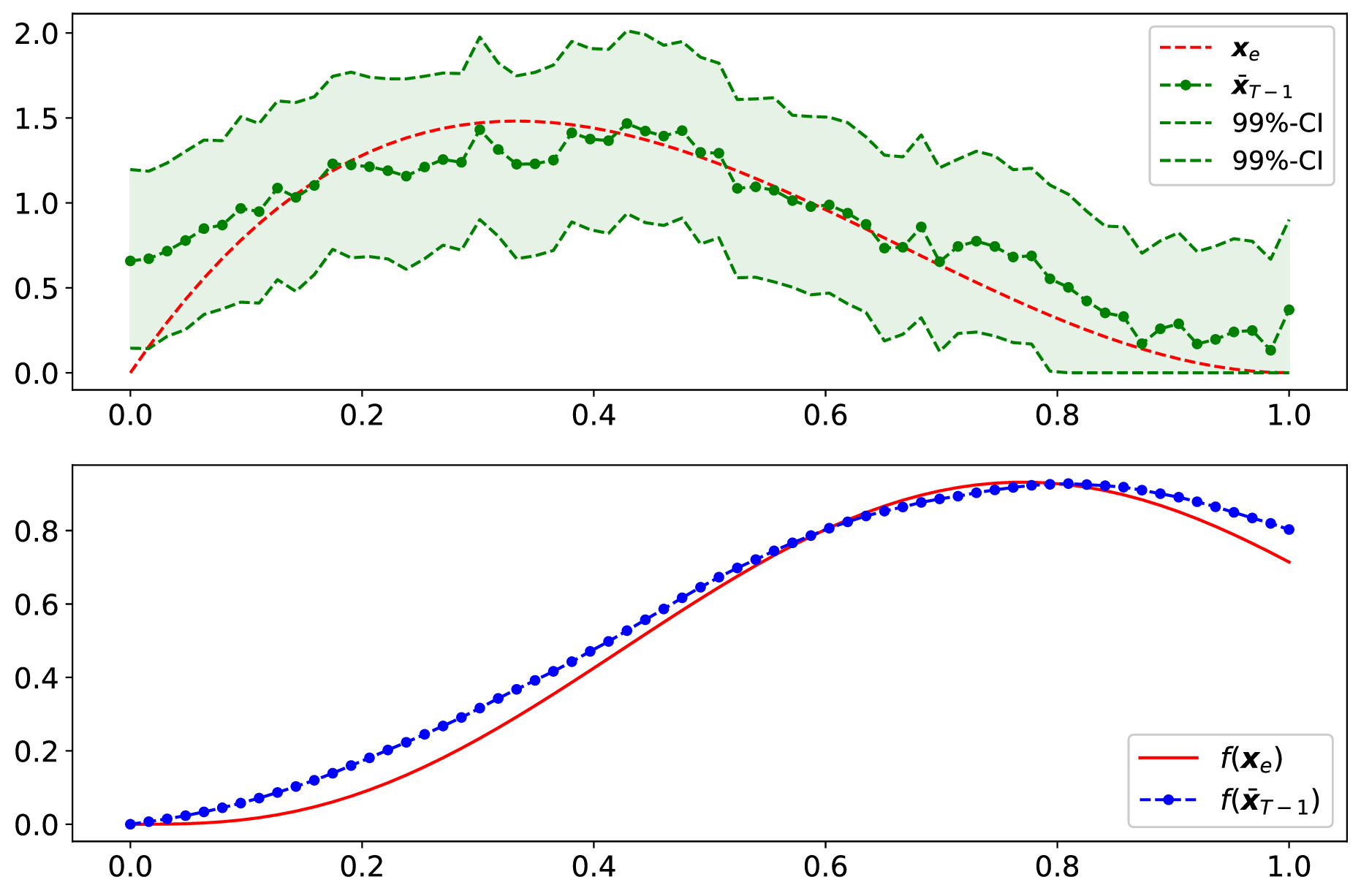}
       	\caption{The Estimations and True Values for Simulation I. $\bar{\mathbf{x}}_{T-1}$ is the mean of 10,000 solution estimates, and the region bounded by the two dashed curves of $99\%$-CI is the confidence interval of the solution estimate obtained by bootstrapping. Since we are looking for a solution in $U^+$, negative entries in the lower bound of the $99\%$-CI are replaced with zeros.}
\label{example1-1obs}
\end{figure}
We now investigate the ability of the trained policy $\pi_{\bm{\theta}}$ to find the solution to (\ref{inversep-exp1}). With $\pi_{\bm{\theta}}$, we generate $L=10,000$ trajectories of length $T=10$ and record the last state $\bm{x}_{T-1}$ in each trajectory (so that we have $\{\bm{x}_{l,T-1}\}_{l=1}^{N}$ and each $\bm{x}_{l,T-1}$ can be viewed as an estimated solution). The $R^2$ between $\mathbf{x}_e$ and the average $\bar{\mathbf{x}}_{T-1}$ of the $N$ estimate is 85.3\%, indicating that $\pi_{\bm{\theta}}$ is a successful policy for searching for a solution to the inverse problem. Their graphs, as well as the corresponding forward signals $f(\mathbf{x}_e)$ and $f(\hat{\mathbf{x}}_{T-1})$, are plotted in Figure \ref{example1-1obs}. The shaded region bounded by the two dashed green curves is the estimated $99\%$-confidence interval obtained by the method of bootstrapping\footnote{Specifically, we bootstrap from the estimates $\{\bm{x}_{l,T-1}\}_{l=1}^{1,000}$ for $B=10,000$ times, compute the mean of each bootstrap copy, and obtain the 0.5\%-percentile and 99.5\%-percentile of these 10,000 means as the lower and upper bound for the estimated 99\%-confidence interval.} (e.g., \cite{Puth2015}). That is, we estimate that for each entry $i\in\{1,2,...,D\}$ this shaded region contains the $i^{th}$ entry of the true solution $\mathbf{x}_e$ with a probability of $99\%$.

\subsubsection{Simulation II}
\label{Simu2}

In our second group of simulations, we aim to find only one solution (any one of two solutions) to the inverse problem (\ref{inversep-exp1}) without the a priori information of solutions. Hence, as a common practice in iterative algorithms, we fix the initial state $\bm{x}_0$ at the origin $\bm{0}\in\mathbb{R}^D$.

\begin{figure}[htb]
	\centering
	\includegraphics[scale=0.4]{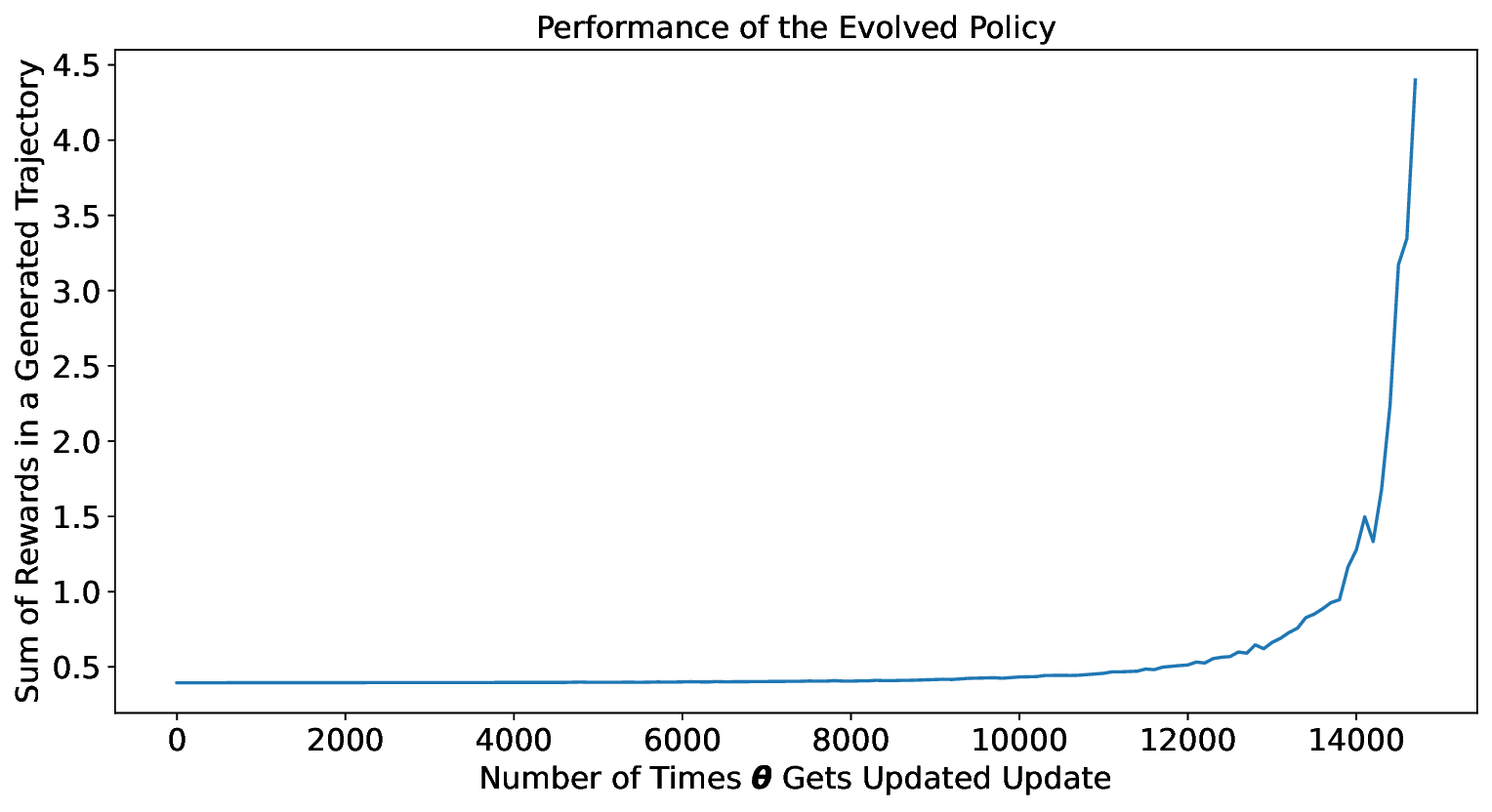}
       	\caption{Performance of the Evolving Policy for Simulation II. The procedures and hyper-parameter values are identical to those in Figure \ref{example1-poslog}, except that $\alpha=0.1$ and the initial state $\bm{x}_0$ for each generated trajectory is fixed at $\bm{0}\in\mathbb{R}^D$.}
\label{example1-neglog}
\end{figure}

The procedures for training and
testing $\pi_{\bm{\theta}}$ are identical to those in Simulation One. The history of $\pi_{\bm{\theta}}$'s performance in the $\bm{\theta}$-updating process is reported in Figure \ref{example1-neglog}, and the solution estimates by the trained model $\pi_{\bm{\theta}}$ are plotted in Figure \ref{example1-negobs}. The $R^2$ statistics between $-\mathbf{x}_e$ and the estimates average $\bar{\bm{x}}_{T-1}$ is $86.7\%$, indicating the success of RL algorithm in finding a solution to (\ref{inversep-exp1}).

Note that in Figure \ref{example1-negobs} the agent $\pi_{\bm{\theta}}$ finds an approximation to $-\mathbf{x}_e$. However, as we repeat this experiment more times (each experiment trains $\pi_{\bm{\theta}}$ independently), we end up with an approximation to $\mathbf{x}_e$, at a chance of approximately $50\%$.

\begin{figure}[H]
	\centering
	\includegraphics[scale=0.38]{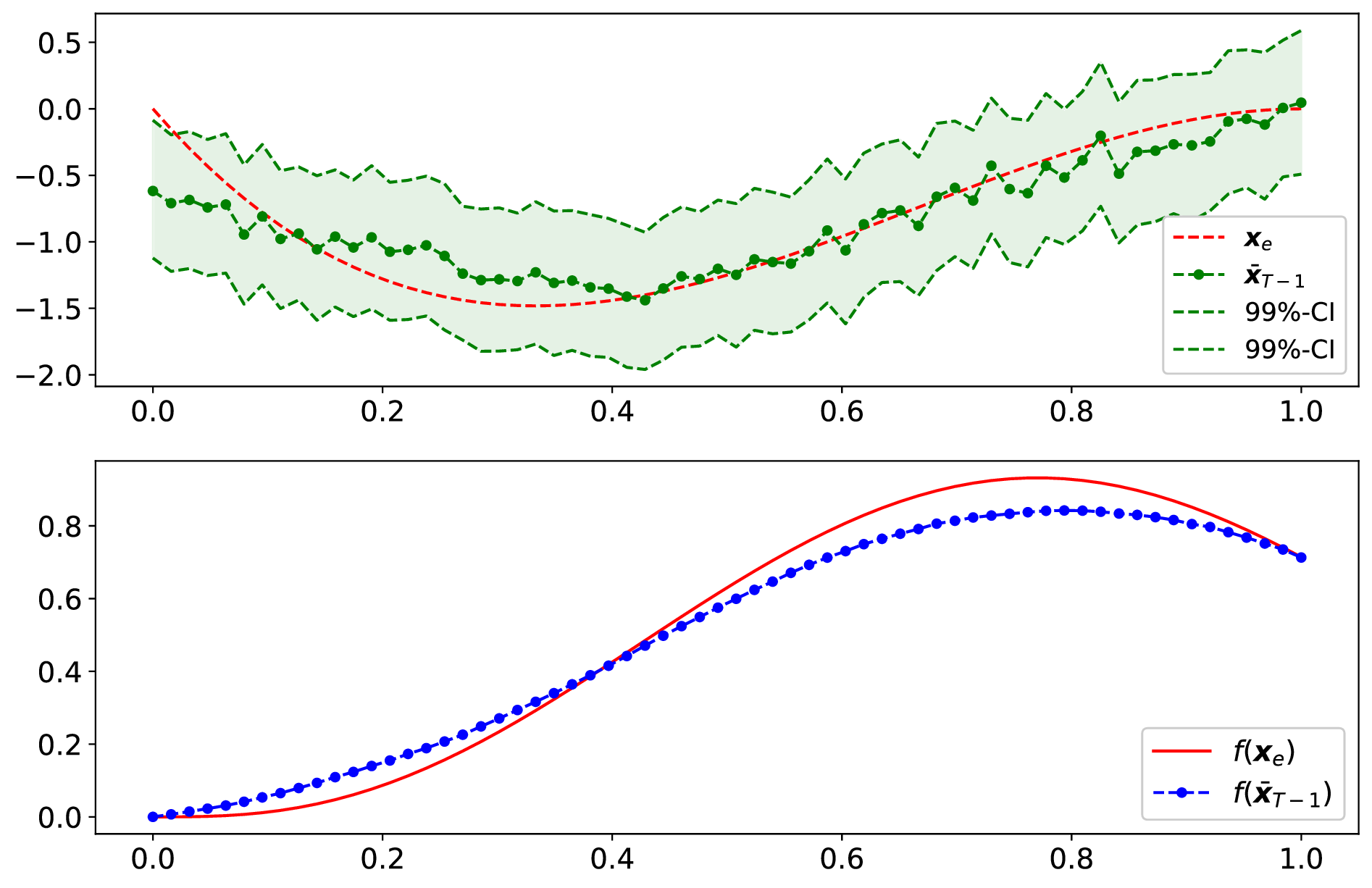}
       	\caption{The Estimations and True Values for Simulation II. The procedures for producing the plots and the symbols are the same as those in Figure \ref{example1-poslog}.}
\label{example1-negobs}
\end{figure}


\subsubsection{Simulation III}
\label{Simu3}

In the last group of simulations, we perform another experiment to illustrate that with a proper setting of the initial state's distribution, REINFORCE-IP is able to find good approximations of both of the two solutions without the need to retrain the model. Suppose we are given the a priori information that one solution is near $3\mathbf{x}_e/4$ and the other is near $-3\mathbf{x}_e/4$ (but we do not know that the true solutions are equal to $4/3$ of the given vectors). With this information, in the $\bm{\theta}$-updating process we set initial state $\bm{x}_0$ of each trajectory as either $3\mathbf{x}_e/4$ or $-3\mathbf{x}_e/4$, each with a probability of $50\%$.

Figure \ref{example1-2log} shows the performance of the policy $\pi_\theta$ that is updated repeatedly according to Eq. (\ref{REINFORCE}) and the performance is computed by the following procedure: (i) Use $\pi_{\bm{\theta}}$ to generated $L=1000$ trajectories of length $T=10$. (ii) The last-step state $\bm{x}_{T-1}$ in each trajectory is taken as an estimate of the solution, and hence we have 1000 estimates, $\{\bm{x}_{l,T-1}\}_{l=1}^{1000}$. (iii) The K-mean algorithm is then applied to divide $\{\bm{x}_{l,T-1}\}_{l=1}^{1000}$ to two groups. Compute the mean of each group, denoted by $\bar{\bm{x}}_1$ and $\bar{\bm{x}}_2$. (4) The performance of $\pi_{\bm{\theta}}$ is computed as $(R(\bar{\bm{x}}_1,\bm{0})+R(\bar{\bm{x}}_2,\bm{0}))/2$.

\begin{figure}[H]
	\centering
	\includegraphics[scale=0.4]{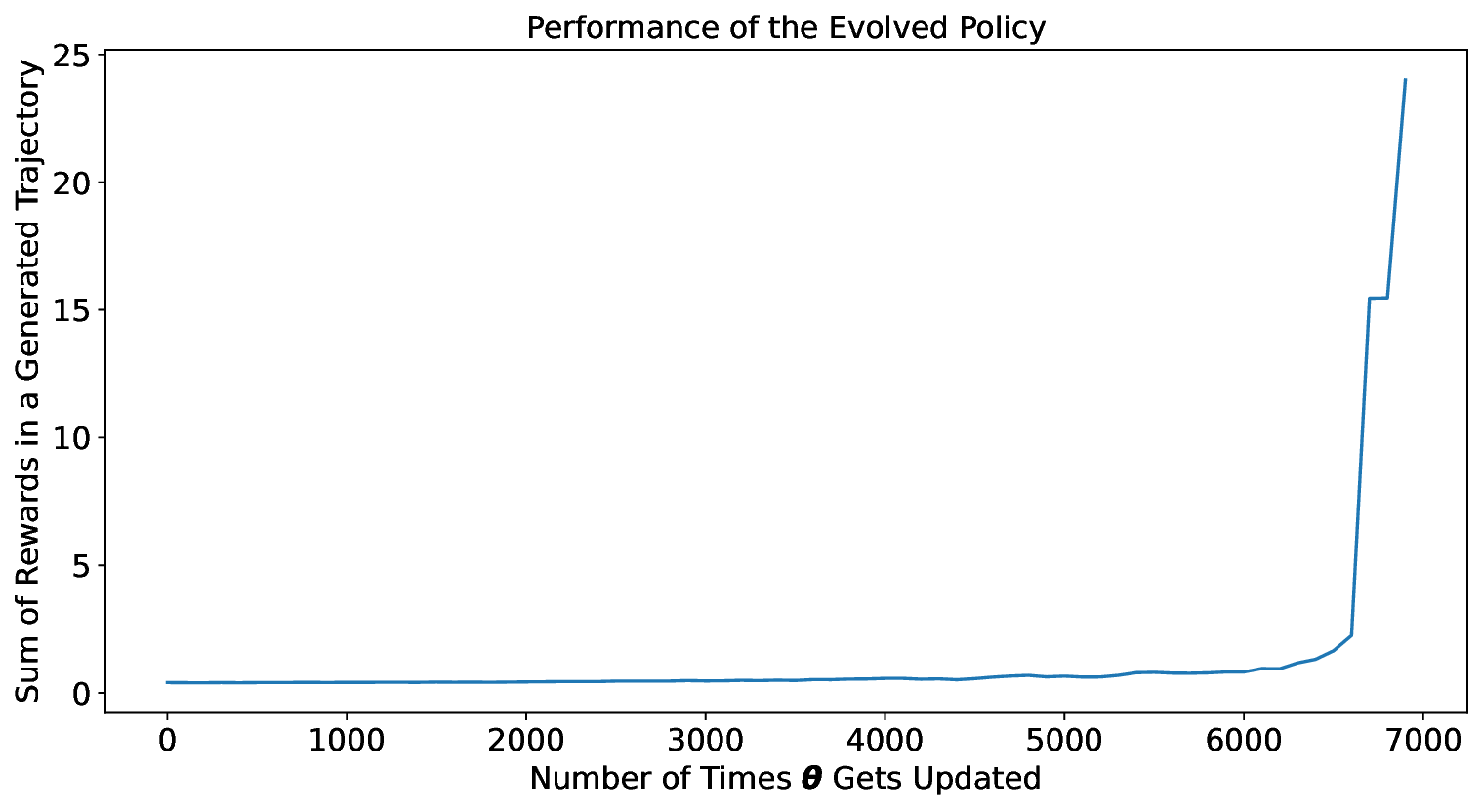}
       	\caption{Performance of the Evolved Policy for Simulation III. Compared with Simulations One and Two, the main difference lies in the distribution of the initial state $\bm{x}_0$ and the method of calculating the performances of $\pi_{\bm{\theta}}$. The threshold for stop training is set as $\text{H}_0=20$.}
\label{example1-2log}
\end{figure}

\begin{figure}[H]
	\centering
	\includegraphics[scale=0.36]{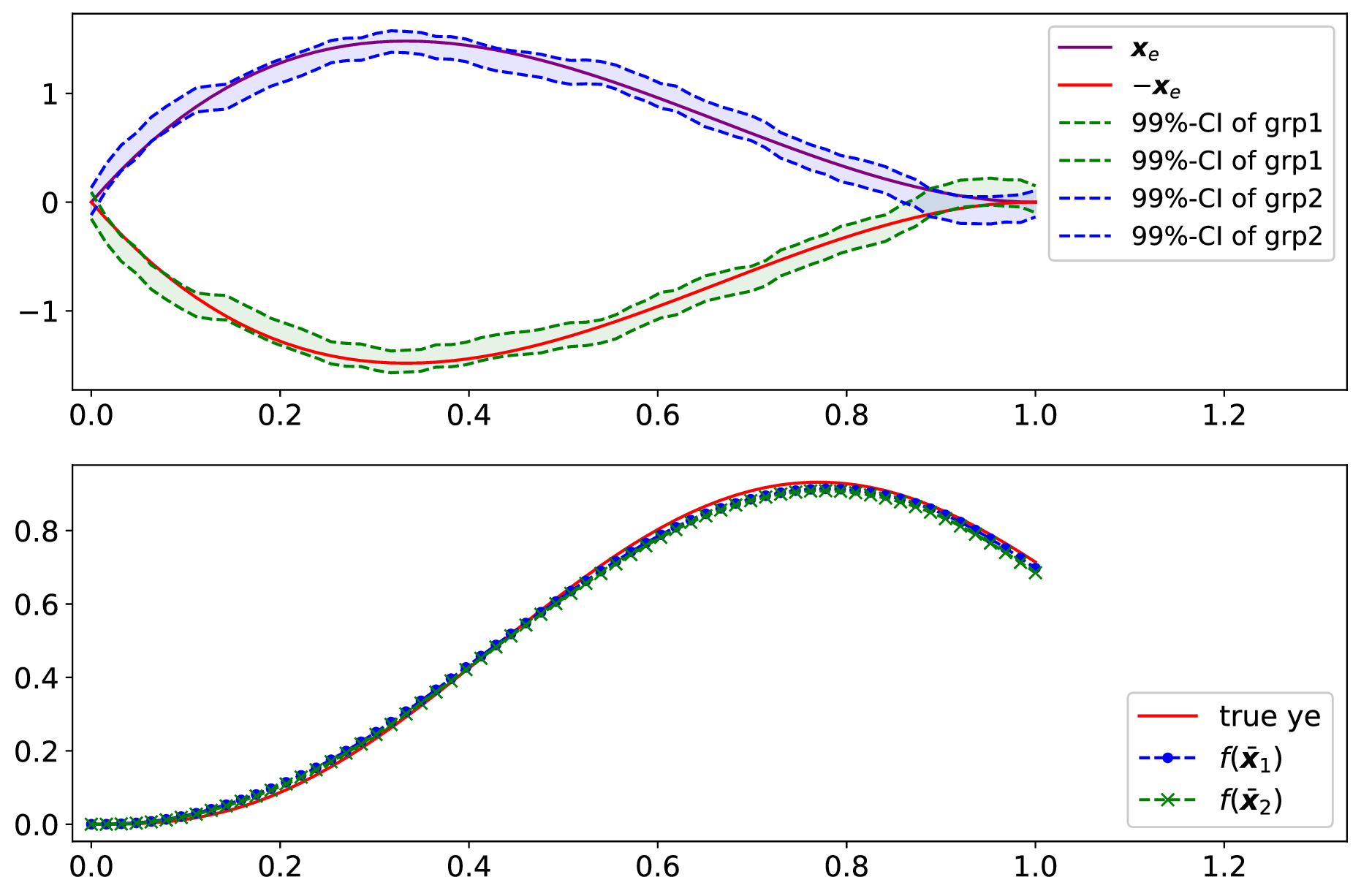}
       	\caption{Upper: Uncertainty quantification of approximate solutions by REINFORCE-OPT; Lower: comparison between simulated right-hand side and the exact right-hand side.}
\label{example1-2obs}
\end{figure}

To test the capacity of the trained policy to find solutions to (\ref{inversep-exp1}), as in the previous experiment, we generate $L=10,000$ trajectories of length $T=10$ with the trained agent $\pi_{\bm{\theta}}$. However, in this experiment the initial state $\mathbf{x}_0$ takes the value of $3\mathbf{x}_e/4$ or $-3\mathbf{x}_e/4$, each with a probability of $50\%$. The last-step state $\bm{x}_{T-1}$ in each trajectory is taken as an estimate of the solution, and hence we have 10,000 estimates, $\{\bm{x}_{l,T-1}\}_{l=1}^{10000}$. The K-mean algorithm is then applied to divide $\{\bm{x}_{l,T-1}\}_{l=1}^{10000}$ into two groups. Then, for each group we use the bootstrapping method (the same one used in the previous experiment) to estimate the 99\% confidence interval of each entry in our $\mathbf{x}_e$-estimate. The results are plotted in Figure \ref{example1-2obs}, where the region bounded by the two dashed blue curves is the estimated $99\%$-confidence interval based on estimates in group 2 (whose initial state is fixed at $0.75\mathbf{x}_e$). The region bounded by the two dashed green curves is for the sample means in group 1 (whose initial state is fixed at $-0.75\mathbf{x}_e$). Hence, in accordance with the definition of a confidence interval, each entry of the true solution lies in the shaded region with a probability of $99\%$.

The lower half of Figure \ref{example1-2obs} plots $\mathbf{y}_e$, $f(\bar{\mathbf{x}}_1)$, and $f(\bar{\mathbf{x}}_2)$, where $\bar{\mathbf{x}}_i$ denotes the mean of group-$i$ estimates, for $i=1$ or 2. The $R^2$-statistic between $\mathbf{x}_e$ and $\bar{\mathbf{x}}_1$ is $97.7\%$ and the $R^2$-statistic between $-\mathbf{x}_e$ and $\bar{\mathbf{x}}_2$ is $98.4\%$, both of which are high. This indicates the capability of the RL algorithm to produce all the solutions to the inverse problem (\ref{inversep-exp1}).


\subsection{Parameter Identification in Nonlinear PDE of Liquid Chromatography}
\label{Ex2}
\subsubsection{Experiment Background}

In the final group of numerical experiments, we consider an inverse problem in mathematical physics where the forward mapping has no closed form. Specifically, we examine the following mass balance equation for a two-component system in a fixed-bed chromatography column with the Danckwerts boundary condition, a common model in column chromatography \cite{zhang2016regularization,LinZhang2018}:
\begin{equation}
  \left\{\begin{array}{ll}
    \frac{\partial C_{i}}{\partial t}+F \frac{\partial q_{i}}{\partial t}+u \frac{\partial C_{i}}{\partial x}=D_{a} \frac{\partial^{2} C_{i}}{\partial x^{2}}, & x \in \mathcal{X} \equiv[0, L_1], t \in (0,T_1] \\
    C_{i}(x, 0)=g_{i}(x), & x \in \mathcal{X}, t=0 \\
    u C_{i}(0, t)-D_{a} \frac{\partial C_{i}(0, t)}{\partial x}=u h_{i}(t), & x=0, t \in (0,T_1] \\
    D_{a} \frac{\partial C_{i}(L_1, t)}{\partial x}=0, & x=L_1, t \in (0,T_1]
    \end{array}\right.,
    \label{eq:solver}
\end{equation}
where $x$ is distance, $t$ is time, and $i = 1,\,2$ refers to the two components. $C$ and $q$ are the concentration in the mobile and stationary phase, respectively, $u=0.1061$ is the mobile phase velocity, $F=0.6667$ is the stationary-to-mobile phase ratio, and $D_{a}=Lu/2N_x$ ($N_x=2000$ denotes the number of theoretical plates) is the diffusion parameter. $L_1=10$ is the length of the chromatographic column, and $T_1$ is an appropriate time point slightly larger than the dead time of chromatographic time $T_0 = L_1/u$. In this work, we set $T_1 = 1.5 T_0$. In addition, $g(x)= [0, 0]^T$ is the initial condition and $h(t)= [5; 5] \cdot H(10-t)$ ($H(\cdot)$ is the Heaviside step function) is the boundary condition, which describes the injection profile in the experiment.

For the numerical experiments, we focus on the case when the adsorption energy distribution is bimodal, namely when the bi-Langmuir adsorption isotherm is adopted:
\begin{eqnarray}\label{eq:q}
  q_{1}\left(C_{1}, C_{2}\right)=\frac{a_{I, 1} C_{1}}{1+b_{I, 1} C_{1}+b_{I, 2} C_{2}}+\frac{a_{I I, 1} C_{1}}{1+b_{I I, 1} C_{1}+b_{I I, 2} C_{2}}, \\
  q_{2}\left(C_{1}, C_{2}\right)=\frac{a_{I, 2} C_{2}}{1+b_{I, 1} C_{1}+b_{I, 2} C_{2}}+\frac{a_{I I, 2} C_{2}}{1+b_{I I, 1} C_{1}+b_{I I, 2} C_{2}}, \notag
\end{eqnarray}
where subscripts I and II refer to two adsorption sites with different adsorption energy.

For convenience of notation, the collection of adsorption isotherm parameters is denoted by
\begin{equation}\label{eq:parameters}
\mathbf{x} = (a_{I,1},a_{II,1},b_{I,1},b_{II,1}, a_{I,2},a_{II,2}, b_{I,2},b_{II,2})^T.
\end{equation}

Next, we consider the structure of the measurement data. In most laboratory and industrial settings, the total response $R(\mathbf{x},t)$ is observed at the column outlet $x=L_1$, with
\begin{equation}\label{eq:data}
R(\mathbf{x},t) =\sum_{i=1}^2 C_i(L_1,t),
\end{equation}
where $C(x,t)$ is the solution of problem (\ref{eq:solver}) with the bi-Langmuir adsorption-isotherm model (\ref{eq:q}), and $C_i(L_1,t)$ represents the concentration of the $i$-th component at the outlet $x=L_1$.  The exact collected data at time grid $\{t_j\}^{T_1}_{j=1} (T_1=800)$ at the outlet is denoted by
\begin{equation}\label{eq:data2}
\mathbf{y}_e = \{R(\mathbf{x},t_j)\}^{T_1}_{j=1}.
\end{equation}

The parameter-to-measurement map $f$: $\mathbb{R}^{8} \to \mathbb{R}^{T_1}_+$ can be expressed as
\begin{equation}\label{eq:operator}
f (\mathbf{x}) = \mathbf{y}_e,
\end{equation}
where the model function $f$ is defined through (\ref{eq:data})--(\ref{eq:data2}). To be more precise, for a given parameter $\mathbf{x}$, a bi-Langmuir adsorption-isotherm model can be constructed according to formula (\ref{eq:q}). The concentration in the mobile phase, denoted as $C$, can then be obtained by solving the PDE in (\ref{eq:solver}). Finally, the experimental data can be collected using the designed sensor in accordance with the physical laws described in (\ref{eq:data}) and (\ref{eq:data2}). The goal of this subsection is to estimate the adsorption isotherm parameters $\mathbf{x}$  from the time-series data $\mathbf{y}_e$ using the integrated mathematical model (\ref{eq:operator}) and RL.

\subsubsection{Inverse Problem to be Solved by REINFORCE-OPT}

Since there is no analytical expression for the forward mapping $f$ in (\ref{eq:operator}), we approximate it using a neural network $\hat{f}$. It is trained using
\footnote{This feed-forward neural network $\hat{f}$ is trained in the same way as $F_{\text{fwn}}$ in \cite[Section 3.1]{XuZhang2022}.}
$160,000$ samples of (the eight parameters $\bm{x}$, the injection profile $\bm{h}$ with two entries\footnote{$\bm{h}$ is the finite analog of injection function $h(t)$ in PDE (\ref{eq:solver}).}, and the corresponding total response $\bm{y}$), where the samples of $(\bm{x},\bm{h})$ are uniformly and randomly drawn from $(0,100)^8\times(0,30)^2$, and the samples of $\bm{y}$ are computed using ($\bm{x},\bm{h}$)-samples by the forward system described in the previous subsection. The $R^2$ score of $\hat{f}$ on 40,000 testing samples is above $95\%$, indicating that it is a good approximation of $f$.

The inverse problem we will use REINFORCE-OPT to solve is
\begin{align}
\label{inversep-exp2}
\hat{f}(\bm{x},\textbf{h}) + \bm{\epsilon} = \bm{y}^{\delta},
\end{align}
where $\bm{y}^{\delta}=\hat{f}(\mathbf{x}_e,\mathbf{h})+\bm{\epsilon}$, $\mathbf{x}_e$ is randomly generated and its value is shown in Table \ref{chromatography-table}, $\textbf{h}$ is the randomly generated projection profile that is considered to be a known parameter (the randomly drawn value is $[10,21]$), and $\bm{\epsilon}$ is a mean-zero noise. Later, we write $\hat{f}(\bm{x})$ instead of $\hat{f}(\bm{x},\mathbf{h})$ since $\mathbf{h}$ is known and fixed.

As before, we assume that a set $\{ \bm{y}^{\delta}_k \}_{k=1}^{100}$ of noised signals is available, where for each $k$, $\bm{y}_k^{\delta}$ is obtained by first randomly shifting\footnote{Let $\mathbf{y}_e:=\hat{f}(\mathbf{x}_e)=[y_0, y_1, ..., y_{799}]$. Then, with a probability of $50\%$, $\mathbf{y}_{k}^{\text{shift}}=[y_1, y_2, ..., y_{799}, 0]$ (shifted to the left for one unit). With a probability of $50\%$, $\mathbf{y}_{k}^{\text{shift}}=[0, y_0, y_1, ..., y_{798}]$ (shifted to the right for one unit).} the entries in the vector $\mathbf{y}_e=\hat{f}(\mathbf{x}_e,\mathbf{h})$ by one unit to obtain $\bm{y}_{k}^{\text{shift}}$, after which we add an artificial noise to it:
\[ \bm{y}_k^{\delta} = \bm{y}_{k}^{\text{shift}}+ \bm{\epsilon}_k ,\]
where $\bm{\epsilon}_k$ is a vector of independent mean-zero Gaussian random variables with standard deviations equal to $1\%\bm{y}_k^{\text{shift}}$. According to \cite{XuZhang2022}, the shift error is one of the major errors in real-world experiments, and hence $\bm{y}_{k}^{\text{shift}}$ needs to be taken into account. Moreover, in our simulations, all $\bm{y}_k^\delta$ are re-arranged at the same time grids by using the spline technique because, in general, the time grids of the numerical solution of PDE (\ref{eq:operator}) do not fit the time grids of the observation data.

\subsubsection{Eexperiment and Results}

Since there is no prior information to use, the regularizer $\Omega(\bm{x})$ is set as 0. Conforming to (\ref{reward-def}), we define the reward function as
\begin{equation}
\label{denominator-R}
\mathcal{L}(\bm{x}) := \frac{1}{(f(\bm{x})-\bm{y}^{\delta})^2/ 800 + 0.001},
\end{equation}
where 800 is the dimension of $\bm{y}^\delta$. As stated in Section \ref{Ex1}, we construct the policy $\pi_{\theta}(\cdot|\bm{x})$ as a multivariate Gaussian distribution whose mean and standard deviation are outputs from a neural network $\mathcal{N}_{\bm{\theta}}$. 
Values of the hyper-parameters are selected through the validation-set approach and are reported in Table \ref{fnn-structure-example2}.

\begin{table}[!htb]
\caption{Structure of $\mathcal{N}_{\bm{\theta}}$ and Hyper-parameter Values}
\label{fnn-structure-example2}
\centering%
\begin{tabular}{ p{0.3cm} p{0.2cm} p{3cm}  p{1.5cm}  p{0.7cm} p{2.8cm} p{0.3cm} }
\midrule
$T$ & $\alpha$ & Hidden Layers-Nodes  &Activation  &  $\beta$ & Learning Rate &$\text{H}_0$ \\
\toprule
$10$ & $0.1$ & $(128, 64, 16)$  	 &relu           &0.01    &$a_n:=\frac{0.001}{100,000+n}$ & 0.45 \\
\bottomrule
\end{tabular}
\end{table}

We repeatedly train the neural net until the sum of rewards in a trajectory generated by the policy stops improving\footnote{Specifically, we log the sum of rewards in a trajectory every 100 updates of $\theta$. If the logged value does not improve for 20 logs, we stop training.} or the total number of $\theta$-updates exceeds 7,500, whichever comes first.  Figure \ref{example2-log} shows the performance of the policy $\pi_\theta$ that is updated repeatedly in the training process.
\begin{figure}[H]
	\centering
	\includegraphics[scale=0.4]{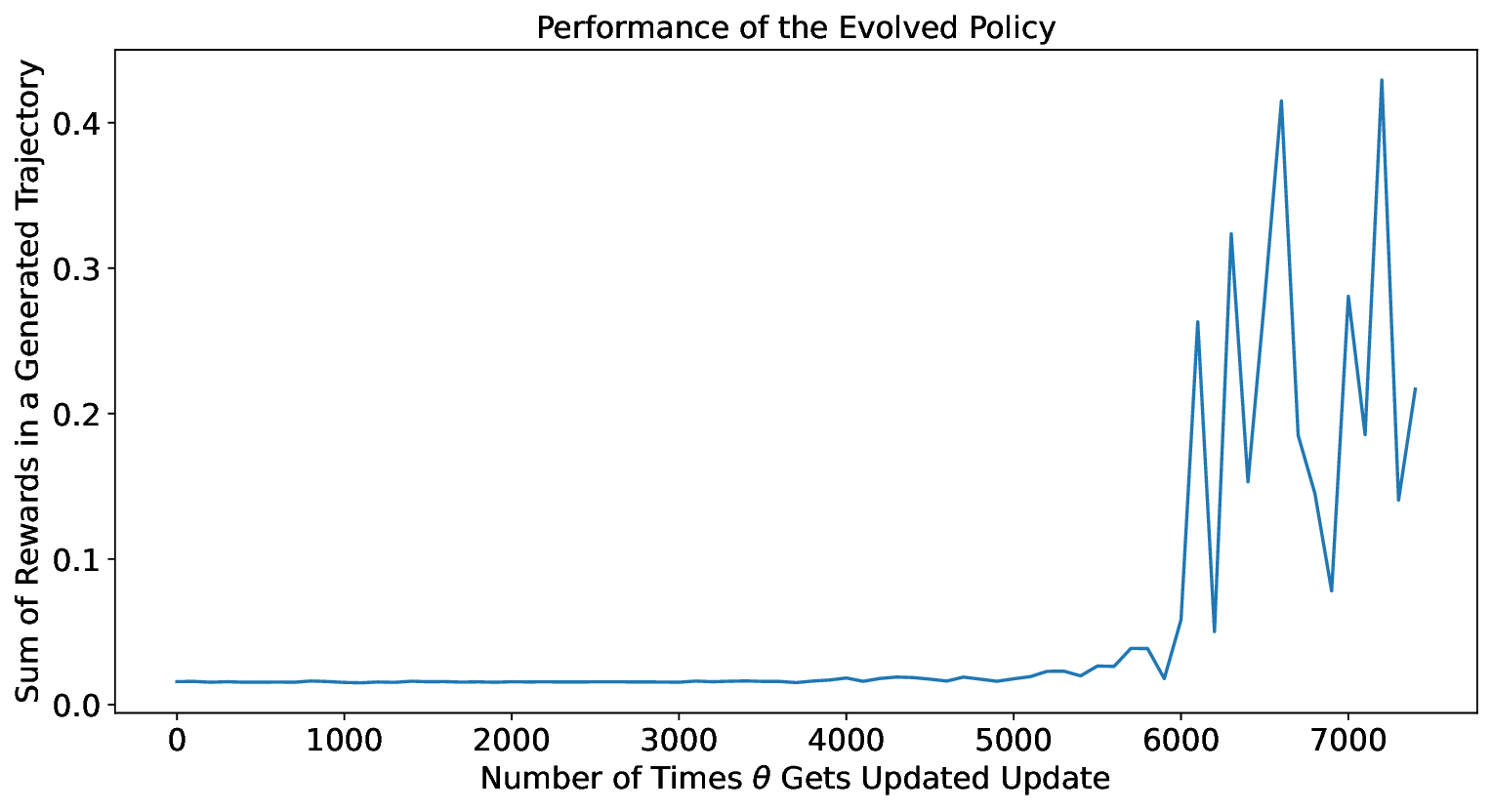}
       	\caption{Performance of the Evolved Policy $\pi_{\theta}$.}
\label{example2-log}
\end{figure}

As with previous experiments, with the trained $\pi_{\bm{\theta}}$, we generate $L=1000$ trajectories of length $T=10$, each starting from the fixed initial state $[50,50,...,50]$, and take the last state $\bm{x}_{l,T-1}$ in each trajectory as an estimate of $\mathbf{x}_e$. The mean $\bar{\mathbf{x}}_{T-1}$ of these $L$ estimates and the estimated $99\%$ confidence interval obtained by bootstrapping from these estimates are reported in Table \ref{chromatography-table}. From the table, it can be observed that the entries of $\bar{\mathbf{x}}_{T-1}$ are close to $\mathbf{x}_e$, except entries 2, 4, and 5. These three entries in $\mathbf{x}_e$ are beyond the estimated 99\% confidence interval. This occurs because this chromatography problem has multiple solutions, and in this experiment REINFORCE-OPT finds one different from $\mathbf{x}_e$.

The forward signals, $f(\mathbf{x}_e)$ and $f(\bar{\mathbf{x}}_{T-1})$ (plotted in Figure \ref{example2}), are close to each other with an $R^2$ between them as high as $91.8\%$, indicating that REINFORCE-IP is efficient for this experiment.

\begin{table}[!htb]
\caption{Estimations by REINFORCE-IP. $\bar{\mathbf{x}}_{T-1}$ represents the mean of $\mathbf{x}_e$-estimates. The row of CI-LB stores the lower bounds of the 99\% confidence interval, and the row of CI-UP stores the upper bounds of the 99\% confidence interval. Column $i$ corresponds to the $i^{th}$ entry, for $i\in \{1,2,...,8 \}$.}
\label{chromatography-table}
\centering%
\begin{tabular}{ p{1.8cm} p{0.8cm} p{0.8cm}  p{0.8cm} p{0.8cm} p{0.8cm} p{0.8cm} p{0.8cm} p{0.8cm} }
\toprule
& 1 & 2 & 3& 4& 5& 6& 7 & 8 \\
\toprule
$\mathbf{x}_e$ & 50.0 & 43.0 & 48.0 & 40.0 & 59.0 & 50.0 &51.0 &49.0 \\
\toprule
$\bar{\mathbf{x}}_{T-1}$ & 48.1 & 49.7 & 48.6 & 50.0 & 55.3 & 48.9 &54.2 &50.3 \\
\toprule
99\%CI-LB & 42.5 & 43.8 & 42.9 & 43.9 & 49.8 & 42.9 &49.0 &44.4 \\
\toprule
99\%CI-UB & 53.6 & 55.8 & 54.7 & 56.6 & 61.1 & 54.7 &59.6 &56.0 \\
\bottomrule
\end{tabular}
\end{table}

\begin{figure}[htb]
	\centering
	\includegraphics[scale=.6]{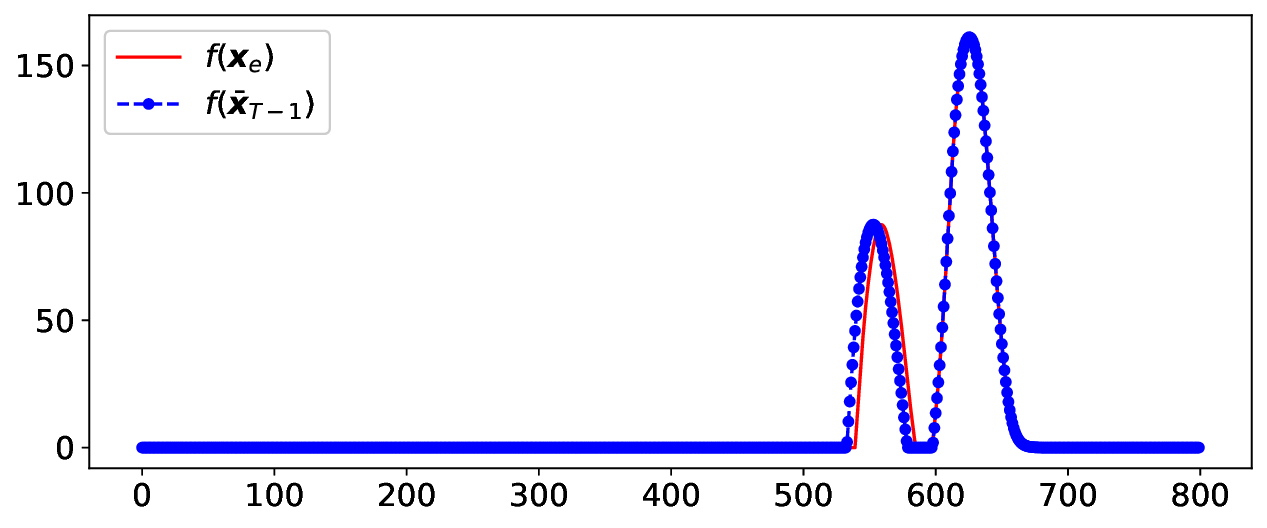}
       	\caption{The Estimations and True Values}
\label{example2}
\end{figure}

\section{Conclusion and Outlook}
\label{Conclusion}

In this work, we develop a new method called REINFORCE-OPT to solve general optimization problems, and we formally introduce this RL-based approach to the field of inverse problems. Our theoretical findings provide a deeper understanding of RL-based optimization methods and their connection to classical regularization techniques for solving general inverse problems. Numerical experiments demonstrate that REINFORCE-OPT is capable of escaping local optima in the solution space (see Figures \ref{escape-local}, \ref{loss-traj}, and \ref{escape-local2D}), is robust to the choice of initial solution candidates (see Tables \ref{compare-results} and \ref{robustk6}), and can quantify uncertainty in error estimates while identifying multiple solutions in ill-posed inverse problems (see Figure \ref{example1-2obs}).

The developed REINFORCE-OPT can also be applied readily to discrete optimization problems by selecting appropriate state and action spaces. Given its success with continuous optimization, we believe that it is well-suited to handle discrete optimization problems involving large state spaces, a challenge for traditional methods.

It should be noted that REINFORCE is one of the early RL algorithms to use function approximators such as $\pi_{\bm{\theta}}$, making it powerful when dealing with infinite state spaces. However, its reliance on a Monte Carlo mechanism leads to high variance in the estimate $\hat{\nabla} J_{\bm{\theta}}$, resulting in a slow learning process, even with the use of a baseline \cite[Section 13.4 and 13.5]{SuttonBarto2018}. This limitation is addressed by the more advanced actor-critic methods \cite{BorkarKonda1997,KondaBorkar1999,KondaTsitsiklis1999,KondaTsitsiklis2003,XuYang2021,HongWai2023,SuttonBarto2018}, where the ``actor'' refers to $\pi_{\bm{\theta}}$ and the ``critic'' refers to a function approximator for $J_T(\bm{\theta})$. To our knowledge, the performance of these actor-critic methods for general inverse problems of the type in (\ref{inversep}) remains unexplored and is a topic for our future investigation. 

\section*{Acknowledgements}
This work was supported by the National Key Research and Development Program (2022YF- C3310300),  the National Natural Science Foundation of China (12171036, 12471295 and 12426306),  the Beijing Foundation for Natural Sciences (Z210001).

\section*{Appendices}

\subsection*{Appendix A Auxiliary Lemmas}	
\begin{lemma}
\label{Gronwal}
\emph{(Gronwall inequality, \cite[Lemma B.1]{Borkar2022})} For continuous $u(\cdot), v(\cdot) \geq 0$ and scalars $C, K, T \geq 0$, the condition
\begin{equation*}
u(t) \leq C+K \int_0^t u(s) v(s) d s \quad \text{for all } t \in[0, T]
\end{equation*}
implies
$$
u(t) \leq C e^{K \int_0^T v(s) d s}, ~ t \in[0, T].
$$
\end{lemma}

From the proof given in \cite[Lemma B.1]{Borkar2022}), Lemma \ref{Gronwal} still holds when $C$ is dependent on $T$.

\begin{lemma}
\label{DiscreteGronwal}
\emph{(Discrete Gronwall Inequality, \cite[Lemma B.2]{Borkar2022})}  Let $\left\{x_n, n \geq 0\right\}$ (resp. $\left\{a_n, n \geq 0\right\}$) be a non-negative (resp. positive) sequence and $C, L \geq 0$ be scalars such that, for all $n$,	
		\begin{equation*}
			x_{n+1} \leq C+L\left(\sum_{m=0}^n a_m x_m\right).
		\end{equation*}
Then, $ x_{n+1} \leq C e^{L s_n}$, where $s_n=\sum_{m=0}^n a_m$ for each non-negative integer $n$.
\end{lemma}
		
		
\begin{lemma}
\label{C3}
\emph{\cite[Theorem C.3]{Borkar2022}}
Suppose $\left\{X_n\right\}_{n=0}^{\infty}$ is a martingale with respect to $\left\{\mathcal{F}_n\right\}$, defined on the probability space $(\Omega, \mathcal{F}, P)$. Also assume that $\mathbb{E}\left[X_n^2\right]<\infty$. Then,
		\begin{itemize}
			\item $\sum_{n=0}^{\infty} \mathbb{E}\left[\left(X_{n+1}-X_n\right)^2 \mid \mathcal{F}_n\right]<+\infty$ implies that, almost surely, $\left\{X_n\right\}_{n=0}^{\infty}$ converges.
			\item  $\sum_{n=0}^{\infty} \mathbb{E}\left[\left(X_{n+1}-X_n\right)^2 \mid \mathcal{F}_n\right]=+\infty$ implies that
			$$
			X_n=o\left(\sum_{m=1}^{n-1} \mathbb{E}\left[\left(X_{m+1}-X_m\right)^2 \mid \mathcal{F}_m\right]\right) \text {, a.s. }
			$$
		\end{itemize}
\end{lemma}

\subsection*{Appendix B Proof of the Implication in Step 1 of the Proof of Lemma \ref{Lemma2.1}}

The following lemma will be used to prove the implication stated in Step 1 of Lemma \ref{Lemma2.1}'s proof.

\textbf{Lemma A.1} \emph{ Consider the ODE $\dot{x}(t) = f(x(t))$, where $f:\mathbb{R}^d\rightarrow \mathbb{R}^d$ is Lipschitz with the Lipschitz constant $L_f>0$. Assume that $x_1(\cdot)$ and $x_2(\cdot)$ are two solutions to the ODE. Then, for any $\Delta>0$ and any $t\in [0,\Delta]$, we have
$$ \lVert x_1 (t)-x_2 (t) \rVert \leq \lVert x_1 (0)-x_2 (0)\rVert e^{L_f\Delta}. $$}

\begin{proof}
Note that
\begin{align*}
\lVert x_1 (t)-x_2 (t) \rVert &= \lVert \int_0^t[f(x_1(s)) - f(x_2(s)) ]ds + (x_1(0)-x_2(0)) \rVert\\
&\leq   \int_0^t \lVert f(x_1(s)) - f(x_2(s)) \rVert ds +  \lVert x_1(0)-x_2(0) \rVert\\
&\leq   L_f\int_0^t \lVert x_1(s) - x_2(s) \rVert ds +  \lVert x_1(0)-x_2(0) \rVert.
\end{align*}
Hence, using the Gronwall inequality (see Lemma \ref{Gronwal} in Appendix A), we have for any $t\in [0,\Delta]$ that
$ \lVert x_1 (t)-x_2 (t) \rVert \leq \lVert x_1(0) - x_2(0)  \rVert e^{L_f\Delta}. $  
\end{proof}

\vskip 0.1in

\textbf{Proof of the implication in Step 1}. Suppose that Eq. (\ref{Lemma2.1a}) holds. Our goal is to prove Eq. (\ref{Lemma2.1-eq}), namely that, for any fixed $\Delta>0$,
\begin{equation*}
\lim_{s\rightarrow+\infty} \left( \sup_{t\in[s,s+\Delta]}\lVert \bar{\bm{\theta}}(t) - \bm{\theta}^{s}(t) \rVert  \right) = 0, \quad a.s.
\end{equation*}
We follow the steps (L1) to (L3) below to accomplish this proof.

\textbf{Step (L1)} We fisrt show that
$$\lim_{n\rightarrow+\infty} \left( \sup_{t\in[s_n,s_n+\Delta]}\lVert \bar{\bm{\theta}}(t) - \bm{\theta}^{s_n}(t) \rVert  \right) = 0.$$
 Indeed, from Eq. (\ref{Lemma2.1a}), we have
\begin{equation*}
\lim_{n\rightarrow+\infty} \left( \sup_{t\in[s_n,[s_n+\Delta+1]^-]}\lVert \bar{\bm{\theta}}(t) - \bm{\theta}^{s_n}(t) \rVert  \right) = 0, \quad \text{for any fixed } \Delta>0.
\end{equation*} 
To finish Step (L1), we need to show that, when $n$ is large enough,
\begin{equation}
\label{L4-2}
s_n+\Delta\leq[s_n+\Delta+1]^-,
\end{equation}
which further implies
\begin{equation*}
\label{L4-3}
\lim_{n\rightarrow+\infty} \left( \sup_{t\in[s_n,s_n+\Delta]}\lVert \bar{\bm{\theta}}(t) - \bm{\theta}^{s_n}(t) \rVert  \right) 
\leq \lim_{n\rightarrow+\infty} \left( \sup_{t\in[s_n,[s_n+\Delta+1]^-]}\lVert \bar{\bm{\theta}}(t) - \bm{\theta}^{s_n}(t) \rVert  \right) = 0.
\end{equation*}
From Point 3 in Assumption \ref{Assump0}, $\lim_{n\rightarrow +\infty} a_n =0$. Therefore, there exists $N>0$ such that, whenever $n>N$, $a_n<1$. Now, suppose $n>N$ and let $k$ denote the positive integer such that $s_k:= [s_n+\Delta+1]^-$. Then,
$$
s_n+\Delta+1-[s_n+\Delta+1]^- \leq [s_n+\Delta+1]^+-[s_n+\Delta+1]^- = s_{k+1}-s_k = a_k.
$$
Subtracting the left and right end of the above expression by 1 gives $s_n+\Delta-[s_n+\Delta+1]^- \leq a_k-1 \leq 0$, where the last inequality is because $k\geq n>N$ (according to the definition of $s_k$), and hence $-1+a_k<0$. This proves (\ref{L4-2}) and finishes the proof for the equality in Step (L1).

\textbf{Step (L2).} We prove that, when $s$ is large enough,
\begin{equation*}
\begin{split}
\sup_{t\in[s,s+\Delta]}\lVert \bar{\bm{\theta}}(t) - \bm{\theta}^{s}(t) \rVert \leq \sup_{t\in[[s]^-,[s]^-+\Delta+1]}\lVert \bar{\bm{\theta}}(t) - \bm{\theta}^{[s]^-}(t) \rVert(1+e^{L_g(\Delta+1)}).
\end{split}
\end{equation*}
Recall that $\lim_{n\rightarrow +\infty} a_n =0$. Let $N>0$ denote the threshold such that, whenever $n\geq N$, $a_n<1$. For any $s\geq s_N$, let $k$ denote the non-negative integer such that $[s]^- = s_k$. Then, $k\geq N$ and
\begin{equation*}
s-[s]^-\leq [s]^+-[s]^-=s_{k+1}-s_{k}=a_k<1,
\end{equation*}
which implies
\begin{equation}
\label{s-size}
s \leq [s]^-+ 1
\Rightarrow s+\Delta \leq [s]^-+ \Delta+1.
\end{equation}
This result is used to derive the second inequality of the following:
\begin{equation}
\begin{split}
\sup_{t\in[s,s+\Delta]}\lVert \bar{\bm{\theta}}(t) - \bm{\theta}^{s}(t) \|
&\leq \sup_{t\in[s,s+\Delta]}\lVert \bar{\bm{\theta}}(t) -\bm{\theta}^{[s]^-}(t) \rVert + \sup_{t\in[s,s+\Delta]}\lVert \bm{\theta}^{[s]^-}(t) -\bm{\theta}^{s}(t) \rVert \\
\label{L1-2}
&\leq \sup_{t\in[[s]^-,[s]^-+\Delta+1]}\lVert \bar{\bm{\theta}}(t) -\bm{\theta}^{[s]^-}(t) \rVert + \sup_{t\in[s,s+\Delta]}\lVert \bm{\theta}^{[s]^-}(t) -\bm{\theta}^{s}(t) \rVert.
\end{split}
\end{equation}
Now, we derive an upper bound for the second sup-term, $\sup_{t\in[s,s+\Delta]}\lVert \bm{\theta}^{[s]^-}(t) -\bm{\theta}^{s}(t) \rVert$. Let $y_{[s]^-}(t):=\bm{\theta}^{[s]^-}(t+s)$ and $y_s (t) := \bm{\theta}^s(t+s)$. Then, $y_{[s]^-}(t)$ and $y_s (t)$ are solutions to ODE (\ref{ODE-f}) since
\begin{align*}
&\dot{y}_{[s]^-}(t) = \dot{\bm{\theta}}^{[s]^-}(t+s) = g(\bm{\theta}^{[s]^-}(t+s)) = g(y_{[s]^-}(t));\\
&\dot{y}_{s}(t) = \dot{\bm{\theta}}^{s}(t+s) = g(\bm{\theta}^{s}(t+s)) = g(y_{s}(t)).
\end{align*}
The definitions of $y_{[s]^-}(t)$ and $y_s (t)$ imply the second line in the following:
\begin{align*}
\sup_{t\in [s,s+\Delta]} \lVert \bm{\theta}^{[s]^-}(t) -\bm{\theta}^{s}(t) \rVert  &=   \sup_{t\in [0,\Delta]} \lVert \bm{\theta}^{[s]^-}(t+s) -\bm{\theta}^{s}(t+s) \rVert\\
&=\sup_{t\in [0,\Delta]} \lVert y_{[s]^-}(t) -y_{s}(t) \rVert\\
&\leq  \lVert y_{[s]^-}(0) -y_{s}(0) \rVert e^{L_g \Delta}, \; \text{by Lemma A.1}\\
&=  \lVert \bm{\theta}^{[s]^-}(s) -\bm{\theta}^{s}(s) \rVert e^{L_g \Delta}\\
&=  \lVert \bm{\theta}^{[s]^-}(s) -\bar{\bm{\theta}}(s) \rVert e^{L_g \Delta},\;\text{By $\bm{\theta}^s$-definition}\\
&\leq \sup_{t\in[[s]^-,[s]^-+\Delta+1]}\lVert \bar{\bm{\theta}}(t) - \bm{\theta}^{[s]^-}(t) \rVert e^{L_g\Delta},
\end{align*}
where the last inequality follows from the fact when $s$ is sufficiently large (i.e., $s\geq s_N$), $s\in [[s]^-,[s]^-+\Delta+1]$ according to (\ref{s-size}). Plugging this result into (\ref{L1-2}) proves Step (L2).

\textbf{Step (L3).} Take the limit of both sides of the inequality in Step (L2):
\begin{equation*}
\lim_{s\rightarrow +\infty} \sup_{t\in[s,s+\Delta]}\lVert \bm{\theta}^{[s]^-}(t) -\bm{\theta}^{s}(t) \rVert \leq \lim_{s\rightarrow +\infty}
\sup_{t\in[[s]^-,[s]^-+\Delta+1]}\lVert \bar{\bm{\theta}}(t) - \bm{\theta}^{[s]^-}(t) \rVert(1+e^{L_g(\Delta+1)})=0,
\end{equation*}
where the equality is in accordance with Step (L1).
Thus the whole proof of the implication in Step 1 of Lemma \ref{Lemma2.1}'s proof.

\subsection*{Appendix C Existence of A Stationary Distribution of Markov Chain}
Let $\{X_k\}_{k=0}^\infty$ be a Markov chain defined on a measurable space $(\textbf{X},\mathcal{X})$, and let $P$ denote its Markov kernel. The invariant distribution of $\{X_k\}$ (also called a stationary or a normalized invariant measure), if it exists, refers to a probability measure $\mu$ on $(\textbf{X},\mathcal{X})$ satisfying
$$\mu P=\mu\text{ and } \mu(\textbf{X})=1.$$
Sufficient but possibly unnecessary conditions for the existence of a unique invariant distribution can be found in \cite[Section 2.1.2, H2.1.8]{Douc2018} and are copied below, which requires $\textbf{X}$ to be a complete and separable metric space: 

 (i) \textit{There exists a measurable function} $K:\textbf{Z}\rightarrow \mathbb{R}_+$ \textit{such that, for all} $(x,y,z)\in \textbf{X}\times\textbf{X}\times \textbf{Z}$,
 $d(f(x,z),f(y,z))\leq K(z)d(x,y), $  
$ \mathbb{E}[\max\{\log K(Z_1),0\}]<\infty, $ and $  \mathbb{E}[\log K(Z_1)]<0, $
 where  $d$ \textit{denotes the metric on} \textbf{X},
  $(\textbf{Z},\mathcal{Z})$ \textit{is a measurable space},
  $\{Z_k\}_{k=0}^\infty$ \textit{is a sequence of i.i.d. random variables taking values in} $\textbf{Z}$, 
and  $f:\textbf{X}\times\textbf{Z}\rightarrow \textbf{X}$ \textit{is a $\mathcal{X}\otimes\mathcal{Z}$-$\mathcal{X}$ measurable map such that $X_{k}=f(X_{k-1},Z_k)$. (Note that its existence is guaranteed by the fact that $\mathcal{X}$ is finitely generated since} \textbf{X} \textit{is separable, and by Theorem 1.3.6 in \cite{Douc2018}}).

 (ii) \textit{ There exists }$x_0\in \textbf{X}$ \textit{such that}
$ \mathbb{E}[\log^+d(x_0,f(x_0,Z_1))]<\infty.$

\bibliography{Literature}
\bibliographystyle{abbrv}
\end{document}